\documentclass[a4paper,reqno]{amsart}
\usepackage{blindtext}

\title[A diagrammatic approach to Springer theory]{Two-block Springer fibers of types C and D: \\a diagrammatic approach to Springer theory}
\author{Catharina Stroppel}
\address{C.S.: Mathematical Institute\\
	University of Bonn\\
	Endenicher Allee 60\\
	53115 Bonn\\
	Germany}
\email{stroppel@math.uni-bonn.de}

\author{Arik Wilbert}
\address{A.W.: School of Mathematics and Statistics\\
	University of Melbourne\\
	Parkville VIC 3010\\
	Peter Hall\\
	Australia}
\email{arik.wilbert@unimelb.edu.au}


\usepackage{geometry}

\usepackage[english]{babel}
\usepackage{amssymb,MnSymbol}
\usepackage{enumitem}
\usepackage{tikz}
\usetikzlibrary{decorations.pathmorphing,shapes}
\usepackage[all, arrow, matrix, curve, knot, poly]{xy}
\usepackage{verbatim}
\usepackage{graphicx}
\usepackage{youngtab}
\numberwithin{equation}{section}

\theoremstyle{definition}
\newtheorem{defi}{Definition}
\theoremstyle{remark}
\newtheorem{ex}[defi]{Example}
\newtheorem{rem}[defi]{Remark}

\theoremstyle{plain}
\newtheorem{lem}[defi]{Lemma}

\newtheorem{cor}[defi]{Corollary}
\newtheorem{prop}[defi]{Proposition}
\newtheorem{thm}[defi]{Theorem}

\newtheorem{introthm}{Theorem}

\newcommand{\SOfiber}{\mathcal S^{2m-k,k}_\mathrm{KL}}
\newcommand{\SOfiberequalsize}{\mathcal S^{m,m}_\mathrm{KL}}
\newcommand{\cupdiags}{\mathbb B^{2m-k,k}_\mathrm{KL}}
\newcommand{\CupConnect}{\text{\----}}
\newcommand{\DCupConnect}{\text{\----}\hspace{-9pt}\filledsquare\hspace{2pt}}

\newcommand{\ba}{{\bf a}}
\newcommand{\bb}{{\bf b}}
\newcommand{\bigslant}[2]{{\raisebox{.2em}{$#1$}\left/\raisebox{-.2em}{$#2$}\right.}}

\newcounter{sarrow}
\newcommand\xrsquigarrow[1]{%
\stepcounter{sarrow}%
\begin{tikzpicture}[decoration=snake]
\node (\thesarrow) {\strut#1};
\draw[->,decorate] (\thesarrow.south west) -- (\thesarrow.south east);
\end{tikzpicture}%
}

\DeclareMathOperator{\Hom}{Hom}
\DeclareMathOperator{\im}{im}

\begin{document}

\begin{abstract}
We explain an elementary topological construction of the Springer representation on the homology of (topological) Springer fibers of types $C$ and $D$ in the case of nilpotent endomorphisms with two Jordan blocks. The Weyl group and component group actions admit a diagrammatic description in terms of cup diagrams which appear in the definition of arc algebras of types $B$ and $D$. We determine the decomposition of the representations into irreducibles and relate our construction to classical Springer theory. As an application we obtain presentations of the cohomology rings of all two-block Springer fibers of types $C$ and $D$. Moreover, we deduce explicit isomorphisms between the Kazhdan-Lusztig cell modules attached to the induced trivial module and the irreducible Specht modules in types $C$ and $D$. 
\end{abstract}

\maketitle
\tableofcontents
\section{Introduction}\label{section:introduction}

\subsection{Background and summary}
Let $G$ be a complex, connected, reductive, algebraic group with Lie algebra $\mathfrak{g}$. Given a nilpotent element $x\in\mathfrak{g}$, the associated (algebraic) Springer fiber $\mathcal B_G^x$ is defined as the variety of all Borel subgroups of $G$ whose Lie algebra contains $x$. Let $A^x_G$ be the component group associated with $x$, i.e.\ the quotient of the centralizer of $x$ in $G$ by its identity component. In seminal work \cite{Spr76}, \cite{Spr78}, Springer constructed a grading-preserving action of the Weyl group $\mathcal W_G$ associated with $G$ and a commuting action of $A^x_G$ on the (complex) cohomology $H^*(\mathcal B_G^x)$ of the Springer fiber. According to the Springer correspondence \cite[Theorem 6.10]{Spr76} the isotypic subspaces of the $A^x_G$-module $H^\mathrm{top}(\mathcal B_G^x)$ are irreducible $\mathcal W_G$-modules and all irreducible $\mathcal W_G$-modules can be constructed in this way. This yields thus a geometric construction of the irreducible representations of $\mathcal W_G$. 

This article connects classical Springer theory for types $C$ and $D$ with the cup diagram combinatorics of arc algebras of types $B$ and $D$. The latter provide combinatorial descriptions of interesting categories arising in the study of the representation theory of orthogonal Lie algebras~\cite{ES15}, non-semisimple Brauer algebras~\cite{ES16}, \cite{CDVMI}, and orthosymplectic Lie superalgebras~\cite{ES16super}, \cite{CH}. In type $A$, such a description was established in work of Russell and Tymoczko, \cite{RT11}, \cite{Rus11}, for nilpotent endomorphisms with two Jordan blocks. In this article we generalize their construction to all classical types (i.e.\ $C$ and $D$) and reconstruct the Springer representation on the homology of so-called topological Springer fibers with the advantage that the construction only uses elementary algebraic topology. The Weyl group action as well as the component group action both admit hereby a combinatorial description accessible for explicit computations (which appears to be difficult using Springer's original theory). This does not only provide a new, combinatorial perspective on the Springer representation but also presentations of the cohomology rings for all two-block Springer fibers of types $C$ and $D$, and moreover explicit isomorphisms relating the cup diagram basis of the Kazhdan-Lusztig cell modules attached to the induced trivial module and the bipolytabloid basis of irreducible Specht modules in types $C$ and $D$.

The surprising fact that two-block Springer theory in types $C$ and $D$ admits a concrete description using the combinatorics of arc algebras of types $B$ and $D$ should be interpreted as a form of Langlands duality between perverse sheaves on isotropic Grassmannians and coherent sheaves on compactified resolutions of Slodowy slices supported on the Springer fiber, see Subsection~\ref{subsection:Langlands}. 

\subsubsection*{Assumption} Throughout this article let $(2m-k,k)$, $1\leq k\leq m$, be a partition of $2m$ labeling a nilpotent orbit in $\mathfrak{so}_{2m}$, i.e.\ $k$ is odd or $m=k$ by the classification of nilpotent orbits for $\mathfrak{so}_{2m}$, see e.g.\ \cite{Wil37}, \cite{Ger61}.

\subsection{Topological Springer fibers and cup diagrams} \label{subsec:intro:cup_and_top_springer}

We begin by recalling the notion of a cup diagram from~\cite{LS13}, \cite{ES15}, as well as the definition of the topological Springer fiber from~\cite{ES12}, \cite{Wil15}, based on the earlier work, \cite{Kho04}, \cite{Rus11}, \cite{SW12}, in type $A$. Cup diagrams are the key diagrammatic tool throughout this article.  

Fix a rectangle in the plane together with $m$ vertices evenly spread along the upper horizontal edge of the rectangle. A {\it cup diagram} is a non-intersecting arrangement of cups and rays inside the rectangle. A cup is defined as a line segment connecting two distinct vertices and a ray is a line segment connecting a vertex with a point on the lower horizontal edge of the rectangle. Every vertex has to be joined with exactly one endpoint of a cup or ray. Additionally, any cup or ray which is accessible from the left side of the rectangle, meaning that there exists a path inside the rectangle connecting this cup or ray to the left vertical edge of the rectangle without intersecting any other part of the diagram, is allowed to be marked with one single marker (a black box). Diagrams related by a planar isotopy fixing the boundary are considered to be the same. 

Let $C_{\mathrm{KL}}(m)$ be the set of all cup diagrams on $m$ vertices such that the number of marked rays plus the number of unmarked cups is even.\footnote{The subscript ``$\mathrm{KL}$'' comes from the fact that these cup diagrams can be used to give a diagrammatic description of the \emph{K}azhdan-\emph{L}usztig basis of a parabolic Hecke module of type $D$ with maximal parabolic of type $A$, \cite{LS13}.} Here are all cup diagrams contained in $C_{\mathrm{KL}}(4)$:
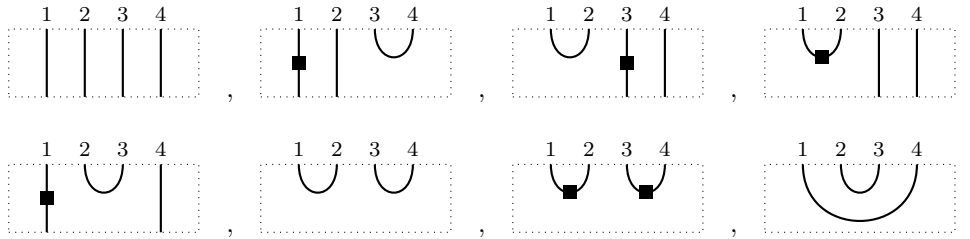
\[
\begin{array}{ccccccc}
\begin{tikzpicture}[baseline={(0,-.9)}]
\draw[dotted] (-.5,0) -- (2,0) -- (2,-.9) -- (-.5,-.9) -- cycle;
\begin{footnotesize}
\node at (0,.2) {1};
\node at (.5,.2) {2};
\node at (1,.2) {3};
\node at (1.5,.2) {4};
\end{footnotesize}

\draw[thick] (0,0) -- +(0,-.9);
\draw[thick] (.5,0) -- +(0,-.9);
\draw[thick] (1,0) -- +(0,-.9);
\draw[thick] (1.5,0) -- +(0,-.9);
\end{tikzpicture}
&,&
\begin{tikzpicture}[baseline={(0,-.9)}]
\draw[dotted] (-.5,0) -- (2,0) -- (2,-.9) -- (-.5,-.9) -- cycle;
\begin{footnotesize}
\node at (0,.2) {1};
\node at (.5,.2) {2};
\node at (1,.2) {3};
\node at (1.5,.2) {4};
\end{footnotesize}
\draw[thick] (1,0) .. controls +(0,-.5) and +(0,-.5) .. +(.5,0);

\draw[thick] (0,0) -- +(0,-.9);
\draw[thick] (.5,0) -- +(0,-.9);
\fill ([xshift=-2.5pt,yshift=-2.5pt]0,-.45) rectangle ++(5pt,5pt);
\end{tikzpicture}
&,&
\begin{tikzpicture}[baseline={(0,-.9)}]
\draw[dotted] (-.5,0) -- (2,0) -- (2,-.9) -- (-.5,-.9) -- cycle;
\begin{footnotesize}
\node at (0,.2) {1};
\node at (.5,.2) {2};
\node at (1,.2) {3};
\node at (1.5,.2) {4};
\end{footnotesize}
\draw[thick] (0,0) .. controls +(0,-.5) and +(0,-.5) .. +(.5,0);

\draw[thick] (1,0) -- +(0,-.9);
\draw[thick] (1.5,0) -- +(0,-.9);
\fill ([xshift=-2.5pt,yshift=-2.5pt]1,-.45) rectangle ++(5pt,5pt);
\end{tikzpicture}
&,&
\begin{tikzpicture}[baseline={(0,-.9)}]
\draw[dotted] (-.5,0) -- (2,0) -- (2,-.9) -- (-.5,-.9) -- cycle;
\begin{footnotesize}
\node at (0,.2) {1};
\node at (.5,.2) {2};
\node at (1,.2) {3};
\node at (1.5,.2) {4};
\end{footnotesize}
\draw[thick] (0,0) .. controls +(0,-.5) and +(0,-.5) .. +(.5,0);

\draw[thick] (1,0) -- +(0,-.9);
\draw[thick] (1.5,0) -- +(0,-.9);
\fill ([xshift=-2.5pt,yshift=-2.5pt].25,-.365) rectangle ++(5pt,5pt);
\end{tikzpicture}\hspace{.5em},
\end{array}
\]
\[
\begin{array}{ccccccc}
\begin{tikzpicture}[baseline={(0,-.9)}]
\draw[dotted] (-.5,0) -- (2,0) -- (2,-.9) -- (-.5,-.9) -- cycle;
\begin{footnotesize}
\node at (0,.2) {1};
\node at (.5,.2) {2};
\node at (1,.2) {3};
\node at (1.5,.2) {4};
\end{footnotesize}
\draw[thick] (.5,0) .. controls +(0,-.5) and +(0,-.5) .. +(.5,0);

\draw[thick] (0,0) -- +(0,-.9);
\draw[thick] (1.5,0) -- +(0,-.9);
\fill ([xshift=-2.5pt,yshift=-2.5pt]0,-.45) rectangle ++(5pt,5pt);
\end{tikzpicture}
&,&
\begin{tikzpicture}[baseline={(0,-.9)}]
\draw[dotted] (-.5,0) -- (2,0) -- (2,-.9) -- (-.5,-.9) -- cycle;
\begin{footnotesize}
\node at (0,.2) {1};
\node at (.5,.2) {2};
\node at (1,.2) {3};
\node at (1.5,.2) {4};
\end{footnotesize}
\draw[thick] (0,0) .. controls +(0,-.5) and +(0,-.5) .. +(.5,0);
\draw[thick] (1,0) .. controls +(0,-.5) and +(0,-.5) .. +(.5,0);

\end{tikzpicture}
&,&
\begin{tikzpicture}[baseline={(0,-.9)}]
\draw[dotted] (-.5,0) -- (2,0) -- (2,-.9) -- (-.5,-.9) -- cycle;
\begin{footnotesize}
\node at (0,.2) {1};
\node at (.5,.2) {2};
\node at (1,.2) {3};
\node at (1.5,.2) {4};
\end{footnotesize}
\draw[thick] (0,0) .. controls +(0,-.5) and +(0,-.5) .. +(.5,0);
\draw[thick] (1,0) .. controls +(0,-.5) and +(0,-.5) .. +(.5,0);

\fill ([xshift=-2.5pt,yshift=-2.5pt].25,-.365) rectangle ++(5pt,5pt);
\fill ([xshift=-2.5pt,yshift=-2.5pt]1.25,-.365) rectangle ++(5pt,5pt);
\end{tikzpicture}
&,&
\begin{tikzpicture}[baseline={(0,-.9)}]
\draw[dotted] (-.5,0) -- (2,0) -- (2,-.9) -- (-.5,-.9) -- cycle;
\begin{footnotesize}
\node at (0,.2) {1};
\node at (.5,.2) {2};
\node at (1,.2) {3};
\node at (1.5,.2) {4};
\end{footnotesize}
\draw[thick] (.5,0) .. controls +(0,-.5) and +(0,-.5) .. +(.5,0);
\draw[thick] (0,0) .. controls +(0,-1) and +(0,-1) .. +(1.5,0);

\end{tikzpicture}\hspace{.5em}.
\end{array}
\]
Usually we neither draw the rectangle nor display the numbering of the vertices.

Let $\cupdiags$ denote the subset of $C_\mathrm{KL}(m)$ consisting of all cup diagrams with exactly $\lfloor\frac{k}{2}\rfloor$ cups. Consider the standard unit sphere $\mathbb S^2\subseteq\mathbb R^3$ on which we fix the points $p=(0,0,1)$ and $q=(1,0,0)$. Given a cup diagram $\ba\in\cupdiags$, define $S_\ba$ as the subset of $\left(\mathbb S^2\right)^m$ consisting of all $(x_1,\ldots,x_m)\in\left(\mathbb S^2\right)^m$ which satisfy the coordinate relations 
\begin{itemize}
\item $x_i=-x_j$ (resp.\ $x_i=x_j$) if the vertices $i$ and $j$ are connected by an unmarked cup (resp.\ marked cup), and
\item $x_i=p$ if the vertex $i$ is connected to a marked ray and $x_i=-p$ (resp. $x_i=q$) if $i$ is connected to an unmarked ray which is the leftmost ray in $\ba$ (resp.\ not the leftmost ray). 
\end{itemize}
\begin{defi}
The {\it topological Springer fiber} $\SOfiber$ is then defined as the union
\[
\SOfiber:=\bigcup_{\ba\in\cupdiags}S_\ba \subseteq \left(\mathbb S^2\right)^m.
\]
\end{defi}
The importance of the topologial Springer fiber is based on the following crucial remark.

\begin{rem} \label{rem:wilbert}
There exists a homeomorphism $\SOfiber\cong\mathcal B^{2m-k,k}_{\mathrm{SO}_{2m}}$ under which the manifolds $S_\ba$ are mapped to the irreducible components of the Springer fiber, \cite[Theorem A]{Wil15}. Moreover, there exists an isomorphism of algebraic varieties $\mathcal B^{2m-k-1,k-1}_{\mathrm{Sp}_{2(m-1)}}\cong\mathcal B^{2m-k,k}_{\mathrm{SO}_{2m}}$ between two-row Springer fibers of type $C$ and type $D$, \cite[Theorem B]{Wil15}. Thus, we have also a homeomorphism $\SOfiber\cong\mathcal B^{2m-k-1,k-1}_{\mathrm{Sp}_{2(m-1)}}$ which allows us to use $\SOfiber$ as a topological model for both, types $C$ and $D$.
\end{rem}

\subsection{An elementary construction of the Weyl group and component group action}

In this subsection we sketch the ideas behind our construction of the Springer representation and state the explicit combinatorial description.

Recall that by classical Springer theory the centralizer $Z_G(x)$ acts on $\mathcal B^x_G$ (from the left) and thus on $H^*(\mathcal B_G^x)$. Since the identity component of $Z_G(x)$ acts trivially on $H^*(\mathcal B_G^x)$ we obtain an induced left action of the component group $A^x_G$ on $H^*(\mathcal B_G^x)$. Unlike the component group action, the $\mathcal W_G$-action on $H^*(\mathcal B_G^x)$ is {\it not} induced from a geometric action on the Springer fiber. Instead, the construction requires more sophisticated machinery such as perverse sheaves, Fourier transform, correspondences or monodromy, see e.g.\ \cite[\S 1.5]{Yun16} for a survey. As a consequence, the Springer action becomes notoriously difficult to understand explicitly. Guided by the abstract theory and inspired by \cite{RT11}, \cite{Rus11}, we present an elementary approach to define these actions based on the following result (cf.\ also Proposition~\ref{prop:basis_general_case}).

\begin{prop} \label{prop:injection}
The map $\gamma_{2m-k,k}\colon H_*(\SOfiber)\to H_*((\mathbb S^2)^m)$ induced by the natural inclusion $\SOfiber\subseteq (\mathbb S^2)^m$ in homology is injective.
\end{prop}

The idea is to explicitly define commuting actions of the groups $\mathcal W_G$ and $A_G^x$ on the space $(\mathbb S^2)^m$ which induce commuting actions on $H_*((\mathbb S^2)^m)$. By the injectivity of the map $\gamma_{2m-k,k}$ from Proposition~\ref{prop:injection}, we can identify $H_*(\SOfiber)$ with its image under $\gamma_{2m-k,k}$ in $H_*((\mathbb S^2)^m)$. This turns out to be a  subspace of $H_*((\mathbb S^2)^m)$ which is stable with respect to both actions.

In Section~\ref{section:section1} we construct a distinguished homology basis of $H_*(\SOfiber)$ using a cell decomposition of the $S_\ba$ which naturally arises from the cup diagram combinatorics. This yields the following diagrammatic description of the vector space $H_*(\SOfiber)$ (cf.\ again Proposition~\ref{prop:basis_general_case}).

\begin{prop}\label{introprop:diagrammatic_homology_basis}
There exists a distinguished basis of $H_{2l}(\SOfiber)$ labeled by the cup diagrams in $C_{\mathrm{KL}}(m)$ with precisely $l$ cups, i.e.\ $H_{2l}(\SOfiber)$ is isomorphic to the vector space freely generated by all cup diagrams in $C_{\mathrm{KL}}(m)$ with $l$ cups. 
\end{prop}

Via this diagrammatic basis, the action of the Weyl group and of the component group can be described combinatorially.

\subsubsection*{The Weyl group action}

The Weyl group $\mathcal W_{D_m}=\mathcal W_{\mathrm{SO}_{2m}}$ is a Coxeter group with Coxeter generators $s_0,s_1,\ldots,s_{m-1}$, see Section~\ref{section:section2}. There is a right action of $\mathcal W_{D_m}$ on $(\mathbb S^2)^m$, where $s_i$ permutes the coordinates $i$ and $i+1$ and $s_0$ permutes the first two coordinates and additionally takes their antipodes. Following the construction outlined above we obtain the first main result. 

\begin{introthm} \label{thm:skein_calculus}
Setting $\ba.s:=\gamma_{2m-k,k}^{-1}(\gamma_{2m-k,k}(\ba).s)$, where $s\in\mathcal W_{D_m}$ and $\ba$ is a cup diagram (viewed as a basis element of homology via Proposition~\ref{introprop:diagrammatic_homology_basis}), yields a well-defined, grading-preserving right action of the Weyl group $\mathcal W_{D_m}$ on $H_*(\SOfiber)$ which can be described using a skein calculus. Diagrammatically, we represent the generators as follows: 
\[
s_0:=\begin{tikzpicture}[thick, scale=.8, baseline={(0,-.5)}]
\draw[thick] (0,0) -- +(.5,-1);
\draw[thick] (.5,0) -- +(-.5,-1);
\draw[thick] (1,0) -- +(0,-1);
\draw[thick] (1.5,0) -- +(0,-1);
\draw[thick] (2.5,0) -- +(0,-1);

\node at (2,-.5) {\dots};

\fill ([xshift=-2.5pt,yshift=-2.5pt].13,-.25) rectangle ++(5pt,5pt);
\fill ([xshift=-2.5pt,yshift=-2.5pt].13,-.75) rectangle ++(5pt,5pt);

\begin{footnotesize}
\node at (0,.2) {1};
\node at (.5,.2) {2};
\end{footnotesize}
\end{tikzpicture}
\hspace{5em}
s_i:=\begin{tikzpicture}[thick, scale=.8, baseline={(0,-.5)}]
\draw[thick] (0,0) -- +(0,-1);
\draw[thick] (.5,0) -- +(0,-1);
\draw[thick] (1.5,0) -- +(.5,-1);
\draw[thick] (2,0) -- +(-.5,-1);
\draw[thick] (3,0) -- +(0,-1);

\node at (1,-.5) {\dots};
\node at (2.5,-.5) {\dots};

\begin{footnotesize}
\node at (1.4,.22) {i};
\node at (2,.2) {i+1};
\end{footnotesize}
\end{tikzpicture}
\hspace{2em}
i\in\{1,\ldots,m-1\}.
\] 
For a cup diagram $\ba$ and $s\in\mathcal W_{D_m}$ with a reduced expression $s=s_{i_1}s_{i_2}\cdots s_{i_l}$  we obtain $\ba.s$ by first stacking the pictures corresponding to the generators $s_{i_1},\ldots,s_{i_l}$ on top of $\ba$ and then resolving crossings according to the rule
\[
\begin{tikzpicture}[thick, scale=.8]
\draw[dotted] (.25,-.5) circle(16pt);
\draw[dotted] (2.25,-.5) circle(16pt);
\draw[dotted] (4.25,-.5) circle(16pt);
\draw[thick] (1.9,-.1) .. controls +(0,-.4) and +(0,-.4) .. +(.7,0);
\draw[thick] (1.9,-.9) .. controls +(0,.4) and +(0,.4) .. +(.7,0);

\draw[thick] (0,0) -- +(.5,-1);
\draw[thick] (.5,0) -- +(-.5,-1);
\draw[thick] (3.9,-.1) .. controls +(.15,0) and +(.15,0) .. +(0,-.8);
\draw[thick] (4.6,-.1) .. controls +(-.15,0) and +(-.15,0) .. +(0,-.8);

\node at (1.2,-.5) {$=$};
\node at (3.2,-.5) {$+$};
\end{tikzpicture}
\]
using the following local relations
\[
\begin{tikzpicture}[thick, scale=.8]
\draw[dotted] (-.25,-.5) circle(16pt);
\draw[dotted] (2.7,-.5) circle(16pt);
\draw[thick] (-.25,-.5) circle(8pt);

\node at (1.3,-.5) {$= (-2)\cdot$};
\end{tikzpicture}
\hspace{3em}
\begin{tikzpicture}[thick, scale=.8]
\draw[dotted] (-.25,-.5) circle(16pt);
\draw[dotted] (1.9,-.5) circle(16pt);

\draw[thick] (-.75,-.5) -- +(1,0);
\draw[thick] (1.4,-.5) -- +(1,0);
\fill ([xshift=-2.5pt,yshift=-2.5pt]-.05,-.5) rectangle ++(5pt,5pt);
\fill ([xshift=-2.5pt,yshift=-2.5pt]-.45,-.5) rectangle ++(5pt,5pt);

\node at (.8,-.5) {$=$};
\end{tikzpicture}
\hspace{3em}
\begin{tikzpicture}[thick, scale=.8]
\draw[dotted] (-.25,-.5) circle(16pt);
\draw[thick] (-.25,-.5) circle(8pt);

\fill ([xshift=-2.5pt,yshift=-2.5pt]-.25,-.225) rectangle ++(5pt,5pt);

\node at (.8,-.5) {$=\,0$};
\end{tikzpicture}
\]
together with the rule that we kill all diagrams with a connected component where both endpoints are at the bottom of the diagram. We obtain $\ba.s$, which is a linear combination of cup diagrams. 
\end{introthm}

By Remark~\ref{rem:wilbert} we have $\SOfiber\cong\mathcal B^{2m-k-1,k-1}_{\mathrm{Sp}_{2(m-1)}}$. According to Springer theory we therefore expect an interesting action of the Weyl group $\mathcal W_{C_{m-1}}\cong\mathcal W_{\mathrm{Sp}_{2(m-1)}}$ of type $C$ on $H_*(\SOfiber)$. The Weyl group $\mathcal W_{C_{m-1}}$ can be identified with the subgroup of $\mathcal W_{D_m}$ generated by $s_0s_1$ and $s_i$ for $i\in\{2,\ldots,m-1\}$. In particular, we obtain an action of the Weyl group of type $C$ on $H_*(\SOfiber)$ by restricting the action of $\mathcal W_{D_m}$. This action turns out to be the Springer action for type $C$. Note that Theorem~\ref{thm:skein_calculus} automatically yields an explicit description of this restricted action.

\subsubsection*{The component group action} 
In the case of classical groups, Spaltenstein showed in \cite[I.2.9]{Spa82} that the component group $A^x_G$ is isomorphic to a finite product of copies of $\mathbb Z/2\mathbb Z$. Moreover, $A^x_G$ depends (up to isomorphism) only on the Jordan type $\lambda$ of $x$ which allows us to use the notation $A^{\lambda}_G$. In the two-block case we have isomorphisms (see also Section~\ref{sec:group_actions} for more details),
\begin{eqnarray}
\label{compgroupsexplicit}
A^{2m-k-1,k-1}_{\mathrm{Sp}_{2(m-1)}} \cong\begin{cases}
\{e\} &\text{if }m=k\text{ is even}, \\
\mathbb Z/2\mathbb Z &\text{if }m=k\text{ is odd}, \\
(\mathbb Z/2\mathbb Z)^2 &\text{if }m\neq k,
\end{cases}
&
\text{and}
&
A^{2m-k,k}_{\mathrm{SO}_{2m}} \cong\begin{cases}
\{e\} &\text{if }m=k, \\
\mathbb Z/2\mathbb Z &\text{if }m\neq k.
\end{cases}
\end{eqnarray}

In type $D$ the action of the component group turns out to be trivial (even though the group itself is not trivial if $m\neq k$). Thus, we only consider the type $C$ case and construct an action of $A^x_{\mathrm{Sp}_{2(m-1)}}$ on $(\mathbb S^2)^m$ by setting $\alpha.(x_1,x_2,\ldots,x_m)=(-x_1,x_2,\ldots,x_m)$, where $\alpha\in A^x_{\mathrm{Sp}_{2(m-1)}}$ is a generator corresponding to a copy of $\mathbb Z/2\mathbb Z$ in $A^x_{\mathrm{Sp}_{2(m-1)}}$. 

\begin{introthm}
\label{thmB}
\begin{enumerate}
\item Setting $\alpha.\ba:=\gamma_{2m-k,k}^{-1}\left(\alpha.\gamma_{2m-k,k}(\ba)\right)$, where $\alpha\in A^{2m-k-1,k-1}_{\mathrm{Sp}_{2(m-1)}}$ and $\ba$ is a cup diagram, yields a well-defined grading-preserving left action of the component group on $H_*(\SOfiber)$. 
\item A generator $\alpha\in A^{2m-k-1,k-1}_{\mathrm{Sp}_{2(m-1)}}$ corresponding to a copy of $\mathbb Z/2\mathbb Z$ in $A^{2m-k-1,k-1}_{\mathrm{Sp}_{2(m-1)}}$ acts as the identity on $\ba$ if $\ba$ has a ray connected to the first vertex. Otherwise it creates (resp.\ kills) a marker on the cup connected to the first vertex if it is unmarked (resp.\ marked) and at the same time creates (resp.\ kills) a marker on the leftmost ray.
\end{enumerate}
\end{introthm}
Here is an example, where the leftmost ray is connected to vertex $j$. Note that if $A^{2m-k-1,k-1}_{\mathrm{Sp}_{2(m-1)}}$ is not trivial there always exists a (leftmost) ray.  
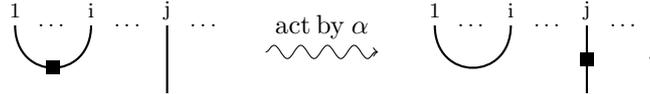
\[
\begin{tikzpicture}[baseline={(0,-.45)}]
\begin{footnotesize}
\node at (0,.2) {$\mathrm{1}$};
\node at (1,.2) {$\mathrm{i}$};
\node at (2,.2) {$\mathrm{j}$};
\node at (.5,0) {\ldots};
\node at (1.5,0) {\ldots};
\node at (2.5,0) {\ldots};
\end{footnotesize}
\draw[thick] (0,0) .. controls +(0,-.75) and +(0,-.75) .. +(1,0);
\draw[thick] (2,0) -- +(0,-.9);
\fill ([xshift=-2.5pt,yshift=-2.5pt].5,-.55) rectangle ++(5pt,5pt);
\end{tikzpicture}
\hspace{1.4em}
\xrsquigarrow{$\mathrm{act}\hspace{.2em}\mathrm{by}\hspace{.25em}\alpha$}
\hspace{1.6em}
\begin{tikzpicture}[baseline={(0,-.45)}]
\begin{footnotesize}
\node at (0,.2) {$\mathrm{1}$};
\node at (1,.2) {$\mathrm{i}$};
\node at (2,.2) {$\mathrm{j}$};
\node at (.5,0) {\ldots};
\node at (1.5,0) {\ldots};
\node at (2.5,0) {\ldots};
\end{footnotesize}
\draw[thick] (0,0) .. controls +(0,-.75) and +(0,-.75) .. +(1,0);
\draw[thick] (2,0) -- +(0,-.9);
\fill ([xshift=-2.5pt,yshift=-2.5pt]2,-.45) rectangle ++(5pt,5pt);
\end{tikzpicture}.
\]

\subsection{Connection to Springer theory} To compare our results to classical Springer theory we determine the decomposition of our representations into irreducibles in each homological degree. 

\subsubsection*{Case 1: Two Jordan blocks of equal size:} Consider the generic Hecke algebra $\mathcal H(\mathcal W_{D_m})$ attached to $\mathcal W_{D_m}$ with its standard basis $H_w$, $w\in\mathcal W_{D_m}$ as a $\mathcal L=\mathbb C[q,q^{-1}]$-module. Inside $\mathcal H(\mathcal W_{D_m})$ we have the subalgebra $\mathcal H(S_m)$ attached to the maximal parabolic subgroup in $\mathcal W_{D_m}$ generated by $s_1,\ldots,s_{m-1}$. Letting $H_{s_i}$, $i\in\{1,\ldots,m-1\}$ act as $q^{-1}$ turns $\mathcal L$ into an $H(S_m)$-module. The parabolic Hecke module $\mathcal M_{D_m}$ is obtained by tensoring $\mathcal L$ over $\mathcal H(S_m)$ with $\mathcal H(\mathcal W_{D_m})$. By specializing $q=1$, $\mathcal M_{D_m}$ is isomorphic to the induced trivial module.

In \cite[Theorem 5.17]{LS13} the authors identified $\mathcal M_{D_m}$ with the free $\mathcal L$-module whose basis is given by all cup diagrams in $C_\mathrm{KL}(m)$. Via this identification a cup diagram corresponds to a Kazhdan-Lusztig basis vector and the right action of $\mathcal H(\mathcal W_{D_m})$ on $\mathcal M_{D_m}$ admits a diagrammatic description similar to the one of Theorem~\ref{thm:skein_calculus} (cf.\ Propositions~\ref{prop:diagrammatic_Hecke_actionI}). In Section~\ref{sec:diagrammatic_Hecke_module} we generalize this construction by providing a similar diagrammatic description of $\mathcal M_{\mathrm{Sp}_{2(m-1)}}$ (cf.\ Proposition~\ref{prop:diagrammatic_Hecke_actionII}). The diagrammatic description of the parabolic Hecke modules is used to prove the following theorem. 

\begin{introthm} \label{thm:induced_module}
\begin{enumerate}
\item The right $\mathcal W_{D_m}$-module $H^*(\SOfiberequalsize)$ constructed in Theorem~\ref{thm:skein_calculus} is isomorphic to the induced trivial module $\mathbb C\otimes_{\mathbb C[S_m]}\mathbb C[\mathcal W_{D_m}]$. 
\item Similarly, we have an isomorphism $H_*(\SOfiberequalsize)\cong\mathbb C\otimes_{\mathbb C[S_{m-1}]}\mathbb C[\mathcal W_{C_{m-1}}]$ of $\mathcal W_{C_{m-1}}$-modules, where the right $\mathcal W_{C_{m-1}}$-action on $H_*(\SOfiberequalsize)$ is obtained by restricting the $\mathcal W_{D_m}$-action.  
\end{enumerate}
\end{introthm}

\begin{rem}
In contrast to the related work \cite{RT11} in type $A$ we work with the entire homology and identify it as an induced representation. This approach has the following advantages: 
\begin{enumerate}
\item Since the Springer representation $H^*(\SOfiberequalsize)$ (in Lusztig's version, \cite{Lus81}) is isomorphic to the induced trivial module, see \cite[Theorem 1.3]{Lus04}, Theorem~\ref{thm:induced_module} immediately implies that Theorem~\ref{thm:skein_calculus} indeed reconstructs the Springer representation (similarly for type $C$). 
\item Since the decomposition of the induced trivial module into irreducibles is known and multiplicity free, see e.g. \cite[Lemma~5.19]{ES12}, we can easily match each of the representations $H_{2l}(\SOfiberequalsize)$, $l\in\{0,\ldots,\lfloor\frac{m}{2}\rfloor\}$, with exactly one irreducible of the induced trivial module in type $D$, see Subsection~\ref{subsection:proof_induced_module} for details. Since the $\mathcal W_{C_{m-1}}$-action is the restriction of the $\mathcal W_{D_m}$-action (for which we have determined the irreducibles) we obtain the desired decomposition by applying results of \cite{Tok84}.
\end{enumerate}
\end{rem}

\subsubsection*{Case 2: Two Jordan blocks of different size:} For the case of two not necessarily equally-sized Jordan blocks, we can identify $H_*(\SOfiber)\cong H_{<k}(\SOfiberequalsize)$ as graded $\mathcal W_{D_m}$-module (and thus also as $\mathcal W_{C_{m-1}}$-module). In fact, this is evident from the diagrammatic description in Theorem~\ref{thm:skein_calculus} and reduces the general two-block case to {\it Case 1} above.

The Springer correspondence was described for the classical groups in work of Shoji, \cite{Sho83}, \cite{Sho79}. Given an irreducible $\mathcal W_G$-module $V$, Shoji gives an algorithm which determines the Jordan type of a nilpotent element $x\in\mathfrak{g}$ as well as the irreducible character $\phi$ of $A^x_G$ for which $H^{\mathrm{top}}_{\phi}(\mathcal B^x_G)\cong V$ in the Springer correspondence, see \cite[Theorem 3.3]{Sho79}. Here, $H^{\mathrm{top}}_{\phi}(\mathcal B^x_G)$ denotes the $\phi$-isotypic subspace in $H^{\mathrm{top}}(\mathcal B^x_G)$. In Section~\ref{subsec:Shoji} we recall this algorithm and use it to check that our construction yields the same results as classical Springer theory. 

\subsection{Applications of our construction}

We provide two applications of our construction.

\subsubsection*{Cohomology rings of two-block Springer fibers in types C and D}

By \cite[Theorem 1.1]{HS77}, the image of the canonical map $H^*(\mathcal B_G)\to H^*(\mathcal B_G^x)$ induced by the inclusion of the Springer fiber into the flag variety $\mathcal B_G$ associated with $G$ is given by the $A_G^x$-invariants in $H_*(\mathcal B_G^x)$. In particular, the canonical map is surjective if and only if the $A_G^x$-action is trivial. In type $A$ the component groups are themselves trivial, \cite[Lemma~3.6.3]{CG97}, and the canonical map can be used to provide presentations of $H^*(\mathcal B_{\mathrm{GL}_n}^x)$ as a quotient of $H^*(\mathcal B_{\mathrm{GL}_n})$, see \cite{DCP81}, \cite{Tan82}.

Outside of type $A$ the cohomology rings of the Springer fibers are much harder to study. In fact, even the Euler characteristic of Springer fibers was only recently computed in \cite{Kim16} (an independent proof for all two-block Springer fibers using completely different methods is included in the Appendix to this article). In \cite{KP12} the cohomology ring of Springer fibers was described as an affine coordinate ring for all cases in which the canonical map is surjective and the nilpotent operator is of standard Levi type. In \cite[Theorem B]{ES12} this description was used to obtain a presentation of the ring $H^*(\mathcal B^x_{\mathrm{SO}_{2m}})$, where $x\in\mathfrak{so}_{2m}$ has Jordan type $(m,m)$. In the general two-block case the methods of~\cite{KP12} are not applicable because the nilpotent operator is not of standard Levi type. In type $C$ even the surjectivity of the canonical map fails. However, by dualizing Proposition~\ref{prop:injection}, our topological approach enables us to realize the cohomology ring of $\SOfiber$ as a quotient of $H^*((\mathbb S^2)^m)$ which yields the following result:

\begin{introthm} \label{introthm:cohomology}
If $m\neq k$, then we have an isomorphism of graded algebras
\[
H^*\left(\SOfiber\right)\cong\bigslant{\mathbb C[X_1,\ldots,X_m]}{\left\langle X_i^2, X_I\;
  \begin{array}{|c}
  1 \leq i \leq m,\\
  I\subseteq  \{1,\ldots,m\}, |I|=\frac{k-1}{2}+1
  \end{array}
  \right\rangle},
\]
where $X_I=\prod_{i\in I}X_i$. In particular, in combination with \cite[Theorem B]{ES12} and Remark~\ref{rem:wilbert} this yields a description of the cohomology rings of all two-block Springer fibers of types $C$ and $D$.
\end{introthm}

\subsubsection*{Relating cup diagram and bipolytabloid bases of Specht modules}
The parabolic Hecke module $\mathcal M_{D_m}$ has a filtration 
\[
\{0\}\subseteq\mathcal M_{D_m}^{\lfloor\frac{m}{2}\rfloor}\subseteq\ldots\subset\mathcal M^l_{D_m}\subseteq\ldots\subseteq\mathcal M^0_{D_m}=\mathcal M_{D_m}
\]
by Kazhdan-Lusztig cells, see Section~\ref{sec:diagrammatic_Hecke_module}. From the diagrammatic description of $\mathcal M_{D_m}$ it follows that $\mathcal M_{D_m}^l$ has a basis given by all cup diagrams with $l$ or more cups. The associated graded with respect to this filtration (specialized at $q=1$) gives the $\mathcal W_{D_m}$-module $H_*(\SOfiber)$ from Theorem~\ref{thm:skein_calculus}, i.e.\ we have isomorphisms of $\mathcal W_{D_m}$-modules between the Kazhdan-Lusztig cell module $\mathcal M^l_{D_m}/\mathcal M^{l+1}_{D_m}$ and $H_{2l}(\SOfiber)$. In particular, the cup diagram basis of $H_*(\SOfiber)$ is not the Kazhdan-Lusztig basis (except in top degree) which disproves the type $D$ analog of \cite[Conjecture 4.4]{RT11} (in fact, this argument shows that this conjecture does not hold in type $A$ either).  

In type $D$ the cell modules are irreducible and in Theorem~\ref{thm:explicit_identification_D} we write down an explicit isomorphism between the cup diagram basis of $H_{2l}(\SOfiber)$ and the bipolytabloid basis of the Specht module. In type $C$ the $\mathcal W_{C_{m-1}}$-module $H_{2l}(\SOfiber)$ is not irreducible anymore (which can be seen geometrically from the nontrivial component group action). In Proposition~\ref{prop:basis_isotypic_comp} we explicitly determine the decomposition of $H_{2l}(\SOfiber)$ into irreducible summands by providing a basis in terms of sums of cup diagrams. As in type $D$ we obtain an explicit isomorphism between each irreducible summand and the corresponding Specht module, see Theorem~\ref{thm:explicit_identification_C1}. This generalizes results for type $A$, see~\cite{Nar89}, to types $C$ and $D$ and answers a question in~\cite[Remark 5.21]{ES12}.  

\subsection{An instance of Langlands duality}\label{subsection:Langlands}

We finish the introductory part of this article by discussing the reason for the appearance of the cup diagrams in our construction of the Springer representation from a higher, categorical point of view. We also explain a (conjectural) reason for the surprising fact that the Springer representations in types $C$ and $D$ both admit a neat description in terms of the same cup diagram combinatorics. 
  
In \cite{ES15}, the authors defined an algebra $\mathbb D_m$ whose underlying vector space has a basis given by oriented circle diagrams (these are obtained by putting a cup diagram upside down on top of another cup diagram and adding an orientation to the connected components of the resulting diagram). The construction of this so-called arc algebra of type $D$ was inspired by the arc algebras of type $A$ defined in \cite{Kho02}, \cite{CK14}, \cite{BS11}.

Let $\Upsilon_m^D$ denote the isotropic Grassmannian of $m$-dimensional isotropic subspaces in $\mathbb C^{2m}$ and let $\mathcal{P}erv\left(\Upsilon_m^D\right)$ be the associated category of perverse sheaves (constructible with respect to the Schubert stratification). Using results of Braden~\cite{Bra02} one can establish an equivalence of categories $\mathbb D_m$-$\mathrm{mod}\simeq\mathcal{P}erv\left(\Upsilon^D_m\right)$, see~\cite[Theorem 9.1]{ES15}, which shows that the cup diagram combinatorics is intricately connected to the geometry of perverse sheaves.

In \cite{SW12} it was shown that the category of finite-dimensional modules over the arc algebra of type $A$ is equivalent to the heart of an exotic $t$-structure (in the sense of \cite{Bez03}) on the bounded derived category $D^b(\mathrm{Coh}(Y^A_m))$ of coherent sheaves on a certain compactification $Y^A_m$ of a resolution of a Slodowy slice of type $A$. We also expect such an equivalence between the category $\mathcal{P}erv(\Upsilon_m^D)$ and the core of an exotic $t$-structure on the derived category of coherent sheaves (supported on the $(m,m)$ Springer fiber) on a compactification of a resolution of a Slowdowy slice $Y_m^D$ of type $D$. In particular, this equivalence relates the cup diagram combinatorics with the geometry of the Springer fiber and thus explains the appearance of the cup diagrams in our description of the Springer representation. As in type $A$, see~\cite{AS}, \cite{KS}, we also expect a direct connection to a Fukaya category related to $Y^D_m$ or to Kleinian singularities of type $D$. 

The constructions in \cite{SW12} show that the equivalence between perverse sheaves on the Grassmannian and the core of the exotic $t$-structure of $D^b(\mathrm{Coh}(Y_m^A))$ can be interpreted as an incarnation of Koszul duality which by \cite{BGS96} suggests that one should actually replace $Y_m^A$ by its Langlands dual $Y_m^{A^\vee}$. This is subtly hidden in the arguments since type $A$ (as well as type $D$) are Langlands self-dual. However, we expect an equivalence of categories between perverse sheaves on the isotropic Grassmannian of type $B$ and a core of an exotic $t$-structure on a derived category of coherent sheaves on a resolution of a Slodowy slice of type $C$. 

By \cite{BP99} (see also \cite[\S9.7]{ES15}) there exists an equivalence $\mathcal{P}erv(\Upsilon^B_{m-1})\cong\mathcal{P}erv(\Upsilon^D_m)$ of categories induced by an isomorphism (compatible with the Schubert stratification) between the respective isotropic Grassmannians. The combinatorics of these categories can be described in terms of the same cup diagrams. On the Langlands dual side, which is related to the Springer fibers, this then corresponds to an equivalence $D^b(\mathrm{Coh}(Y^C_{m-1}))\cong D^b(\mathrm{Coh}(Y^D_m))$ induced by the isomorphism $Y^C_{m-1}\cong Y^D_m$ between (compactified) resolutions of Slodowy slices recently discovered in~\cite[\S8.4]{Li18}. This gives a conceptual explanation of the surprising fact that the Springer representations in types $C$ and $D$ both admit a neat description in terms of the same combinatorics underlying the arc algebra $\mathbb D_m$.

\subsection*{Notation and conventions} 
In this article the term ``vector space'' means ``complex vector space'' and a ``graded vector space'' is a $\mathbb Z$-graded complex vector space. Given a finite set $S$ and a commutative ring $R$, we write $R[S]$ to denote the free $R$-module with basis $S$. If $X$ is a topological space, we write $H_*(X)$ and $H^*(X)$ to denote its singular homology and cohomology with complex coefficients, i.e.\ $H_*(X)=H_*(X,\mathbb C)$ and $H^*(X)=H^*(X,\mathbb C)$. Since we work over the complex numbers, homology and cohomology are dual, i.e.\ the universal coefficient theorem provides natural isomorphisms $H^*(X)\cong\Hom_{\mathbb C}(H_*(X),\mathbb C)$ of graded vector spaces, which enable us to transfer results obtained for cohomology to homology and vice versa. Throughout this article we view $\mathbb Z/2\mathbb Z\cong\{1,-1\}\subseteq\mathbb C^*$ as a subgroup of the multiplicative group.   

\subsection*{Acknowledgements}
The authors thank Michael Ehrig and Daniel Tubbenhauer for many helpful comments and interesting discussions and Daniel Juteau and Martina Lanini for raising the question how to make the action of the component group explicit in our construction. Parts of this article appeared in the second author's PhD thesis under the supervision of the first author. 

\section[sec]{\for{toc}{A diagrammatic homology basis}\except{toc}{A diagrammatic homology basis and the proof of Theorem~\ref{introthm:cohomology}}}
\label{section:section1} 

Recall the set $\cupdiags$ of cup diagrams from the introduction. Given $\ba\in\cupdiags$, let $\phi_\ba\colon S_\ba\hookrightarrow\SOfiber$, $\psi_\ba\colon S_\ba\hookrightarrow (\mathbb S^2)^m$ be the inclusions. We obtain a commutative diagram
\begin{equation} \label{eq:key_comm_diagram}
\begin{xy}
	\xymatrix{
		\bigoplus\limits_{\ba\in\cupdiags} H_*(S_\ba) \ar[rr]^{\phi_{2m-k,k}} \ar@/^2pc/[rrrr]^{\psi_{2m-k,k}} && H_*(\SOfiber) \ar[rr]^{\gamma_{2m-k,k}}	&& H_*((\mathbb S^2)^m)
	}
\end{xy},
\end{equation} 
where $\phi_{2m-k,k}$ (resp.\ $\psi_{2m-k,k}$) is the direct sum of the maps induced by $\phi_\ba$ (resp.\ $\psi_\ba$) and $\gamma_{2m-k,k}$ is induced by the inclusion $\SOfiber\hookrightarrow\left(\mathbb S^2\right)^m$. The goal of this section is to understand the diagram~(\ref{eq:key_comm_diagram}) from an explicit combinatorial point of view. In particular, this includes the construction of a diagrammatic basis of $H_*\left(\SOfiber\right)$. At the end of this section we will prove Theorem~\ref{introthm:cohomology}. 

\subsection{Homology bases via cell decompositions}

We begin by defining cell decompositions of the manifolds $\left(\mathbb S^2\right)^m$ and $S_\ba$, $\ba\in\cupdiags$, which are used to obtain a diagrammatic description of the homology of the respective spaces. Recall the notion of a line diagram, \cite[Definition 3.1]{RT11}:
\begin{defi}
A {\it line diagram} on $m$ vertices is obtained by attaching vertical rays, each possibly decorated with a single white dot, to $m$ vertices on a horizontal line. Given $U\subseteq \{1,\ldots,m\}$, let $l_U$  be the unique line diagram with white dots precisely on the lines connected to vertices not contained in $U$ (vertices are counted from left to right). We write $\mathfrak{L}_m$ to denote the set of all line diagrams on $m$ vertices. 
\end{defi}

\begin{ex} \label{ex:line_diagrams}
Here is a complete list of all line diagrams in $\mathfrak{L}_2$: 
\[
l_\emptyset=
\begin{tikzpicture}[baseline={(0,-.5)},scale=.8]
\draw[thick] (0,0) -- +(0,-1);
\draw[thick] (.5,0) -- +(0,-1);
\draw (0,-.5) circle(3pt);
\draw (.5,-.5) circle(3pt);
\end{tikzpicture}\hspace{1em},
\hspace{3em}
l_{\{1\}}=
\begin{tikzpicture}[baseline={(0,-.5)},scale=.8]
\draw[thick] (0,0) -- +(0,-1);
\draw[thick] (.5,0) -- +(0,-1);
\draw (.5,-.5) circle(3pt);
\end{tikzpicture}\hspace{1em},
\hspace{3em}
l_{\{2\}}=
\begin{tikzpicture}[baseline={(0,-.5)},scale=.8]
\draw[thick] (0,0) -- +(0,-1);
\draw[thick] (.5,0) -- +(0,-1);
\draw (0,-.5) circle(3pt);
\end{tikzpicture}\hspace{1em},
\hspace{3em}
l_{\{1,2\}}=
\begin{tikzpicture}[baseline={(0,-.5)},scale=.8]
\draw[thick] (0,0) -- +(0,-1);
\draw[thick] (.5,0) -- +(0,-1);
\end{tikzpicture}\,\,.
\]
\end{ex}

Now, the two-sphere $\mathbb S^2$ has a cell decomposition $\mathbb S^2=\left(\mathbb S^2\setminus\{p\}\right)\cup \{p\}$ consisting of a point and a two-cell. In the following we fix this CW-structure and equip $(\mathbb S^2)^m$ with the Cartesian product CW-structure. Then we obtain a one-to-one correspondence between $\mathfrak{L}_m$ and the cells of the CW-complex $(\mathbb S^2)^m$ by sending a line diagram $l_U$ to the cell $C_{l_U}$, where $C_{l_U}$ is defined by choosing the $0$-cell (resp.\ $2$-cell) for the $i$th sphere in $(\mathbb S^2)^m$ if there is a dot (resp.\ no dot) on the $i$th line of $l_U$. Since $H_*((\mathbb S^2)^m)$ has a basis given by the homology classes $[C_{l_U}]$ of the cells, we obtain an isomorphism of vector spaces
\begin{equation} \label{eq:homology_identification_for_spheres}
\mathbb C[\mathfrak{L}_m] \xrightarrow\cong H_*\left((\mathbb S^2)^m\right)\,,\,\,l_U\mapsto [C_{l_U}].
\end{equation}
Henceforth, we will view a line diagram as a basis element of $H_*\left((\mathbb S^2)^m\right)$ via the identification (\ref{eq:homology_identification_for_spheres}). Note that its homological degree is given by twice the number of lines without a dot. 

\begin{ex}
The cells of $\mathbb S^2\times\mathbb S^2$ corresponding to the line diagrams in Example~\ref{ex:line_diagrams} are given  (in the same order) respectively by
\[
\{(p,p)\},\hspace{.6em}\{(x,p)\mid x\in\mathbb S^2\setminus\{p\}\},\hspace{.6em}\{(p,x)\mid x\in\mathbb S^2\setminus\{p\}\},\hspace{.6em}\{(x,y)\mid x,y\in\mathbb S^2\setminus\{p\}\}.
\]
\end{ex}

Next we will construct a basis of $H_*(S_\ba)$ labeled by enriched cup diagrams.

\begin{defi}
An {\it enriched cup diagram} is a cup diagram (as defined in the introduction) in which some of the cups and all of the rays are decorated with a white dot. We allow at most one dot per component and we do not impose any accessibility condition as for the markers. If a cup (resp.\ ray) is marked with a marker and a dot we place the marker to the left of (resp.\ above) the dot. Let $\widetilde{\mathbb B}^{2m-k,k}_\mathrm{KL}$ be the set of all enriched cup diagrams on $m$ vertices with $\lfloor\frac{k}{2}\rfloor$ cups such that the number of marked rays plus unmarked cups is even.  
\end{defi}

\begin{ex}
The enriched cup diagrams corresponding to the cup diagrams in $\mathbb B^{3,3}_\mathrm{KL}$ are given by
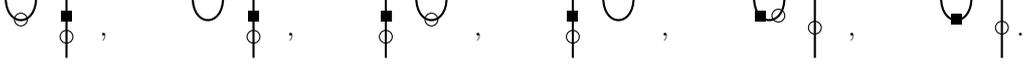
\[
\begin{tikzpicture}[baseline={(0,-.5)},scale=.8]
\begin{scope}[xshift=3cm]
\draw[thick] (0,0) .. controls +(0,-.5) and +(0,-.5) .. +(.5,0);

\draw[thick] (1,0) -- +(0,-1);
\fill ([xshift=-2.5pt,yshift=-2.5pt]1,-.3) rectangle ++(5pt,5pt);
\draw (1,-.65) circle(3pt);

\draw (.25,-.365) circle(3pt);
\end{scope}
\end{tikzpicture}\hspace{1em},
\hspace{3em}
\begin{tikzpicture}[baseline={(0,-.5)}, scale=.8]
\begin{scope}[xshift=3cm]
\draw[thick] (0,0) .. controls +(0,-.5) and +(0,-.5) .. +(.5,0);

\draw[thick] (1,0) -- +(0,-1);
\fill ([xshift=-2.5pt,yshift=-2.5pt]1,-.3) rectangle ++(5pt,5pt);
\draw (1,-.65) circle(3pt);
\end{scope}
\end{tikzpicture}\hspace{1em},
\hspace{3em}
\begin{tikzpicture}[baseline={(0,-.5)}, scale=.8]
\begin{scope}[xshift=3cm]
\draw[thick] (.5,0) .. controls +(0,-.5) and +(0,-.5) .. +(.5,0);
\draw (.75,-.365) circle(3pt);

\draw[thick] (0,0) -- +(0,-1);
\fill ([xshift=-2.5pt,yshift=-2.5pt]0,-.3) rectangle ++(5pt,5pt);
\draw (0,-.65) circle(3pt);
\end{scope}
\end{tikzpicture}\hspace{1em},
\hspace{3em}
\begin{tikzpicture}[baseline={(0,-.5)}, scale=.8]
\begin{scope}[xshift=3cm]
\draw[thick] (.5,0) .. controls +(0,-.5) and +(0,-.5) .. +(.5,0);

\draw[thick] (0,0) -- +(0,-1);
\fill ([xshift=-2.5pt,yshift=-2.5pt]0,-.3) rectangle ++(5pt,5pt);
\draw (0,-.65) circle(3pt);
\end{scope}
\end{tikzpicture}\hspace{1em},
\hspace{3em}
\begin{tikzpicture}[baseline={(0,-.5)}, scale=.8]
\begin{scope}[xshift=3cm]
\draw[thick] (0,0) .. controls +(0,-.5) and +(0,-.5) .. +(.5,0);

\draw[thick] (1,0) -- +(0,-1);
\draw (1,-.5) circle(3pt);

\draw (.4,-.3) circle(3pt);
\fill ([xshift=-2.5pt,yshift=-2.5pt].1,-.3) rectangle ++(5pt,5pt);
\end{scope}
\end{tikzpicture}\hspace{1em},
\hspace{3em}
\begin{tikzpicture}[baseline={(0,-.5)}, scale=.8]
\begin{scope}[xshift=3cm]
\draw[thick] (0,0) .. controls +(0,-.5) and +(0,-.5) .. +(.5,0);

\draw[thick] (1,0) -- +(0,-1);
\draw (1,-.5) circle(3pt);

\fill ([xshift=-2.5pt,yshift=-2.5pt].25,-.365) rectangle ++(5pt,5pt);
\end{scope}
\end{tikzpicture}\,\,.
\]
\end{ex}

Given $\ba\in\mathbb B^{2m-k,k}_\mathrm{KL}$, let $i_1<i_2<\cdots<i_{\lfloor\frac{k}{2}\rfloor}$ denote the left endpoints of the cups in $\ba$. Then we have a homeomorphism
\begin{equation}\label{eq:homeo_giving_cell_decomp}
\xi_\ba\colon S_\ba \xrightarrow\cong (\mathbb S^2)^{\lfloor\frac{k}{2}\rfloor}\,,\,\,(x_1,\ldots,x_m) \mapsto (y_1,\ldots,y_{\lfloor\frac{k}{2}\rfloor}),
\end{equation}
where
$y_j=x_{i_j}$, if $i_j$ is an endpoint of an unmarked cup, and $y_j=-x_{i_j}$ if $i_j$ is an endpoint of a marked cup, $1\leq j\leq\lfloor\frac{k}{2}\rfloor$.

We equip $S_\ba$ with the structure of a CW-complex, where the cells are obtained as the preimages of the cells of $(\mathbb S^2)^{\lfloor\frac{k}{2}\rfloor}$  under the homeomorphism (\ref{eq:homeo_giving_cell_decomp}). We obtain a one-to-one correspondence between $\widetilde{\mathbb B}^{2m-k,k}_\mathrm{KL}$ and the cells of $S_\ba$ by sending an enriched cup diagram $M$ to the cell $C_M$, where a cup with a dot (resp.\ a cup without a dot) means that we have chosen the $0$-cell (resp.\ $2$-cell) for the corresponding sphere via (\ref{eq:homeo_giving_cell_decomp}). Since the basis elements of the homology $H_*(S_\ba)$ are in one-to-one correspondence with the homology classes of the cells of $S_\ba$ we obtain an isomorphism of vector spaces
\begin{equation} \label{eq:homology_identification_for_S_ba}
\mathbb C[\widetilde{\mathbb B}^{2m-k,k}_\mathrm{KL}] \xrightarrow\cong H_*(S_\ba)\,,\,\, M\mapsto [C_M]. 
\end{equation}
The homological degree of an enriched cup diagram is given as twice the number of cups without a dot.

\begin{rem}
The white dots in this article play exactly the same role as the black dots on the cup diagrams in the related work in type $A$, \cite{RT11}, \cite{Rus11}. In order to distinguish them from the black markers (which do not appear in type $A$) we chose to color them white. 
\end{rem}

\subsection{A combinatorial description of $\psi_{2m-k,k}$}

The next step is to explicitly describe the map $\psi_{2m-k,k}$ in~(\ref{eq:key_comm_diagram}) in terms of the diagrammatic bases we have chosen in the previous subsection.

\begin{defi}
Given an enriched cup diagram $M\in\widetilde{\mathbb B}^{2m-k,k}_\mathrm{KL}$, we define the associated {\it line diagram sum} as 
\[ 
L_M=\sum_{U\in\mathcal U_M}(-1)^{\Lambda_M(U)}l_U \in H_*\left((\mathbb S^2)^m\right),
\]
where $\mathcal U_M$ is the set of all subsets $U\subseteq\{1,\ldots,m\}$ containing precisely one endpoint of every cup without a dot, and $\Lambda_M(U)$ counts the total number of all endpoints of cups with a marker and no dot, plus the number of right endpoints of cups in $U$ with neither a marker nor a dot.
\end{defi}

Similar line diagram sums (for type $A$) appear in \cite[Definition 3.11]{RT11}. Their significance lies in the following result.  

\begin{prop} \label{prop:image_of_homology_gen}
Via identifications (\ref{eq:homology_identification_for_spheres}) and (\ref{eq:homology_identification_for_S_ba}), the map $(\psi_\ba)_*\colon H_*(S_\ba)\to H_*((\mathbb S^2)^m)$ induced by the natural inclusion $\psi_\ba\colon S_\ba\hookrightarrow (\mathbb S^2)^m$ is explicitly given by the assignment $M\mapsto L_M$.
\end{prop}

\begin{ex}
If $m=1$ the maps induced on homology by the natural inclusion $S_{\begin{tikzpicture}[scale=.5]
\draw[thick] (0,0) -- +(0,-.4);
\fill ([xshift=-2.5pt,yshift=-2.5pt]0,-.2) rectangle ++(5pt,5pt);
\end{tikzpicture}}\hookrightarrow\mathbb S^2$ respectively by the inclusion $S_{\begin{tikzpicture}[scale=.5]
\draw[thick] (0,0) -- +(0,-.4);
\end{tikzpicture}}\hookrightarrow\mathbb S^2$ are given by 
\begin{eqnarray*}
\begin{tikzpicture}[baseline={(0,-.4)},scale=.8]
\draw[thick] (0,0) -- +(0,-.8);
\fill ([xshift=-2.5pt,yshift=-2.5pt]0,-.25) rectangle ++(5pt,5pt);
\draw (0,-.55) circle(3.5pt);
\end{tikzpicture}
\,
\mapsto
\,
\begin{tikzpicture}[baseline={(0,-.4)},scale=.8]
\draw[thick] (0,0) -- +(0,-.8);
\draw (0,-.4) circle(3.5pt);
\end{tikzpicture}
&\quad\text{respectively}\quad&
\begin{tikzpicture}[baseline={(0,-.4)},scale=.8]
\draw[thick] (0,0) -- +(0,-.8);
\draw (0,-.4) circle(3.5pt);
\end{tikzpicture}
\,
\mapsto
\,
\begin{tikzpicture}[baseline={(0,-.4)},scale=.8]
\draw[thick] (0,0) -- +(0,-.8);
\draw (0,-.4) circle(3.5pt);
\end{tikzpicture}
\end{eqnarray*}
since the first map sends $p$ to the $0$-cell of $\mathbb S^2$, and the second  is clearly homotopic to the map $S_{\begin{tikzpicture}[scale=.5]
\draw[thick] (0,0) -- +(0,-.4);
\end{tikzpicture}}\to\mathbb S^2$, $-p\mapsto p$, respecting the $0$-skeleta. More generally, the map induced by $S_\ba\hookrightarrow(\mathbb S^2)^m$, where $\ba$ consists of $m$ subsequent rays, sends the generator corresponding to the point $S_\ba$ to the line diagram consisting of $m$ subsequent rays with a dot. Furthermore, by arguing similarly as in \cite[Lemma~3.2]{RT11}, one can show that the maps induced by the inclusions $S_{\begin{tikzpicture}[scale=.5]
\draw[thick] (.5,0) .. controls +(0,-.5) and +(0,-.5) .. +(.5,0);
\fill ([xshift=-2.5pt,yshift=-2.5pt].75,-.365) rectangle ++(5pt,5pt);
\end{tikzpicture}}\hookrightarrow\mathbb S^2\times\mathbb S^2$ respectively $S_{\begin{tikzpicture}[scale=.5]
\draw[thick] (.45,0) .. controls +(0,-.45) and +(0,-.45) .. +(.45,0);
\end{tikzpicture}}\hookrightarrow(\mathbb S^2)^2$ are given by
\begin{eqnarray} \label{eq:induced_map_cup_1}
\quad \quad
\begin{tikzpicture}[baseline={(0,-.2)},scale=.9]
\draw[thick] (0,0) .. controls +(0,-.5) and +(0,-.5) .. +(.5,0);
\draw (0.4,-.3) circle(3pt);
\fill ([xshift=-2.5pt,yshift=-2.5pt].1,-.3) rectangle ++(5pt,5pt);
\end{tikzpicture}
\,
\mapsto
\,
\begin{tikzpicture}[baseline={(0,-.4)},scale=.8]
\draw[thick] (0,0) -- +(0,-.8);
\draw[thick] (.5,0) -- +(0,-.8);
\draw (0,-.4) circle(3pt);
\draw (.5,-.4) circle(3pt);
\end{tikzpicture}
\hspace{3.5em}
\begin{tikzpicture}[baseline={(0,-.3)},scale=.9]
\draw[thick] (0,0) .. controls +(0,-.5) and +(0,-.5) .. +(.5,0);
\fill ([xshift=-2.5pt,yshift=-2.5pt].25,-.365) rectangle ++(5pt,5pt);
\end{tikzpicture}
\,
\mapsto
\,
-\,
\begin{tikzpicture}[baseline={(0,-.4)},scale=.8]
\draw[thick] (0,0) -- +(0,-.8);
\draw[thick] (.5,0) -- +(0,-.8);
\draw (.5,-.4) circle(3pt);
\end{tikzpicture}
-
\begin{tikzpicture}[baseline={(0,-.4)},scale=.8]
\draw[thick] (0,0) -- +(0,-.8);
\draw[thick] (.5,0) -- +(0,-.8);
\draw (0,-.4) circle(3pt);
\end{tikzpicture}
&\text{respectively}&
\begin{tikzpicture}[baseline={(0,-.3)},scale=.9]
\draw[thick] (0,0) .. controls +(0,-.5) and +(0,-.5) .. +(.5,0);
\draw (0.25,-.365) circle(3pt);
\end{tikzpicture}
\,
\mapsto
\,
\begin{tikzpicture}[baseline={(0,-.4)},scale=.8]
\draw[thick] (0,0) -- +(0,-.8);
\draw[thick] (.5,0) -- +(0,-.8);
\draw (0,-.4) circle(3pt);
\draw (0.5,-.4) circle(3pt);
\end{tikzpicture}
\hspace{3.5em}
\begin{tikzpicture}[baseline={(0,-.25)},scale=.9]
\draw[thick] (0,0) .. controls +(0,-.5) and +(0,-.5) .. +(.5,0);
\end{tikzpicture}
\,
\mapsto
\,
\begin{tikzpicture}[baseline={(0,-.4)},scale=.8]
\draw[thick] (0,0) -- +(0,-.8);
\draw[thick] (.5,0) -- +(0,-.8);
\draw (0.5,-.4) circle(3pt);
\end{tikzpicture}
-
\begin{tikzpicture}[baseline={(0,-.4)},scale=.8]
\draw[thick] (0,0) -- +(0,-.8);
\draw[thick] (.5,0) -- +(0,-.8);
\draw (0,-.4) circle(3pt);
\end{tikzpicture}
\end{eqnarray}
which shows that Proposition~\ref{prop:image_of_homology_gen} is true for all cup diagrams with one connected component. 
\end{ex}

Before we prove Proposition~\ref{prop:image_of_homology_gen} in general, we note the following useful combinatorial fact whose (straightforward) proof is omitted (cf.\ \cite[Lemma~3.11]{RT11} for a similar statement).

\begin{lem}
Suppose that the enriched cup diagram $M'$ is obtained from $M$ by deleting a cup connecting vertices $i$ and $j$, $i<j$. For $U\in\mathcal U_{M'}$ and $S\subseteq\{i,j\}$ we define $\widetilde{U}_S\subseteq\{1,\ldots,m\}$ as 
\[
\widetilde{U}_S=S\cup\{x\in U\mid x<i\}\cup\{x+1\mid x\in U \text{ and } i\leq x<j\}\cup\{x+2\mid x\in U \text{ and }j\leq x\}.
\]
Then we can write $L_M=\sum_{U\in\mathcal U_{M'}}(-1)^{\Lambda_{M'}(U)}l_{\widetilde{U_\emptyset}}$, if the cup connecting vertices $i$ and $j$ has a dot. If the cup does not have a dot, then we have
\[
L_M=(-1)^\sigma \sum_{U\in\mathcal U_{M'}}(-1)^{\Lambda_{M'}(U)}l_{\widetilde{U_i}} - \sum_{U\in\mathcal U_{M'}}(-1)^{\Lambda_{M'}(U)}l_{\widetilde{U_j}},
\]
where $\sigma=1$, if the cup has a marker, and $\sigma=0$, if it has neither a marker nor a dot. 
\end{lem}


\begin{proof}[Proof of Proposition~\ref{prop:image_of_homology_gen}.]
We begin by using the small examples above in order to deduce the claim for the cup diagram $\ba$ which consists of $\lfloor\frac{k}{2}\rfloor$ subsequent cups connecting neighboring vertices with (possibly) an additional collection of subsequent rays to the right of the cups. Assume that the cup connecting the first two vertices is unmarked and identify $S_\ba$ with $S_{\begin{tikzpicture}[scale=.5]
\draw[thick] (.5,0) .. controls +(0,-.5) and +(0,-.5) .. +(.5,0);
\end{tikzpicture}}\times S_{\ba'}$ and the inclusion $S_\ba\hookrightarrow (\mathbb S^2)^m$ with $\psi_{\begin{tikzpicture}[scale=.5]
\draw[thick] (.5,0) .. controls +(0,-.5) and +(0,-.5) .. +(.5,0);
\end{tikzpicture}}\times\psi_{\ba'}$. Here $\ba'$ is the cup diagram obtained by removing the leftmost cup. Via the K\"unneth isomorphism, the map induced by $\psi_{\begin{tikzpicture}[scale=.5]
\draw[thick] (.5,0) .. controls +(0,-.5) and +(0,-.5) .. +(.5,0);
\end{tikzpicture}}\times\psi_{\ba'}$ is the tensor product $(\psi_{\begin{tikzpicture}[scale=.5]
\draw[thick] (.5,0) .. controls +(0,-.5) and +(0,-.5) .. +(.5,0);
\end{tikzpicture}})_*\otimes(\psi_{\ba'})_*$. Hence, if the cup connecting the two leftmost vertices in $M$ has a dot, we compute
\[
(\psi_{\begin{tikzpicture}[scale=.5]
\draw[thick] (.5,0) .. controls +(0,-.5) and +(0,-.5) .. +(.5,0);
\end{tikzpicture}})_*(\begin{tikzpicture}[scale=.5]
\draw[thick] (.5,0) .. controls +(0,-.5) and +(0,-.5) .. +(.5,0);
\draw (0.75,-.365) circle(3.5pt);
\end{tikzpicture})\otimes(\psi_{\ba'})_*(M')=l_\emptyset\otimes\left(\sum_{U\in\mathcal U_{M'}}(-1)^{\Lambda_{M'}(U)}l_U\right)=\sum_{U\in\mathcal U_{M'}}(-1)^{\Lambda_{M'}(U)}l_\emptyset\otimes l_U,
\]     
where the first equality follows by induction and the small cases above. Here $M'$ denotes the enriched cup diagram obtained by removing the leftmost cup. Applying the K\"unneth isomorphism to the above equation yields the claim. 

On the other hand, if the leftmost cup in $M$ does not have a dot, we compute
\begin{align*}
(\psi_{\begin{tikzpicture}[scale=.5]
\draw[thick] (.5,0) .. controls +(0,-.5) and +(0,-.5) .. +(.5,0);
\end{tikzpicture}})_*(\begin{tikzpicture}[scale=.5]
\draw[thick] (.5,0) .. controls +(0,-.5) and +(0,-.5) .. +(.5,0);
\end{tikzpicture})\otimes(\psi_{\ba'})_*(M') &=\left(l_{\{1\}}-l_{\{2\}}\right)\otimes\left(\sum_{U\in\mathcal U_{M'}}(-1)^{\Lambda_{M'}(U)}l_U\right) \\
																					  &=\sum_{U\in\mathcal U_{M'}}(-1)^{\Lambda_{M'}(U)}l_{\{1\}}\otimes l_U - \sum_{U\in\mathcal U_{M'}}(-1)^{\Lambda_{M'}(U)}l_{\{2\}}\otimes l_U.
\end{align*} 
If the leftmost cup has a marker we argue similarly. 

For the general case let $\ba$ be a cup diagram and $\ba_0$ the cup diagram obtained by rearranging the cups in such a way that the diagram is completely unnested. Let $\tau_\ba$ be the permutation of the vertices which realizes the change of passing from $\ba$ to $\ba'$. This induces a homeomorphism $\tau_\ba\colon(\mathbb S^2)^m\to(\mathbb S^2)^m$ which permutes the coordinates accordingly and hence restricts to a homeomorphism $S_\ba\to S_{\ba_0}$. The inclusion $S_\ba\hookrightarrow(\mathbb S^2)^m$ can be written as the composition $\tau_\ba\circ\psi_{\ba_0}\circ\tau^{-1}_\ba\vert_{S_\ba}$. Let $M_0$ be the cup diagram obtained from $\ba_0$ whose $i$th component has a dot if and only if the $i$th component of $M$ has a dot. Then the claim follows from 
\[
\tau_\ba(\psi_{\ba_0}(\tau^{-1}_\ba\vert_{S_\ba}(M)))=\tau_\ba(\psi_{\ba_0}(M_0))=\tau_\ba(\sum_{U\in\mathcal U_{M_0}}(-1)^{\Lambda_{M_0}(U)}l_U)=\sum_{U\in\mathcal U_M}(-1)^{\Lambda_M(U)}l_U.\qedhere
\]      
\end{proof}


\subsection{A linearly independent set of line diagram sums}

In this subsection we introduce a subset of $\widetilde{\mathbb B}^{2m-k,k}_\mathrm{KL}$ having the property that the images of the elements from this set under $\psi_{2m-k,k}$ are linearly independent. This will be an important technical tool in obtaining a diagrammatic basis of $H_*(\SOfiber)$ in the next subsection.

\begin{defi}
A {\it standard enriched cup diagram}, or shorter {\it standard cup diagram}, is an enriched cup diagram in which all cups marked with a dot can be connected to the right side of the rectangle 
by a path which does neither intersect the rest of the diagram nor any given path connecting a marker with the left side. The set of all standard cup diagrams on $m$ vertices with $\lfloor\frac{k}{2}\rfloor$ cups such that the number of marked rays plus unmarked cups is even is denoted by $s\widetilde{\mathbb B}^{2m-k,k}_\mathrm{KL}$.   
\end{defi}



\begin{ex} \label{ex:standard_cup_diagrams}
Here are the standard cup diagrams contained in $s\widetilde{\mathbb B}^{3,3}_\mathrm{KL}$:
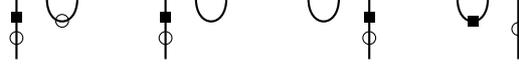
\[
\begin{tikzpicture}[scale=.8]
\begin{scope}[xshift=3cm]
\draw[thick] (.5,0) .. controls +(0,-.5) and +(0,-.5) .. +(.5,0);
\draw (.75,-.365) circle(3pt);

\draw[thick] (0,0) -- +(0,-1);
\fill ([xshift=-2.5pt,yshift=-2.5pt]0,-.3) rectangle ++(5pt,5pt);
\draw (0,-.65) circle(3pt);
\end{scope}
\end{tikzpicture}
\hspace{3em}
\begin{tikzpicture}[scale=.8]
\begin{scope}[xshift=3cm]
\draw[thick] (.5,0) .. controls +(0,-.5) and +(0,-.5) .. +(.5,0);

\draw[thick] (0,0) -- +(0,-1);
\fill ([xshift=-2.5pt,yshift=-2.5pt]0,-.3) rectangle ++(5pt,5pt);
\draw (0,-.65) circle(3pt);
\end{scope}
\end{tikzpicture}
\hspace{3em}
\begin{tikzpicture}[scale=.8]
\begin{scope}[xshift=3cm]
\draw[thick] (0,0) .. controls +(0,-.5) and +(0,-.5) .. +(.5,0);

\draw[thick] (1,0) -- +(0,-1);
\fill ([xshift=-2.5pt,yshift=-2.5pt]1,-.3) rectangle ++(5pt,5pt);
\draw (1,-.65) circle(3pt);
\end{scope}
\end{tikzpicture}
\hspace{3em}
\begin{tikzpicture}[scale=.8]
\begin{scope}[xshift=3cm]
\draw[thick] (0,0) .. controls +(0,-.5) and +(0,-.5) .. +(.5,0);

\draw[thick] (1,0) -- +(0,-1);
\draw (1,-.5) circle(3pt);

\fill ([xshift=-2.5pt,yshift=-2.5pt].25,-.365) rectangle ++(5pt,5pt);
\end{scope}
\end{tikzpicture}
\]
\end{ex}

Let $C_\mathrm{KL}^{\leq\lfloor\frac{k}{2}\rfloor}(m)$ denote the subset of $C_\mathrm{KL}(m)$ consisting of all cup diagrams with at most $\lfloor\frac{k}{2}\rfloor$ cups. If $m=k$ we have $C_\mathrm{KL}^{\leq\lfloor\frac{m}{2}\rfloor}(m)=C_\mathrm{KL}(m)$.

We define a map $\mathsf{C}^{2m-k,k}\colon s\widetilde{\mathbb B}^{2m-k,k}_\mathrm{KL}\to C_\mathrm{KL}^{\leq\lfloor\frac{k}{2}\rfloor}(m)$, called {\it cutting}, which takes a standard cup diagram $M\in s\widetilde{\mathbb B}^{2m-k,k}_\mathrm{KL}$ and replaces each dotted cup by two unmarked rays (this can be thought of as cutting the cup along the dot). Furthermore, the dots on the rays of $M$ are deleted (again, the reader might think of this as cutting along the dots and throw away the resulting line segment which is not connected to the vertex). If the number of marked rays plus the number of unmarked cups is not even in the resulting diagram we additionally either delete or put a marker on the leftmost ray.   

The map $\mathsf{G}^{2m-k,k}\colon C_\mathrm{KL}^{\leq\lfloor\frac{k}{2}\rfloor}(m)\to s\widetilde{\mathbb B}^{2m-k,k}_\mathrm{KL}$, called {\it gluing}, is defined as follows: given a cup diagram $\ba\in C_\mathrm{KL}^{\leq\lfloor\frac{k}{2}\rfloor}(m)$ we replace the two rightmost rays by a cup with a dot (this can be thought of as gluing the two endpoints of the rays resulting in a cup in which the gluing point is decorated with a dot). The resulting cup has a marker if and only if one of the rays from which it resulted had a marker. Take the new diagram and repeat the procedure until it contains exactly $\lfloor\frac{k}{2}\rfloor$ cups. 






\begin{ex}
The standard cup diagrams in Example~\ref{ex:standard_cup_diagrams} are obtained by applying $\mathsf{G}^{3,3}$ to the following diagrams in $C_{\mathrm{KL}}(3)$ (in the same order) respectively. 
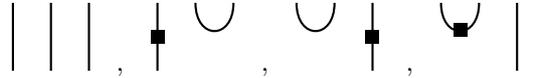
\[
\begin{array}{ccccccc}
\begin{tikzpicture}[baseline={(0,-.9)}]
\draw[thick] (0,0) -- +(0,-.9);
\draw[thick] (.5,0) -- +(0,-.9);
\draw[thick] (1,0) -- +(0,-.9);
\end{tikzpicture}
&,&
\begin{tikzpicture}[baseline={(0,-.9)}]
\draw[thick] (.5,0) .. controls +(0,-.5) and +(0,-.5) .. +(.5,0);

\draw[thick] (0,0) -- +(0,-.9);
\fill ([xshift=-2.5pt,yshift=-2.5pt]0,-.45) rectangle ++(5pt,5pt);
\end{tikzpicture}
&,&
\begin{tikzpicture}[baseline={(0,-.9)}]
\draw[thick] (0,0) .. controls +(0,-.5) and +(0,-.5) .. +(.5,0);

\draw[thick] (1,0) -- +(0,-.9);
\fill ([xshift=-2.5pt,yshift=-2.5pt]1,-.45) rectangle ++(5pt,5pt);
\end{tikzpicture}
&,&
\begin{tikzpicture}[baseline={(0,-.9)}]
\draw[thick] (0,0) .. controls +(0,-.5) and +(0,-.5) .. +(.5,0);

\draw[thick] (1,0) -- +(0,-.9);

\fill ([xshift=-2.5pt,yshift=-2.5pt].25,-.365) rectangle ++(5pt,5pt);
\end{tikzpicture}\hspace{.5em}.
\end{array}
\]
\end{ex}

\begin{lem} \label{lem:comb_bij} 
The assignments $\mathsf{G}^{2m-k,k}$ and $\mathsf{C}^{2m-k,k}$ define mutually inverse maps between 
\[
s\widetilde{\mathbb B}^{2m-k,k}_\mathrm{KL}\quad\text{and}\quad C_\mathrm{KL}^{\leq\lfloor\frac{k}{2}\rfloor}(m).
\]
\end{lem}
\begin{proof}
This is evident from the definitions. 
\end{proof}

 
\begin{prop} \label{prop:technical_prop_1}
The line diagram sums $L_M$, i.e.\ the images of $M$ under $\psi_{m,m}$, where $M$ varies over all standard cup diagrams in $s\widetilde{\mathbb B}^{m,m}_\mathrm{KL}$, are linearly independent in $H_*((\mathbb S^2)^m)$. 
\end{prop}

For the proof of this proposition we use the following technical lemma.

\begin{lem} \label{lem:lin_indep}
Let $M\in s\widetilde{\mathbb B}^{m,m}_\mathrm{KL}$ be standard cup diagram without markers. Then the line diagram sums $L_M$, where $M$ varies over all standard cup diagrams in $s\widetilde{\mathbb B}^{m,m}_\mathrm{KL}$ which can be obtained by decorating $M$ with markers, are linearly independent.
\end{lem}


\begin{proof}
We prove the statement by induction on the total number of components of $M$ which can be marked with a marker. If there is precisely one decorable component in $M$, then there exists precisely one cup diagram with the claimed properties and the claim follows. In case that $M$ consists of a single cup without a dot we note that all the possible decorations with markers (without fixing the parity) are linearly independent. 

So suppose there is more than one component in $M$ which is allowed to have a marker. In particular, since $M$ is standard, there exists a cup in $M$ without a dot which may be marked with a marker. In the following we fix such a cup connecting $i$ and $j$. Let $M=M_1,N_1,M_2,N_2,\ldots,M_r,N_r$ be the standard cup diagrams obtained from all the possible decorations of $M$ with markers and a fixed parity of markers. The diagrams are listed in such a way that $M_i$ and $N_i$ are the same except that $M_i$ does not have a marker on the fixed cup, but a marker on the leftmost component with a dot. On the other hand, $N_i$ has a marker on the fixed cup but no marker on the leftmost component with a dot. In case that $M$ has no dots we simply consider all standard cup diagrams $M_1,N_1,\ldots,M_r,N_r$ (not only the ones with a fixed parity of markers) and drop the condition on the leftmost component with a dot. Let $M^\prime_1,N^\prime_1,\ldots,M^\prime_r,N^\prime_r$ be the cup diagrams obtained by deleting the cup connecting $i$ and $j$. We compute
\begin{align*}
0 &= \sum_{l=1}^r \lambda_lL_{M_l} + \sum_{l=1}^r \mu_lL_{N_l} \\
	&= \sum_{l=1}^r \lambda_l\left(-\sum_{U\in\mathcal U_{M^\prime_l}}(-1)^{\Lambda_{M^\prime_l}(U)}l_{\widetilde{U_j}}+\sum_{U\in\mathcal U_{M^\prime_l}}(-1)^{\Lambda_{M^\prime_l}(U)}l_{\widetilde{U_i}}\right)\\
	&\hspace{1em} + \sum_{l=1}^r \mu_l\left(-\sum_{U\in\mathcal U_{N^\prime_l}}(-1)^{\Lambda_{N^\prime_l}(U)}l_{\widetilde{U_j}}-\sum_{U\in\mathcal U_{N^\prime_l}}(-1)^{\Lambda_{N^\prime_l}(U)}l_{\widetilde{U_i}}\right) \\
	&= \sum_{U\in\mathcal U_{M^\prime}} \left(\sum_{l=1}^r(-\mu_l-\lambda_l)(-1)^{\Lambda_{M^\prime_l}(U)}\right)l_{\widetilde{U_j}} + \sum_{U\in\mathcal U_{M^\prime}} \left(\sum_{l=1}^r(-\mu_l+\lambda_l)(-1)^{\Lambda_{M^\prime_l}(U)}\right)l_{\widetilde{U_i}}. 
\end{align*}  
Since the family of vectors $l_{\widetilde{U}_i}, l_{\widetilde{U}_j}$ is linearly independent if $U$ varies over all elements in $\mathcal U_{M^\prime}$, we deduce the equations 
\begin{equation} \label{eq:equations_lin_independence}
\sum_{l=1}^r(-\mu_l-\lambda_l)(-1)^{\Lambda_{M^\prime_l}(U)}=0 \hspace{1.6em}\text{and}\hspace{1.6em} \sum_{l=1}^r(-\mu_l+\lambda_l)(-1)^{\Lambda_{M^\prime_l}(U)}=0
\end{equation}
for all $U\in\mathcal U_{M^\prime}$. Using the left equation in (\ref{eq:equations_lin_independence}) we compute
\[
0 = \sum_{U\in\mathcal U_{M^\prime}}\left(\sum_{l=1}^r(-\mu_l-\lambda_l)(-1)^{\Lambda_{M^\prime_l}(U)}\right)l_U = \sum_{l=1}^r(-\mu_l-\lambda_l)\left(\sum_{U\in\mathcal U_{M^\prime}}(-1)^{\Lambda_{M^\prime_l}(U)}l_U\right) ,
\]
which equals $\sum_{l=1}^r(-\mu_l-\lambda_l)L_{M^\prime_l}$ and thus implies $\lambda_l=-\mu_l$ for all $l\in\{1,\ldots,r\}$ because the $L_{M^\prime_l}$ are linearly independent by induction. Similarly we  obtain $\lambda_l=\mu_l$ by repeating the calculation with the right equation in (\ref{eq:equations_lin_independence}). Hence, we deduce that $\lambda_l=\mu_l=0$ for all $l\in\{1,\ldots,r\}$.  
\end{proof}

\begin{proof}[Proof of Proposition~\ref{prop:technical_prop_1}]
It suffices to prove that the elements $L_M$, where $M$ varies over all standard cup diagrams in $s\widetilde{\mathbb B}^{m,m}_{\mathrm{KL},l}$ with precisely $l$ cups without a dot, are linearly independent in $H_{2l}((\mathbb S^2)^m)$, $0\leq l\leq m$. We define a total order on the subsets of $\{1,\ldots,m\}$ of cardinality $l$:
\begin{equation} \label{eq:total_order}
\{i_1<\cdots<i_l\}<\{i'_1<\cdots<i'_l\} :\Leftrightarrow \exists r\colon i_r<i_r' \text{ and }i_{r+1}=i'_{r+1},\ldots,i_l=i'_l
\end{equation}
which induces a total order on the line diagrams $l_U$, where $U\subseteq\{1,\ldots,m\}$, $|U|=l$, i.e.\ an order on our basis of $H_{2l}((\mathbb S^2)^m)$. Define $f\colon s\widetilde{\mathbb B}^{m,m}_{\mathrm{KL},l}\to s\widetilde{\mathbb B}^{m,m}_{\mathrm{unmar},l}$ as the map which forgets all markers.

Given a standard cup diagram $M$, we write $U_M\in\mathcal U_M$ to denote the set containing all right endpoints of the cups without a dot. Note that $U_M$ is maximal in $\mathcal U_M$ with respect to the order (\ref{eq:total_order}). Assume that, for some $\lambda_M\in\mathbb C$,
\begin{equation} \label{eq:lin_indep_ansatz}
0=\sum_{M\in s\widetilde{\mathbb B}^{m,m}_\mathrm{KL}} \lambda_ML_M = \sum_{N\in s\widetilde{\mathbb B}^{m,m}_\mathrm{unmar}}\sum_{\begin{subarray}{l}
M\in s\widetilde{\mathbb B}^{m,m}_\mathrm{KL}\\ f(M)=N\end{subarray}}\lambda_ML_M. 
\end{equation}
 
Fix $N_\text{max}$ such that $U_{N_\text{max}}$ is maximal amongst all $N\in s\widetilde{\mathbb B}^{m,m}_\mathrm{unmar}$. Since the cup diagrams without markers are determined by the right endpoints of cups there is a unique such $N$. 

Note that the basis vector $l_{U_M}$ occurs (with non-zero coefficient) in each line diagram sum $L_{M^\prime}$ with $f(M^\prime)=M$ but it does not occur in any $L_{M^\prime}$, where $f(M^\prime)\neq M$ because these $L_{M^\prime}$ only contain basis vectors which are strictly smaller with respect to the total order (this follows from our maximality assumption on $M$ and the fact that $U_M$ is maximal in $\mathcal U_M$ for every $M$).

Hence, equation (\ref{eq:lin_indep_ansatz}) decouples into two independent equations    
\[
\sum_{\begin{subarray}{l}
M\in s\widetilde{\mathbb B}^{m,m}_\mathrm{KL}\\ f(M)=M'_\text{max}\end{subarray}}\lambda_ML_M = 0 \hspace{2.5em}\text{and} \hspace{2.5em}
\sum_{\begin{subarray}{l}
M'\in s\widetilde{\mathbb B}^{m,m}_\mathrm{unmar}\\ M'\neq M'_\text{max}\end{subarray}}\sum_{\begin{subarray}{l}
M\in s\widetilde{\mathbb B}^{m,m}_\mathrm{KL}\\ f(M)=M'\end{subarray}}\lambda_ML_M = 0.
\]
By iterating the above argument we obtain equations 
$
\sum_{\begin{subarray}{l}
M\in s\widetilde{\mathbb B}^{m,m}_\mathrm{KL}\\ f(M)=N\end{subarray}}\lambda_ML_M = 0,
$
for all $N\in s\widetilde{\mathbb B}^{m,m}_\mathrm{unmar}$. Hence, Proposition~\ref{prop:technical_prop_1} follows from Lemma~\ref{lem:lin_indep}.       
\end{proof}

\subsection{An explicit homology basis}

In this subsection we use the results obtained so far to construct a basis of the homology of $\SOfiber$ and give an explicit description of the map $\gamma_{2m-k,k}$.

\begin{prop} \label{prop:basis_general_case}
The elements $\phi_{2m-k,k}(\mathsf{G}^{2m-k,k}(\ba))$, where $\ba\in C_\mathrm{KL}^{\leq\lfloor\frac{k}{2}\rfloor}(m)$, form a basis of the homology $H_*(\SOfiber)$. 
Moreover, we have a commutative diagram
\[
\begin{xy}
	\xymatrix{
H_{<k}(\SOfiberequalsize) \ar@^{(->}[dr]^{\gamma_{m,m}} \ar[d]^\cong_{\Gamma^{2m-k,k}} \\ 
H_*(\SOfiber) \ar@^{(->}[r]^{\gamma_{2m-k,k}} & H_*((\mathbb S^2)^m) 
	}
\end{xy}
\]
where $\Gamma^{2m-k,k}$ is the isomorphism of vector spaces sending a basis element $\phi_{m,m}\left(\mathsf{G}^{2m-k,k}(\ba)\right)$ of $H_{<k}(\SOfiberequalsize)$ to $\phi_{2m-k,k}(\mathsf{G}^{2m-k,k}(\ba))\in H_*(\SOfiber)$.  
\end{prop}

\begin{rem}
According to Proposition~\ref{prop:basis_general_case} above, we can identify $\mathbb C[C_\mathrm{KL}^{\leq\lfloor\frac{k}{2}\rfloor}(m)]$ with $H_*(\SOfiber)$ via the following isomorphism of vector spaces, 
\begin{equation} \label{eq:combinatorial_springer_homology}
\mathbb C[C_\mathrm{KL}^{\leq\lfloor\frac{k}{2}\rfloor}(m)]\xrightarrow\cong H_*(\SOfiber)\,,\,\,\ba\mapsto\phi_{2m-k,k}\left(\mathsf{G}^{2m-k,k}(\ba)\right),
\end{equation}
thereby obtaining a diagrammatic basis of the homology. Under this identification the cup diagrams with $l$ cups form a basis of $H_{2l}(\SOfiber)$, $0\leq l\leq\lfloor\frac{k}{2}\rfloor$.
\end{rem}

\begin{lem} \label{lem:iso_onto_image}
The map $\gamma_{m,m}\colon H_*(\SOfiberequalsize) \to H_*((\mathbb S^2)^m)$ induced by the inclusion $\SOfiberequalsize\hookrightarrow(\mathbb S^2)^m$ is injective.
\end{lem}
\begin{proof}
By considering the commutative diagram (\ref{eq:key_comm_diagram}) we deduce $\im\left(\gamma_{m,m}\right)=\im\left(\psi_{m,m}\right)$, because $\phi_{m,m}$ is surjective (cf.\ \cite[p.\ 23]{ES12}). Since $\dim\left(\im\left(\psi_m\right)\right)\geq\vert s\widetilde{\mathbb B}^{m,m}_\mathrm{KL}\vert=2^{m-1}$ (the inequality follows from Proposition~\ref{prop:technical_prop_1} and the equality by combining Lemma~\ref{lem:comb_bij} with the fact that $\vert C_\mathrm{KL}(m)\vert=2^{m-1}$ which will be shown in Subsection~\ref{sec:diagrammatic_Hecke_module}) we deduce that $\dim\left(\im\left(\gamma_{m,m}\right)\right)=\dim\left(\im\left(\psi_{m,m}\right)\right)\geq 2^{m-1}$. Hence, since $\dim H_*(\SOfiberequalsize)=2^{m-1}$ by Remark~\ref{rem:dim_homology_equal-row}, the rank nullity theorem implies $\dim\left(\im\left(\gamma_{m,m}\right)\right)=2^{m-1}$ and $\ker(\gamma_{m,m})=\{0\}$.
\end{proof}

\begin{lem}
The set $\phi_{m,m}\left(\mathsf{G}^{m,m}(\ba)\right)$, where $\ba\in C_\mathrm{KL}(m)$, is a basis of $H_*(\SOfiberequalsize)$. 
\end{lem}
\begin{proof}
The commutative diagram (\ref{eq:key_comm_diagram}) combined with Lemma~\ref{lem:iso_onto_image} gives a commutative diagram
\begin{eqnarray}
\label{diagProp24}
\begin{xy}
	\xymatrix{
		\bigoplus\limits_{\ba\in\mathbb B^{m,m}_\mathrm{KL}} H_*(S_\ba) \ar[r]^{\phi_{m,m}} \ar@/^1.5pc/[rr]^{\psi_{m,m}} & H_*(\SOfiberequalsize) \ar[r]^{\gamma_{m,m}}_\cong	& \im(\psi_{m,m})\subseteq H_*((\mathbb S^2)^m),
	}
\end{xy}
\end{eqnarray}
which allows us to write $\phi\left(\mathsf{G}^{m,m}(\ba)\right)=\gamma^{-1}_{m,m}\left(\psi_{m,m}\left(\mathsf{G}^{m,m}(\ba)\right)\right)$ for all  cup diagrams $\ba$ on $m$ vertices. Since we have $\{\mathsf{G}^{m,m}(\ba)\mid\ba\in C_\mathrm{KL}(m)\}=s\widetilde{\mathbb B}^{m,m}_\mathrm{KL}$, it follows from Proposition~\ref{prop:technical_prop_1} that the collection of elements $\psi_{m,m}(\mathsf{G}^{m,m}(\ba))\in H_*((\mathbb S^2)^m)$, where $\ba$ varies over all cup diagrams in $C_\mathrm{KL}(m)$, is linearly independent. This remains true after applying the linear isomorphism $\gamma_{m,m}^{-1}$.  

As $\vert\{\mathsf{G}^{m,m}(\ba)\mid\ba\in C_\mathrm{KL}(m)\}\vert = 2^{m-1} = \dim H_*(\SOfiberequalsize)$ we indeed have a basis of $H_*(\SOfiberequalsize)$. 
\end{proof}

\begin{proof}[Proof of Proposition~\ref{prop:basis_general_case}]
From the appendix we deduce that the dimension of $H^*(\SOfiber)$ (and hence also of its homology by duality) is given by counting the number of all cup diagrams in $C_\mathrm{KL}^{\leq\lfloor\frac{k}{2}\rfloor}(m)$. Thus, in order to check that the elements $\phi_{2m-k,k}(\mathsf{G}^{2m-k,k}(\ba))$ form a basis, where $\ba$ varies over all cup diagrams in $C_\mathrm{KL}^{\leq\lfloor\frac{k}{2}\rfloor}(m)$, it suffices to see that they are linearly independent. Under the assumption that they are not linearly independent, it follows that the vectors $\gamma_{2m-k,k}\left(\phi_{2m-k,k}(\mathsf{G}^{2m-k,k}(\ba)\right)=\psi_{2m-k,k}(\mathsf{G}^{2m-k,k}(\ba))$ are not linearly independent either. But Proposition~\ref{prop:image_of_homology_gen} directly implies that $\psi_{2m-k,k}(\mathsf{G}^{2m-k,k}(\ba))=\psi_{m,m}(\mathsf{G}^{m,m}(\ba))$ and the $\psi_{m,m}(\mathsf{G}^{m,m}(\ba))$ are linearly independent (as a subset of the linearly independent vectors in Proposition~\ref{prop:technical_prop_1}), a contradiction. In particular, the map $\Gamma^{2m-k,k}$ defines an isomorphism. 

The commutativity of the diagram follows from $\psi_{2m-k,k}(\mathsf{G}^{2m-k,k}(\ba))=\psi_{m,m}(\mathsf{G}^{m,m}(\ba))$. Since $\gamma_{m,m}$ is injective by Lemma~\ref{lem:iso_onto_image}, the same is true for $\gamma_{2m-k,k}$ by the commutativity of \eqref{diagProp24}.
\end{proof}

The next proposition (which is an immediate consequence of Proposition~\ref{prop:image_of_homology_gen} and the commutative diagram (\ref{eq:key_comm_diagram})) provides a combinatorial description of the map $\gamma_{2m-k,k}$.  

\begin{prop} \label{prop:explicit_isom}
The isomorphism $\gamma_{2m-k,k}\colon H_*(\SOfiber)\xrightarrow\cong\im(\gamma_{2m-k,k})\subseteq H_*((\mathbb S^2)^m)$ is given by 
\[
\ba\mapsto L_\ba=\sum_{U\in\mathcal U_\ba}(-1)^{\Lambda_\ba(U)}l_U,
\]
where $\mathcal U_\ba$ is the set of all subsets $U\subseteq\{1,\ldots,m\}$ containing precisely one endpoint of every cup of $\ba$, and $\Lambda_\ba(U)$ is the number of right endpoints of unmarked cups plus the number of endpoints of marked cups in $U$.
\end{prop}

\begin{ex}
If $m=3$ we obtain the diagrammatic homology basis
\[
H_0\left(\mathcal S^{3,3}_\mathrm{KL}\right) = \Biggl\langle\;
\begin{tikzpicture}[baseline={(0,-.45)}, scale=.8]
\draw[thick] (0,0) -- +(0,-.9);
\draw[thick] (.5,0) -- +(0,-.9);
\draw[thick] (1,0) -- +(0,-.9);
\end{tikzpicture}
\;\Biggr\rangle,
\hspace{2.2em}
H_2\left(\mathcal S^{3,3}_\mathrm{KL}\right) = \Biggl\langle
\begin{array}{cccccc}
\begin{tikzpicture}[scale=.8]
\draw[thick] (.5,0) .. controls +(0,-.5) and +(0,-.5) .. +(.5,0);
\draw[thick] (0,0) -- +(0,-.9);
\fill ([xshift=-2.5pt,yshift=-2.5pt]0,-.5) rectangle ++(5pt,5pt);
\end{tikzpicture}
&,&
\begin{tikzpicture}[scale=.8]
\draw[thick] (0,0) .. controls +(0,-.5) and +(0,-.5) .. +(.5,0);
\draw[thick] (1,0) -- +(0,-.9);
\fill ([xshift=-2.5pt,yshift=-2.5pt]1,-.5) rectangle ++(5pt,5pt);
\end{tikzpicture}
&,&
\begin{tikzpicture}[scale=.8]
\draw[thick] (0,0) .. controls +(0,-.5) and +(0,-.5) .. +(.5,0);
\fill ([xshift=-2.5pt,yshift=-2.5pt].25,-.365) rectangle ++(5pt,5pt);
\draw[thick] (1,0) -- +(0,-.9);
\end{tikzpicture}
\end{array}
\Biggr\rangle,
\]
and according to Proposition~\ref{prop:explicit_isom} the map $\gamma_{m,m}$ is given by 
\[
\begin{tikzpicture}[baseline={(0,-.45)}, scale=.8]
\draw[thick] (0,0) -- +(0,-.9);
\draw[thick] (.5,0) -- +(0,-.9);
\draw[thick] (1,0) -- +(0,-.9);
\end{tikzpicture}
\,\,\longmapsto\,\,
\begin{tikzpicture}[baseline={(0,-.45)}, scale=.8]
\draw[thick] (0,0) -- +(0,-.9);
\draw[thick] (.5,0) -- +(0,-.9);
\draw[thick] (1,0) -- +(0,-.9);
\draw (0,-.45) circle(3pt);
\draw (.5,-.45) circle(3pt);
\draw (1,-.45) circle(3pt);
\end{tikzpicture}
\hspace{9em}
\begin{tikzpicture}[baseline={(0,-.45)}, scale=.8]
\draw[thick] (0,0) .. controls +(0,-.5) and +(0,-.5) .. +(.5,0);
\draw[thick] (1,0) -- +(0,-.9);
\fill ([xshift=-2.5pt,yshift=-2.5pt]1,-.5) rectangle ++(5pt,5pt);
\end{tikzpicture}
\,\,\longmapsto\,\,
\begin{tikzpicture}[baseline={(0,-.45)}, scale=.8]
\draw[thick] (0,0) -- +(0,-.9);
\draw[thick] (.5,0) -- +(0,-.9);
\draw[thick] (1,0) -- +(0,-.9);
\draw (.5,-.45) circle(3pt);
\draw (1,-.45) circle(3pt);
\end{tikzpicture}
\,-\,
\begin{tikzpicture}[baseline={(0,-.45)}, scale=.8]
\draw[thick] (0,0) -- +(0,-.9);
\draw[thick] (.5,0) -- +(0,-.9);
\draw[thick] (1,0) -- +(0,-.9);
\draw (0,-.45) circle(3pt);
\draw (1,-.45) circle(3pt);
\end{tikzpicture}
\]
\[
\hspace{1.4em}
\begin{tikzpicture}[baseline={(0,-.45)}, scale=.8]
\draw[thick] (.5,0) .. controls +(0,-.5) and +(0,-.5) .. +(.5,0);
\draw[thick] (0,0) -- +(0,-.9);
\fill ([xshift=-2.5pt,yshift=-2.5pt]0,-.5) rectangle ++(5pt,5pt);
\end{tikzpicture}
\,\,\longmapsto\,\,
\begin{tikzpicture}[baseline={(0,-.45)}, scale=.8]
\draw[thick] (0,0) -- +(0,-.9);
\draw[thick] (.5,0) -- +(0,-.9);
\draw[thick] (1,0) -- +(0,-.9);
\draw (0,-.45) circle(3pt);
\draw (1,-.45) circle(3pt);
\end{tikzpicture}
\,-\,
\begin{tikzpicture}[baseline={(0,-.45)}, scale=.8]
\draw[thick] (0,0) -- +(0,-.9);
\draw[thick] (.5,0) -- +(0,-.9);
\draw[thick] (1,0) -- +(0,-.9);
\draw (0,-.45) circle(3pt);
\draw (.5,-.45) circle(3pt);
\end{tikzpicture}
\hspace{5em}
\begin{tikzpicture}[baseline={(0,-.45)}, scale=.8]
\draw[thick] (0,0) .. controls +(0,-.5) and +(0,-.5) .. +(.5,0);
\fill ([xshift=-2.5pt,yshift=-2.5pt].25,-.365) rectangle ++(5pt,5pt);
\draw[thick] (1,0) -- +(0,-.9);
\end{tikzpicture}
\,\,\longmapsto\,\,
-\,
\begin{tikzpicture}[baseline={(0,-.45)}, scale=.8]
\draw[thick] (0,0) -- +(0,-.9);
\draw[thick] (.5,0) -- +(0,-.9);
\draw[thick] (1,0) -- +(0,-.9);
\draw (.5,-.45) circle(3pt);
\draw (1,-.45) circle(3pt);
\end{tikzpicture}
\,-\,
\begin{tikzpicture}[baseline={(0,-.45)}, scale=.8]
\draw[thick] (0,0) -- +(0,-.9);
\draw[thick] (.5,0) -- +(0,-.9);
\draw[thick] (1,0) -- +(0,-.9);
\draw (0,-.45) circle(3pt);
\draw (1,-.45) circle(3pt);
\end{tikzpicture}
\]
\end{ex}



\subsection{Proof of Theorem~\ref{introthm:cohomology}}

Using the results obtained so far we can provide a presentation of the cohomology ring of $\SOfiber$. We assume that $m\neq k$.

Firstly, the natural inclusion $\SOfiber\subseteq (\mathbb S^2)^m$ induces a surjective homomorphism of algebras $H^*((\mathbb S^2)^m)\twoheadrightarrow H^*(\SOfiber)$. The surjectivity follows directly from dualizing Proposition~\ref{prop:basis_general_case}, where we proved that the induced map in homology $H_*(\SOfiber)\hookrightarrow H_*((\mathbb S^2)^m)$ is injective. 

Secondly, the complex dimension of the algebraic variety $\mathcal B^{2m-k,k}_{\mathrm{SO}_{2m}}$ is known to be $\frac{k-1}{2}$, see e.g.~\cite[Theorem 6.5]{ES12}, and thus $H^i(\SOfiber)\cong H^i(\mathcal B^{2m-k,k}_{\mathrm{SO}_{2m}})=0$ for all $i>2\cdot\frac{k-1}{2}=k-1$. In particular, the graded ideal $H^{\geq k}((\mathbb S^2)^m)=\oplus_{i\geq k} H^i((\mathbb S^2)^m)$ is contained in the kernel of our surjection and we obtain an induced surjective homomorphism of algebras 
\begin{equation} \label{eq:surjection_onto_cohomology}
\bigslant{H^*((\mathbb S^2)^m)}{H^{\geq k}((\mathbb S^2)^m)} \twoheadrightarrow H^*(\SOfiber),
\end{equation} 
which we claim to be an isomorphism. We prove this by comparing their dimensions. Recall that 
\begin{equation} \label{eq:cohomology_of_product_of_spheres}
H^*((\mathbb S^2)^m)\cong \bigslant{\mathbb C[X_1,\ldots,X_m]}{\langle X_i^2 \;| 1\leq i\leq m \rangle}       
\end{equation}
with $\textrm{deg}(X_i)=2$. Thus, the dimension of $H^{2i}((\mathbb S^2)^m)$ is given by counting all monomials of degree $i$ with pairwise different factors, i.e.\ $\dim(H^{2i}(\mathbb S^2)^m))=\tbinom{m}{i}$. It follows that the dimension of the quotient in (\ref{eq:surjection_onto_cohomology}) is $\sum_{i=0}^{\frac{k-1}{2}}\tbinom{m}{i}$ which equals $\dim H^*(\SOfiber)$ by Proposition~\ref{prop:dim_formula}.

Hence, the surjection (\ref{eq:surjection_onto_cohomology}) is in fact an isomorphism which we can translate (using the isomorphism (\ref{eq:cohomology_of_product_of_spheres}) and the fact that the ideal $H^{\geq k}((\mathbb S^2)^m)$ is generated by all monomials of degree~$k$) into the description of $H^*(\SOfiber)$ given in Theorem~\ref{introthm:cohomology} from the introduction.\hfill $\square$

\section[sec]{\for{toc}{Diagrammatic description of the actions}\except{toc}{Diagrammatic description of the actions and the proof of Theorems~\ref{thm:skein_calculus} and~\ref{thmB}}}
\label{section:section2}
In this section we prove Theorem~\ref{thm:skein_calculus} and Theorem~\ref{thmB} from the introduction.

\subsection{Weyl group actions on homology}
Assume $m\geq 4$ and let $\mathcal W_{D_m}$ be the Weyl group of type $D_m$, i.e.\ the Coxeter group with generators $s^D_0,s^D_1,\ldots,s^D_{m-1}$ subject to $\left(s^D_is^D_j\right)^{\alpha^D_{ij}}=e$ where
\begin{equation*}
\begin{array}{cc}
\alpha^D_{ij}=\begin{cases}
	1 & \text{if }i=j,\\
	3 & \text{if }i\,\begin{tikzpicture}[thick, scale=.55, baseline={(0,-.1)}]
\draw (-.4,0) -- +(.8,0);
\end{tikzpicture}\,j\text{ are connected in }\Gamma_{D_m},\\
	2 & \text{else}.
\end{cases}
&\hspace{3em}
\begin{tikzpicture}[thick, scale=.55, baseline={(0,-.1)}]
\node at (-1.3,0) {$\Gamma_{D_m}:$};
\draw (1,0) -- +(1.4,0);
\draw (2.4,0) -- +(.6,0);
\draw (0,-1) -- +(1,1);
\draw (0,1) -- +(1,-1);
\draw (4.6,0) -- +(.6,0);
\fill (0,1) circle(4pt);
\fill (0,-1) circle(4pt);
\fill (1,0) circle(4pt);
\fill (2.4,0) circle(4pt);
\fill (5.2,0) circle(4pt);
\begin{footnotesize}
\node at (0,-1.5) {$0$};
\node at (0,1.5) {$1$};
\node at (1,.5) {$2$};
\node at (2.4,.5) {$3$};
\node at (5.2,.5) {$m-1$};
\end{footnotesize}
\node at (3.8,0) {$\dots$};
\end{tikzpicture}
\end{array}
\end{equation*}
The Weyl group $\mathcal W_{C_m}$ of type $C_m$ (which is isomorphic to the Weyl group of type $B_m$) is the Coxeter group with generators $s_0^C,\ldots,s_{m-1}^C$ and relations $\left(s^C_is^C_j\right)^{\alpha^C_{ij}}=e$, where
\begin{equation*}
\begin{array}{cc}
\alpha^C_{ij}=\begin{cases}
	1 & \text{if }i=j,\\
	3 & \text{if }i\,\begin{tikzpicture}[thick, scale=.55, baseline={(0,-.1)}]
\draw (-.4,0) -- +(.8,0);
\end{tikzpicture}\,j\text{ are connected in }\Gamma_{C_m},\\
	4 & \text{if }i\,\begin{tikzpicture}[thick, scale=.55, baseline={(0,-.1)}]
\draw (-.4,.1) -- +(.8,0);
\draw (-.4,-.1) -- +(.8,0);
\end{tikzpicture}\,j\text{ are connected in }\Gamma_{C_m},\\
	2 & \text{else}.
\end{cases}
&\hspace{3em}
\begin{tikzpicture}[thick, scale=.55, baseline={(0,-.1)}]
\node at (-1.9,0) {$\Gamma_{C_m}:$};
\draw (1,0) -- +(1.4,0);
\draw (2.4,0) -- +(.6,0);
\draw (-.4,.1) -- +(1.4,0);
\draw (-.4,-.1) -- +(1.4,0);
\draw (4.6,0) -- +(.6,0);
\fill (-.4,0) circle(4pt);
\fill (1,0) circle(4pt);
\fill (2.4,0) circle(4pt);
\fill (5.2,0) circle(4pt);
\begin{footnotesize}
\node at (-.4,.5) {$0$};
\node at (1,.5) {$1$};
\node at (2.4,.5) {$2$};
\node at (5.2,.5) {$m-1$};
\end{footnotesize}
\node at (3.8,0) {$\dots$};
\end{tikzpicture}
\end{array}
\end{equation*}

\begin{rem} \label{rem:embedding_of_Weyl_grps}
One easily checks, see also \cite[\S 2.1]{Tok84}, that the maps 
\begin{eqnarray*}
\mathcal W_{C_{m-1}} \hookrightarrow\mathcal W_{D_m}\,,\,\,s_i^C \mapsto\begin{cases}
s_0^Ds^D_1 & \text{if }i=0,\\
s^D_{i+1} & \text{if }i\neq 0,
\end{cases}
&&
\mathcal W_{D_m} \hookrightarrow\mathcal W_{C_m}\,,\,\,s_i^D \mapsto\begin{cases}
s_0^Cs^C_1s_0^C & \text{if }i=0,\\
s^C_i & \text{if }i\neq 0,
\end{cases}
\end{eqnarray*}
define embeddings of groups. 
\end{rem}

The group $\mathcal W_{D_m}$ (and thus by Remark~\ref{rem:embedding_of_Weyl_grps} also the subgroup $W_{C_{m-1}}$) acts from the right on $\left(\mathbb S^2\right)^m$ according to the following rules:
\begin{align*}
&(x_1,x_2,x_3,\ldots,x_m).s^D_0=(-x_2,-x_1,x_3,\ldots,x_m), \\
&(x_1,\ldots,x_i,x_{i+1},\ldots,x_m).s^D_i=(x_1,\ldots,x_{i+1},x_i,\ldots,x_m), \hspace{2em} i\in\{1,\ldots,m-1\}.
\end{align*}

Given $i\in\{1,\ldots,m-1\}$, let $\tau_i\colon\{1,\ldots,m\}\to\{1,\ldots,m\}$ be the involution of sets which exchanges $i$ and $i+1$ and fixes all other elements. Given a subset $U\subseteq\{1,\ldots,m\}$, we define $\tau_iU$ as the subset of $\{1,\ldots,m\}$ containing the images of the elements of $U$ under the map $\tau_i$. 

By standard algebraic topology we can then describe the induced action by the formulas 
\begin{eqnarray}
\label{actionbeforeProp29}
l_U.s^D_0=(-1)^{\chi(U)}l_{\tau_1U}\,\hspace{1.6em} l_U.s^D_i=l_{\tau_iU}\,,\,\,i\in\{1,\ldots,m-1\},
\end{eqnarray}
where $\chi(U)$ is the cardinality of the set $\{1,2\}\cap U$.

\subsection{Proof of Theorem~\ref{thm:skein_calculus}}
\begin{proof}[]
The main insight is the following observation.
\begin{prop}\label{proposition:well_defined_action}
Given a cup diagram $\ba\in H_*(\SOfiber)$ and a generator $s_i^D\in\mathcal W_{D_m}$, $i \in\{0,\ldots,m-1\}$, we claim that 
\[
\ba.s_i^D:=\gamma_{2m-k,k}^{-1}\left(\gamma_{2m-k,k}(\ba).s_i^D\right)
\]
yields a well-defined right action of $\mathcal W_{D_m}$ on $H_*(\SOfiber)$ which is given by Table~\ref{tab:s_i^D_action} and Table~\ref{tab:s_0^D_action}.
\end{prop}

Theorem~\ref{thm:skein_calculus} is now a consequence of Proposition~\ref{proposition:well_defined_action}: A straightforward case-by-case analysis shows that the $\mathcal W_{D_m}$-action from Tables~\ref{tab:s_i^D_action} and~\ref{tab:s_0^D_action} coincides with that from the skein calculus.\end{proof}

\begin{table}
\centering
\begin{tabular}{ c | c || c | c}
$\ba$ & $\ba.s_i^D$ & $\ba$ & $\ba.s_i^D$ \\
\hline\hline
\begin{tikzpicture}[scale=.8]
\draw[thick] (0,0) -- +(0,-.9);
\draw[thick] (.5,0) -- +(0,-.9);
\node at (0,0) {$\times$};
\node at (.5,0) {$\times$};
\end{tikzpicture}
&\begin{tikzpicture}[scale=.8]
\draw[thick] (0,0) -- +(0,-.9);
\draw[thick] (.5,0) -- +(0,-.9);
\node at (0,0) {$\times$};
\node at (.5,0) {$\times$};
\end{tikzpicture}
&
\begin{tikzpicture}[scale=.8]
\draw[thick] (0,0) -- +(0,-.9);
\fill ([xshift=-2.5pt,yshift=-2.5pt]0,-.5) rectangle ++(5pt,5pt);
\draw[thick] (.5,0) -- +(0,-.9);
\node at (0,0) {$\times$};
\node at (.5,0) {$\times$};
\end{tikzpicture}
&\begin{tikzpicture}[scale=.8]
\draw[thick] (0,0) -- +(0,-.9);
\fill ([xshift=-2.5pt,yshift=-2.5pt]0,-.5) rectangle ++(5pt,5pt);
\draw[thick] (.5,0) -- +(0,-.9);
\node at (0,0) {$\times$};
\node at (.5,0) {$\times$};
\end{tikzpicture}
\\ 
\begin{tikzpicture}[scale=.8]
\draw[thick] (0,0) .. controls +(0,-.5) and +(0,-.5) .. +(.5,0);
\node at (0,0) {$\times$};
\node at (.5,0) {$\times$};
\end{tikzpicture}
&
$-$\begin{tikzpicture}[scale=.8]
\draw[thick] (0,0) .. controls +(0,-.5) and +(0,-.5) .. +(.5,0);
\node at (0,0) {$\times$};
\node at (.5,0) {$\times$};
\end{tikzpicture}
&
\begin{tikzpicture}[scale=.8]
\draw[thick] (0,0) .. controls +(0,-.5) and +(0,-.5) .. +(.5,0);
\fill ([xshift=-2.5pt,yshift=-2.5pt].25,-.365) rectangle ++(5pt,5pt);
\node at (0,0) {$\times$};
\node at (.5,0) {$\times$};
\end{tikzpicture}
&
\begin{tikzpicture}[scale=.8]
\draw[thick] (0,0) .. controls +(0,-.5) and +(0,-.5) .. +(.5,0);
\fill ([xshift=-2.5pt,yshift=-2.5pt].25,-.365) rectangle ++(5pt,5pt);
\node at (0,0) {$\times$};
\node at (.5,0) {$\times$};
\end{tikzpicture}
\\
\begin{tikzpicture}[scale=.8,baseline={(0,-.5)}]
\draw[thick] (0,0) .. controls +(0,-.5) and +(0,-.5) .. +(.5,0);
\draw[thick] (1,0) -- +(0,-.9);
\node at (.5,0) {$\times$};
\node at (1,0) {$\times$};
\end{tikzpicture}
&
\begin{tikzpicture}[scale=.8,baseline={(0,-.5)}]
\draw[thick] (0,0) .. controls +(0,-.5) and +(0,-.5) .. +(.5,0);
\draw[thick] (1,0) -- +(0,-.9);
\node at (.5,0) {$\times$};
\node at (1,0) {$\times$};
\end{tikzpicture}
$+$
\begin{tikzpicture}[scale=.8,baseline={(0,-.5)}]
\draw[thick] (.5,0) .. controls +(0,-.5) and +(0,-.5) .. +(.5,0);
\draw[thick] (0,0) -- +(0,-.9);
\node at (.5,0) {$\times$};
\node at (1,0) {$\times$};
\end{tikzpicture}
&
\begin{tikzpicture}[scale=.8,baseline={(0,-.5)}]
\draw[thick] (0,0) .. controls +(0,-.5) and +(0,-.5) .. +(.5,0);
\draw[thick] (1,0) -- +(0,-.9);

\fill ([xshift=-2.5pt,yshift=-2.5pt].25,-.365) rectangle ++(5pt,5pt);
\fill ([xshift=-2.5pt,yshift=-2.5pt]1,-.5) rectangle ++(5pt,5pt);
\node at (.5,0) {$\times$};
\node at (1,0) {$\times$};
\end{tikzpicture}
&
\begin{tikzpicture}[scale=.8,baseline={(0,-.5)}]
\draw[thick] (0,0) .. controls +(0,-.5) and +(0,-.5) .. +(.5,0);
\draw[thick] (1,0) -- +(0,-.9);

\fill ([xshift=-2.5pt,yshift=-2.5pt].25,-.365) rectangle ++(5pt,5pt);
\fill ([xshift=-2.5pt,yshift=-2.5pt]1,-.5) rectangle ++(5pt,5pt);
\node at (.5,0) {$\times$};
\node at (1,0) {$\times$};
\end{tikzpicture}
$+$
\begin{tikzpicture}[scale=.8,baseline={(0,-.5)}]
\draw[thick] (.5,0) .. controls +(0,-.5) and +(0,-.5) .. +(.5,0);
\draw[thick] (0,0) -- +(0,-.9);
\node at (.5,0) {$\times$};
\node at (1,0) {$\times$};
\end{tikzpicture}
\\ 
\begin{tikzpicture}[scale=.8,baseline={(0,-.5)}]
\draw[thick] (.5,0) .. controls +(0,-.5) and +(0,-.5) .. +(.5,0);
\draw[thick] (0,0) -- +(0,-.9);
\node at (0,0) {$\times$};
\node at (.5,0) {$\times$};
\end{tikzpicture}
&
\begin{tikzpicture}[scale=.8,baseline={(0,-.5)}]
\draw[thick] (.5,0) .. controls +(0,-.5) and +(0,-.5) .. +(.5,0);
\draw[thick] (0,0) -- +(0,-.9);
\node at (0,0) {$\times$};
\node at (.5,0) {$\times$};
\end{tikzpicture}
$+$
\begin{tikzpicture}[scale=.8,baseline={(0,-.5)}]
\draw[thick] (0,0) .. controls +(0,-.5) and +(0,-.5) .. +(.5,0);
\draw[thick] (1,0) -- +(0,-.9);
\node at (0,0) {$\times$};
\node at (.5,0) {$\times$};
\end{tikzpicture}
&
\begin{tikzpicture}[scale=.8,baseline={(0,-.5)}]
\draw[thick] (0,0) .. controls +(0,-.5) and +(0,-.5) .. +(.5,0);
\draw[thick] (1,0) -- +(0,-.9);

\fill ([xshift=-2.5pt,yshift=-2.5pt]1,-.5) rectangle ++(5pt,5pt);
\node at (.5,0) {$\times$};
\node at (1,0) {$\times$};
\end{tikzpicture}
&
\begin{tikzpicture}[scale=.8,baseline={(0,-.5)}]
\draw[thick] (0,0) .. controls +(0,-.5) and +(0,-.5) .. +(.5,0);
\draw[thick] (1,0) -- +(0,-.9);

\fill ([xshift=-2.5pt,yshift=-2.5pt]1,-.5) rectangle ++(5pt,5pt);
\node at (.5,0) {$\times$};
\node at (1,0) {$\times$};
\end{tikzpicture}
$+$
\begin{tikzpicture}[scale=.8,baseline={(0,-.5)}]
\draw[thick] (.5,0) .. controls +(0,-.5) and +(0,-.5) .. +(.5,0);
\draw[thick] (0,0) -- +(0,-.9);
\fill ([xshift=-2.5pt,yshift=-2.5pt]0,-.5) rectangle ++(5pt,5pt);
\node at (.5,0) {$\times$};
\node at (1,0) {$\times$};
\end{tikzpicture}
\\
\begin{tikzpicture}[scale=.8,baseline={(0,-.5)}]
\draw[thick] (0,0) .. controls +(0,-.5) and +(0,-.5) .. +(.5,0);
\draw[thick] (1,0) -- +(0,-.9);

\fill ([xshift=-2.5pt,yshift=-2.5pt].25,-.365) rectangle ++(5pt,5pt);
\node at (.5,0) {$\times$};
\node at (1,0) {$\times$};
\end{tikzpicture}
&
\begin{tikzpicture}[scale=.8,baseline={(0,-.5)}]
\draw[thick] (0,0) .. controls +(0,-.5) and +(0,-.5) .. +(.5,0);
\draw[thick] (1,0) -- +(0,-.9);

\fill ([xshift=-2.5pt,yshift=-2.5pt].25,-.365) rectangle ++(5pt,5pt);
\node at (.5,0) {$\times$};
\node at (1,0) {$\times$};
\end{tikzpicture}
$+$
\begin{tikzpicture}[scale=.8,baseline={(0,-.5)}]
\draw[thick] (.5,0) .. controls +(0,-.5) and +(0,-.5) .. +(.5,0);
\draw[thick] (0,0) -- +(0,-.9);
\fill ([xshift=-2.5pt,yshift=-2.5pt]0,-.5) rectangle ++(5pt,5pt);
\node at (.5,0) {$\times$};
\node at (1,0) {$\times$};
\end{tikzpicture}
&
\begin{tikzpicture}[scale=.8,baseline={(0,-.5)}]
\draw[thick] (.5,0) .. controls +(0,-.5) and +(0,-.5) .. +(.5,0);
\draw[thick] (0,0) -- +(0,-.9);

\fill ([xshift=-2.5pt,yshift=-2.5pt]0,-.5) rectangle ++(5pt,5pt);
\node at (0,0) {$\times$};
\node at (.5,0) {$\times$};
\end{tikzpicture}
&
\begin{tikzpicture}[scale=.8,baseline={(0,-.5)}]
\draw[thick] (.5,0) .. controls +(0,-.5) and +(0,-.5) .. +(.5,0);
\draw[thick] (0,0) -- +(0,-.9);

\fill ([xshift=-2.5pt,yshift=-2.5pt]0,-.5) rectangle ++(5pt,5pt);
\node at (0,0) {$\times$};
\node at (.5,0) {$\times$};
\end{tikzpicture}
$+$
\begin{tikzpicture}[scale=.8,baseline={(0,-.5)}]
\draw[thick] (0,0) .. controls +(0,-.5) and +(0,-.5) .. +(.5,0);
\draw[thick] (1,0) -- +(0,-.9);
\fill ([xshift=-2.5pt,yshift=-2.5pt]1,-.5) rectangle ++(5pt,5pt);
\node at (0,0) {$\times$};
\node at (.5,0) {$\times$};
\end{tikzpicture}
\\
\begin{tikzpicture}[scale=.8]
\draw[thick] (0,0) .. controls +(0,-.5) and +(0,-.5) .. +(.5,0);
\draw[thick] (1,0) .. controls +(0,-.5) and +(0,-.5) .. +(.5,0);
\node at (.5,0) {$\times$};
\node at (1,0) {$\times$};
\end{tikzpicture}
&
\begin{tikzpicture}[scale=.8]
\draw[thick] (0,0) .. controls +(0,-.5) and +(0,-.5) .. +(.5,0);
\draw[thick] (1,0) .. controls +(0,-.5) and +(0,-.5) .. +(.5,0);
\node at (.5,0) {$\times$};
\node at (1,0) {$\times$};
\end{tikzpicture}
$+$
\begin{tikzpicture}[scale=.8]
\draw[thick] (0,0) .. controls +(0,-1) and +(0,-1) .. +(1.5,0);
\draw[thick] (.5,0) .. controls +(0,-.5) and +(0,-.5) .. +(.5,0);
\node at (.5,0) {$\times$};
\node at (1,0) {$\times$};
\end{tikzpicture}
&
\begin{tikzpicture}[scale=.8]
\draw[thick] (0,0) .. controls +(0,-.5) and +(0,-.5) .. +(.5,0);
\draw[thick] (1,0) .. controls +(0,-.5) and +(0,-.5) .. +(.5,0);

\fill ([xshift=-2.5pt,yshift=-2.5pt]1.25,-.365) rectangle ++(5pt,5pt);
\fill ([xshift=-2.5pt,yshift=-2.5pt].25,-.365) rectangle ++(5pt,5pt);
\node at (.5,0) {$\times$};
\node at (1,0) {$\times$};
\end{tikzpicture}
&
\begin{tikzpicture}[scale=.8]
\draw[thick] (0,0) .. controls +(0,-.5) and +(0,-.5) .. +(.5,0);
\draw[thick] (1,0) .. controls +(0,-.5) and +(0,-.5) .. +(.5,0);

\fill ([xshift=-2.5pt,yshift=-2.5pt]1.25,-.365) rectangle ++(5pt,5pt);
\fill ([xshift=-2.5pt,yshift=-2.5pt].25,-.365) rectangle ++(5pt,5pt);
\node at (.5,0) {$\times$};
\node at (1,0) {$\times$};
\end{tikzpicture}
$+$
\begin{tikzpicture}[scale=.8]
\draw[thick] (0,0) .. controls +(0,-1) and +(0,-1) .. +(1.5,0);
\draw[thick] (.5,0) .. controls +(0,-.5) and +(0,-.5) .. +(.5,0);
\node at (.5,0) {$\times$};
\node at (1,0) {$\times$};
\end{tikzpicture}
\\
\begin{tikzpicture}[scale=.8]
\draw[thick] (0,0) .. controls +(0,-.5) and +(0,-.5) .. +(.5,0);
\draw[thick] (1,0) .. controls +(0,-.5) and +(0,-.5) .. +(.5,0);

\fill ([xshift=-2.5pt,yshift=-2.5pt].25,-.365) rectangle ++(5pt,5pt);
\node at (.5,0) {$\times$};
\node at (1,0) {$\times$};
\end{tikzpicture}
&
\begin{tikzpicture}[scale=.8]
\draw[thick] (0,0) .. controls +(0,-.5) and +(0,-.5) .. +(.5,0);
\draw[thick] (1,0) .. controls +(0,-.5) and +(0,-.5) .. +(.5,0);

\fill ([xshift=-2.5pt,yshift=-2.5pt].25,-.365) rectangle ++(5pt,5pt);
\node at (.5,0) {$\times$};
\node at (1,0) {$\times$};
\end{tikzpicture}
$+$
\begin{tikzpicture}[scale=.8]
\draw[thick] (0,0) .. controls +(0,-1) and +(0,-1) .. +(1.5,0);
\draw[thick] (.5,0) .. controls +(0,-.5) and +(0,-.5) .. +(.5,0);

\fill ([xshift=-2.5pt,yshift=-2.5pt].75,-.75) rectangle ++(5pt,5pt);
\node at (.5,0) {$\times$};
\node at (1,0) {$\times$};
\end{tikzpicture}
&
\begin{tikzpicture}[scale=.8]
\draw[thick] (0,0) .. controls +(0,-.5) and +(0,-.5) .. +(.5,0);
\draw[thick] (1,0) .. controls +(0,-.5) and +(0,-.5) .. +(.5,0);
\fill ([xshift=-2.5pt,yshift=-2.5pt]1.25,-.365) rectangle ++(5pt,5pt);
\node at (.5,0) {$\times$};
\node at (1,0) {$\times$};
\end{tikzpicture}
&
\begin{tikzpicture}[scale=.8]
\draw[thick] (0,0) .. controls +(0,-.5) and +(0,-.5) .. +(.5,0);
\draw[thick] (1,0) .. controls +(0,-.5) and +(0,-.5) .. +(.5,0);

\fill ([xshift=-2.5pt,yshift=-2.5pt]1.25,-.365) rectangle ++(5pt,5pt);
\node at (.5,0) {$\times$};
\node at (1,0) {$\times$};
\end{tikzpicture}
$+$
\begin{tikzpicture}[scale=.8]
\draw[thick] (0,0) .. controls +(0,-1) and +(0,-1) .. +(1.5,0);
\draw[thick] (.5,0) .. controls +(0,-.5) and +(0,-.5) .. +(.5,0);

\fill ([xshift=-2.5pt,yshift=-2.5pt].75,-.74) rectangle ++(5pt,5pt);
\node at (.5,0) {$\times$};
\node at (1,0) {$\times$};
\end{tikzpicture}
\\
\begin{tikzpicture}[scale=.8]
\draw[thick] (0,0) .. controls +(0,-1) and +(0,-1) .. +(1.5,0);
\draw[thick] (.5,0) .. controls +(0,-.5) and +(0,-.5) .. +(.5,0);
\node at (1,0) {$\times$};
\node at (1.5,0) {$\times$};
\end{tikzpicture}
&
\begin{tikzpicture}[scale=.8]
\draw[thick] (0,0) .. controls +(0,-1) and +(0,-1) .. +(1.5,0);
\draw[thick] (.5,0) .. controls +(0,-.5) and +(0,-.5) .. +(.5,0);
\node at (1,0) {$\times$};
\node at (1.5,0) {$\times$};
\end{tikzpicture}
$+$
\begin{tikzpicture}[scale=.8]
\draw[thick] (0,0) .. controls +(0,-.5) and +(0,-.5) .. +(.5,0);
\draw[thick] (1,0) .. controls +(0,-.5) and +(0,-.5) .. +(.5,0);
\node at (1,0) {$\times$};
\node at (1.5,0) {$\times$};
\end{tikzpicture}
&
\begin{tikzpicture}[scale=.8]
\draw[thick] (0,0) .. controls +(0,-1) and +(0,-1) .. +(1.5,0);
\draw[thick] (.5,0) .. controls +(0,-.5) and +(0,-.5) .. +(.5,0);

\fill ([xshift=-2.5pt,yshift=-2.5pt].75,-.74) rectangle ++(5pt,5pt);
\node at (1,0) {$\times$};
\node at (1.5,0) {$\times$};
\end{tikzpicture}
&
\begin{tikzpicture}[scale=.8]
\draw[thick] (0,0) .. controls +(0,-1) and +(0,-1) .. +(1.5,0);
\draw[thick] (.5,0) .. controls +(0,-.5) and +(0,-.5) .. +(.5,0);

\fill ([xshift=-2.5pt,yshift=-2.5pt].75,-.74) rectangle ++(5pt,5pt);
\node at (1,0) {$\times$};
\node at (1.5,0) {$\times$};
\end{tikzpicture}
$+$
\begin{tikzpicture}[scale=.8]
\draw[thick] (0,0) .. controls +(0,-.5) and +(0,-.5) .. +(.5,0);
\draw[thick] (1,0) .. controls +(0,-.5) and +(0,-.5) .. +(.5,0);

\fill ([xshift=-2.5pt,yshift=-2.5pt].25,-.365) rectangle ++(5pt,5pt);
\node at (1,0) {$\times$};
\node at (1.5,0) {$\times$};
\end{tikzpicture}
\\
\begin{tikzpicture}[scale=.8]
\draw[thick] (0,0) .. controls +(0,-1) and +(0,-1) .. +(1.5,0);
\draw[thick] (.5,0) .. controls +(0,-.5) and +(0,-.5) .. +(.5,0);
\node at (0,0) {$\times$};
\node at (.5,0) {$\times$};
\end{tikzpicture}
&
\begin{tikzpicture}[scale=.8]
\draw[thick] (0,0) .. controls +(0,-1) and +(0,-1) .. +(1.5,0);
\draw[thick] (.5,0) .. controls +(0,-.5) and +(0,-.5) .. +(.5,0);
\node at (0,0) {$\times$};
\node at (.5,0) {$\times$};
\end{tikzpicture}
$+$
\begin{tikzpicture}[scale=.8]
\draw[thick] (0,0) .. controls +(0,-.5) and +(0,-.5) .. +(.5,0);
\draw[thick] (1,0) .. controls +(0,-.5) and +(0,-.5) .. +(.5,0);
\node at (0,0) {$\times$};
\node at (.5,0) {$\times$};
\end{tikzpicture}
&
\begin{tikzpicture}[scale=.8]
\draw[thick] (0,0) .. controls +(0,-1) and +(0,-1) .. +(1.5,0);
\draw[thick] (.5,0) .. controls +(0,-.5) and +(0,-.5) .. +(.5,0);

\fill ([xshift=-2.5pt,yshift=-2.5pt].75,-.74) rectangle ++(5pt,5pt);
\node at (0,0) {$\times$};
\node at (.5,0) {$\times$};
\end{tikzpicture}
&
\begin{tikzpicture}[scale=.8]
\draw[thick] (0,0) .. controls +(0,-1) and +(0,-1) .. +(1.5,0);
\draw[thick] (.5,0) .. controls +(0,-.5) and +(0,-.5) .. +(.5,0);

\fill ([xshift=-2.5pt,yshift=-2.5pt].75,-.74) rectangle ++(5pt,5pt);
\node at (0,0) {$\times$};
\node at (.5,0) {$\times$};
\end{tikzpicture}
$+$
\begin{tikzpicture}[scale=.8]
\draw[thick] (0,0) .. controls +(0,-.5) and +(0,-.5) .. +(.5,0);
\draw[thick] (1,0) .. controls +(0,-.5) and +(0,-.5) .. +(.5,0);

\fill ([xshift=-2.5pt,yshift=-2.5pt]1.25,-.365) rectangle ++(5pt,5pt);
\node at (0,0) {$\times$};
\node at (.5,0) {$\times$};
\end{tikzpicture}
\end{tabular}
\caption{The action of a generator $s_i^D\in\mathcal W_{D_m}$, $i\in\{1,\ldots,m-1\}$, on a cup diagram (crosses indicate vertices $i$ and $i+1$).}
\label{tab:s_i^D_action}
\end{table}

\begin{table}
\centering
\begin{tabular}{ c | c || c | c}
$\ba$ & $\ba.s^D_0$ & $\ba$ & $\ba.s^D_0$ \\
\hline
\begin{tikzpicture}[scale=.8,baseline={(0,-.5)}]
\draw[thick] (0,0) -- +(0,-.9);
\draw[thick] (.5,0) -- +(0,-.9);
\node at (0,0) {$\times$};
\node at (.5,0) {$\times$};
\end{tikzpicture}
&\begin{tikzpicture}[scale=.8,baseline={(0,-.5)}]
\draw[thick] (0,0) -- +(0,-.9);
\draw[thick] (.5,0) -- +(0,-.9);
\node at (0,0) {$\times$};
\node at (.5,0) {$\times$};
\end{tikzpicture}
&
\begin{tikzpicture}[scale=.8,baseline={(0,-.5)}]
\draw[thick] (0,0) -- +(0,-.9);
\fill ([xshift=-2.5pt,yshift=-2.5pt]0,-.5) rectangle ++(5pt,5pt);
\draw[thick] (.5,0) -- +(0,-.9);
\node at (0,0) {$\times$};
\node at (.5,0) {$\times$};
\end{tikzpicture}
&\begin{tikzpicture}[scale=.8,baseline={(0,-.5)}]
\draw[thick] (0,0) -- +(0,-.9);
\fill ([xshift=-2.5pt,yshift=-2.5pt]0,-.5) rectangle ++(5pt,5pt);
\draw[thick] (.5,0) -- +(0,-.9);
\node at (0,0) {$\times$};
\node at (.5,0) {$\times$};
\end{tikzpicture}
\\
\begin{tikzpicture}[scale=.8,baseline={(0,-.5)}]
\draw[thick] (0,0) .. controls +(0,-.5) and +(0,-.5) .. +(.5,0);
\node at (0,0) {$\times$};
\node at (.5,0) {$\times$};
\end{tikzpicture}
&
\begin{tikzpicture}[scale=.8,baseline={(0,-.5)}]
\draw[thick] (0,0) .. controls +(0,-.5) and +(0,-.5) .. +(.5,0);
\node at (0,0) {$\times$};
\node at (.5,0) {$\times$};
\end{tikzpicture}
&
\begin{tikzpicture}[scale=.8,baseline={(0,-.5)}]
\draw[thick] (0,0) .. controls +(0,-.5) and +(0,-.5) .. +(.5,0);
\fill ([xshift=-2.5pt,yshift=-2.5pt].25,-.365) rectangle ++(5pt,5pt);
\node at (0,0) {$\times$};
\node at (.5,0) {$\times$};
\end{tikzpicture}
&
$-$\begin{tikzpicture}[scale=.8,baseline={(0,-.5)}]
\draw[thick] (0,0) .. controls +(0,-.5) and +(0,-.5) .. +(.5,0);
\fill ([xshift=-2.5pt,yshift=-2.5pt].25,-.365) rectangle ++(5pt,5pt);
\node at (0,0) {$\times$};
\node at (.5,0) {$\times$};
\end{tikzpicture}
\\
\begin{tikzpicture}[scale=.8,baseline={(0,-.5)}]
\draw[thick] (.5,0) .. controls +(0,-.5) and +(0,-.5) .. +(.5,0);
\draw[thick] (0,0) -- +(0,-.9);
\node at (0,0) {$\times$};
\node at (.5,0) {$\times$};
\end{tikzpicture}
&
\begin{tikzpicture}[scale=.8,baseline={(0,-.5)}]
\draw[thick] (.5,0) .. controls +(0,-.5) and +(0,-.5) .. +(.5,0);
\draw[thick] (0,0) -- +(0,-.9);
\node at (0,0) {$\times$};
\node at (.5,0) {$\times$};
\end{tikzpicture}
$+$
\begin{tikzpicture}[scale=.8,baseline={(0,-.5)}]
\draw[thick] (0,0) .. controls +(0,-.5) and +(0,-.5) .. +(.5,0);
\draw[thick] (1,0) -- +(0,-.9);
\fill ([xshift=-2.5pt,yshift=-2.5pt].25,-.365) rectangle ++(5pt,5pt);
\fill ([xshift=-2.5pt,yshift=-2.5pt]1,-.5) rectangle ++(5pt,5pt);
\node at (0,0) {$\times$};
\node at (.5,0) {$\times$};
\end{tikzpicture}
&
\begin{tikzpicture}[scale=.8,baseline={(0,-.5)}]
\draw[thick] (.5,0) .. controls +(0,-.5) and +(0,-.5) .. +(.5,0);
\draw[thick] (0,0) -- +(0,-.9);

\fill ([xshift=-2.5pt,yshift=-2.5pt]0,-.5) rectangle ++(5pt,5pt);
\node at (0,0) {$\times$};
\node at (.5,0) {$\times$};
\end{tikzpicture}
&
\begin{tikzpicture}[scale=.8,baseline={(0,-.5)}]
\draw[thick] (.5,0) .. controls +(0,-.5) and +(0,-.5) .. +(.5,0);
\draw[thick] (0,0) -- +(0,-.9);

\fill ([xshift=-2.5pt,yshift=-2.5pt]0,-.5) rectangle ++(5pt,5pt);
\node at (0,0) {$\times$};
\node at (.5,0) {$\times$};
\end{tikzpicture}
$+$
\begin{tikzpicture}[scale=.8,baseline={(0,-.5)}]
\draw[thick] (0,0) .. controls +(0,-.5) and +(0,-.5) .. +(.5,0);
\draw[thick] (1,0) -- +(0,-.9);
\fill ([xshift=-2.5pt,yshift=-2.5pt].25,-.365) rectangle ++(5pt,5pt);
\node at (0,0) {$\times$};
\node at (.5,0) {$\times$};
\end{tikzpicture}
\\
\begin{tikzpicture}[scale=.8,baseline={(0,-.5)}]
\begin{scope}[xshift=3cm]
\draw[thick] (0,0) .. controls +(0,-1) and +(0,-1) .. +(1.5,0);
\draw[thick] (.5,0) .. controls +(0,-.5) and +(0,-.5) .. +(.5,0);
\node at (0,0) {$\times$};
\node at (.5,0) {$\times$};
\end{scope}
\end{tikzpicture}
&
\begin{tikzpicture}[scale=.8,baseline={(0,-.5)}]
\begin{scope}[xshift=3cm]
\draw[thick] (0,0) .. controls +(0,-1) and +(0,-1) .. +(1.5,0);
\draw[thick] (.5,0) .. controls +(0,-.5) and +(0,-.5) .. +(.5,0);
\node at (0,0) {$\times$};
\node at (.5,0) {$\times$};
\end{scope}
\end{tikzpicture}
$+$
\begin{tikzpicture}[scale=.8,baseline={(0,-.5)}]
\begin{scope}[xshift=3cm]
\draw[thick] (0,0) .. controls +(0,-.5) and +(0,-.5) .. +(.5,0);
\draw[thick] (1,0) .. controls +(0,-.5) and +(0,-.5) .. +(.5,0);

\fill ([xshift=-2.5pt,yshift=-2.5pt].25,-.365) rectangle ++(5pt,5pt);
\fill ([xshift=-2.5pt,yshift=-2.5pt]1.25,-.365) rectangle ++(5pt,5pt);
\node at (0,0) {$\times$};
\node at (.5,0) {$\times$};
\end{scope}
\end{tikzpicture}
&
\begin{tikzpicture}[scale=.8,baseline={(0,-.5)}]
\begin{scope}[xshift=3cm]
\draw[thick] (0,0) .. controls +(0,-1) and +(0,-1) .. +(1.5,0);
\draw[thick] (.5,0) .. controls +(0,-.5) and +(0,-.5) .. +(.5,0);

\fill ([xshift=-2.5pt,yshift=-2.5pt].75,-.74) rectangle ++(5pt,5pt);
\node at (0,0) {$\times$};
\node at (.5,0) {$\times$};
\end{scope}
\end{tikzpicture}
&
\begin{tikzpicture}[scale=.8,baseline={(0,-.5)}]
\begin{scope}[xshift=3cm]
\draw[thick] (0,0) .. controls +(0,-1) and +(0,-1) .. +(1.5,0);
\draw[thick] (.5,0) .. controls +(0,-.5) and +(0,-.5) .. +(.5,0);

\fill ([xshift=-2.5pt,yshift=-2.5pt].75,-.74) rectangle ++(5pt,5pt);
\node at (0,0) {$\times$};
\node at (.5,0) {$\times$};
\end{scope}
\end{tikzpicture}
$+$
\begin{tikzpicture}[scale=.8,baseline={(0,-.5)}]
\begin{scope}[xshift=3cm]
\draw[thick] (0,0) .. controls +(0,-.5) and +(0,-.5) .. +(.5,0);
\draw[thick] (1,0) .. controls +(0,-.5) and +(0,-.5) .. +(.5,0);

\fill ([xshift=-2.5pt,yshift=-2.5pt].25,-.365) rectangle ++(5pt,5pt);
\node at (0,0) {$\times$};
\node at (.5,0) {$\times$};
\end{scope}
\end{tikzpicture}
\end{tabular}
\caption{The action of the generator $s_0^D\in\mathcal W_{D_m}$ on a cup diagram (crosses indicate vertices $1$ and $2$).}
\label{tab:s_0^D_action}
\end{table}

\begin{proof}[Proof of Proposition~\ref{proposition:well_defined_action}]
Throughout the proof we fix a generator $s_i^D\in\mathcal W_{D_m}$, $i\in\{0,\ldots,m-1\}$. We distinguish between four different cases depending on what the cup diagram $\ba$ locally looks like at the vertices $i$ and $i+1$.

\bigskip
\noindent {\it Case 1: Vertices i and i+1 are both connected to a ray.} 

If $i$ and $i+1$ are both connected to rays in $\ba$, then we have $\tau_iU=U$ for all $U\in\mathcal U_\ba$ and 
\begin{equation} \label{eq:case_1_calculation}
L_\ba.s_i^D = \sum_{U\in\mathcal U_\ba} (-1)^{\Lambda_\ba(U)}l_{\tau_iU}=\sum_{U\in\mathcal U_\ba} (-1)^{\Lambda_\ba(U)}l_U = L_\ba,
\end{equation}
if $i\in\{1,\ldots,m-1\}$. Now the claim follows from applying the inverse of the isomorphism of Proposition~\ref{prop:explicit_isom} to (\ref{eq:case_1_calculation}). Since $\chi_U=0$ for all $U\in\mathcal U_\ba$, calculation (\ref{eq:case_1_calculation}) also proves that $s_0^D$ acts as indicated in Table~\ref{tab:s_0^D_action}. 

\bigskip
\noindent {\it Case 2: Vertices i and i+1 are connected by a cup.}

We decompose $\mathcal U_\ba=\mathcal U_{\ba,i}\sqcup\mathcal U_{\ba,i+1}$, where $\mathcal U_{\ba,i}$ (resp.\ $\mathcal U_{\ba,i+1}$) consists of all sets $U\in\mathcal U_\ba$ containing $i$ (resp.\ $i+1$). Then we have
\[
L_\ba=\sum_{U\in\mathcal U_\ba} (-1)^{\Lambda_\ba(U)}l_U =\sum_{U\in\mathcal U_{\ba,i}} (-1)^{\Lambda_\ba(U)}l_U + \sum_{U\in\mathcal U_{\ba,i+1}} (-1)^{\Lambda_\ba(U)}l_U.
\]
In the following we abbreviate $L_{\ba,i}=\sum_{U\in\mathcal U_{\ba,i}} (-1)^{\Lambda_\ba(U)}l_U$ and $L_{\ba,i+1}=\sum_{U\in\mathcal U_{\ba,i+1}} (-1)^{\Lambda_\ba(U)}l_U$.

Firstly, note that the assignment $U\mapsto \tau_iU$ induces a self-inverse bijection $\tau_i\colon\mathcal U_{\ba,i}\to\mathcal U_{\ba,i+1}$. Secondly, for all $U\in\mathcal U_\ba$ we have equalities $(-1)^{\Lambda_\ba(\tau_iU)}=(-1)^{\Lambda_\ba(U)-1}$ (resp.\ $(-1)^{\Lambda_\ba(\tau_iU)}=(-1)^{\Lambda_\ba(U)}$) if $i$ and $i+1$ are connected by an unmarked cup (resp.\ marked cup). Now we compute
\begin{align*}
L_{\ba,i}.s_i^D&=\sum_{U\in\mathcal U_{\ba,i}} (-1)^{\Lambda_\ba(U)}l_{\tau_iU}=\sum_{U\in\mathcal U_{\ba,i+1}} (-1)^{\Lambda_\ba(\tau_iU)}l_U \\
				&={\begin{cases}
					\sum_{U\in\mathcal U_{\ba,i+1}} (-1)^{\Lambda_\ba(U)-1}l_U=-L_{\ba,i+1} &\text{if }i\CupConnect (i+1), \\
					\sum_{U\in\mathcal U_{\ba,i+1}} (-1)^{\Lambda_\ba(U)}l_U=L_{\ba,i+1} &\text{if }i\DCupConnect (i+1).
				\end{cases}}
\end{align*}
Analogously, we obtain $L_{\ba,i+1}.s_i^D=-L_{\ba,i}$ (resp.\ $L_{\ba,i+1}.s_i^D=L_{\ba,i}$) if the cup connecting $i$ and $i+1$ is unmarked (resp.\ marked). Hence, we compute
\[
L_\ba.s_i^D=L_{\ba,i}.s_i^D+L_{\ba,i+1}.s_i^D=\begin{cases}
		-L_\ba &\text{if }i\CupConnect (i+1), \\
		L_\ba &\text{if }i\DCupConnect (i+1).
	\end{cases}
\] 
Note that $l_U.s_0^D=-l_{\tau_1U}$ for all $U\in\mathcal U_{\ba}$. Modifying the above calculation with the additional sign yields the claim for the $s_0^D$-action. 

\bigskip
\noindent {\it Case 3: Vertices i and i+1 are incident with one cup and one ray.} 

Assume that $i$ and $j$, $j<i$, are connected by an unmarked cup (if $i+1$ and $j$, $i+1<j$, are connected by a cup we only need to switch $i$ and $i+1$ in the proof below). Decompose $\mathcal U_\ba=\mathcal U_{\ba,i}\sqcup\mathcal U_{\ba,j}$ and $L_\ba=L_{\ba,i}+L_{\ba,j}$ as in {\it Case 2} above. Let ${\bf b}$ be the cup diagram different from $\ba$ appearing in the sum $\ba.s_i^D$ and decompose $\mathcal U_{\bf b}=\mathcal U_{{\bf b},i} \sqcup \mathcal U_{{\bf b},i+1}$ as well as $L_{\bf b}=L_{{\bf b},i}+L_{{\bf b},i+1}$. Since $\tau_iU=U$ for all $U\in\mathcal U_{\ba,j}$, we compute $L_{\ba,j}.s_i^D = L_{\ba,j}$ as in (\ref{eq:case_1_calculation}).

Moreover, we see that the assignment $U\mapsto\tau_iU$ defines an involution $\mathcal U_{\ba,i}\to \mathcal U_{{\bf b},i+1}$ and we have $\Lambda_\ba(\tau_iU)=\Lambda_{\bf b}(U)$ for all $U\in\mathcal U_{{\bf b},i+1}$. Thus, we can compute
\begin{eqnarray*}&
L_{\ba,i}.s_i^D=\sum_{U\in\mathcal U_{\ba,i}} (-1)^{\Lambda_\ba(U)}l_{\tau_iU}=\sum_{U\in\mathcal U_{{\bf b},i+1}} (-1)^{\Lambda_\ba(\tau_iU)}l_U 
=\sum_{U\in\mathcal U_{{\bf b},i+1}} (-1)^{\Lambda_{\bf b}(U)}l_U=L_{{\bf b},i+1},\quad&
\end{eqnarray*}
where the second equation is a simple change of the summation index using the bijection $\tau_i$.

We see that $\mathcal U_{\ba,i}=\mathcal U_{{\bf b},i}$ and $\Lambda_\ba(U)=\Lambda_{\bf b}(U)-1$ for all $U\in\mathcal U_{\ba,i}$ which implies $L_{\ba,i}=-L_{{\bf b},i}$ and 
\begin{align*}
L_\ba.s_i^D &= L_{\ba,j}.s_i^D+L_{\ba,i}.s_i^D = L_{\ba,j} + L_{{\bf b},i+1} = L_{\ba,j} + L_{{\bf b},i+1} + L_{\ba,i} - L_{\ba,i} \\
						 &= L_{\ba,j} + L_{{\bf b},i+1} + L_{\ba,i} + L_{{\bf b},i} = L_\ba + L_{\bf b}.
\end{align*}

If the cup connecting $i$ and $j$ is marked, we argue similarly. We omit to justify that $s_0^D$ acts as claimed in Table~\ref{tab:s_0^D_action}. 

\bigskip
\noindent {\it Case 4: Vertices i and i+1 are incident with two different cups.} 

Let $i$ and $j$ as well as $i+1$ and $k$ be connected by a cup. We decompose $\mathcal U_\ba=\mathcal U_{\ba,j,i+1} \sqcup \mathcal U_{\ba,j,k} \sqcup \mathcal U_{\ba,i,i+1} \sqcup \mathcal U_{\ba,i,k}$ and similarly $\mathcal U_{\bf b}=\mathcal U_{{\bf b},i,j} \sqcup \mathcal U_{{\bf b},i,k} \sqcup \mathcal U_{{\bf b},i+1,j} \sqcup \mathcal U_{{\bf b},i+1,k}$, where ${\bf b}$ is the cup diagram different from $\ba$ appearing in the sum $\ba.s_i^D$. Note that $\tau_i\colon\mathcal U_{\ba,j,k} \to \mathcal U_{\ba,j,k}$ and $\tau_i\colon\mathcal U_{\ba,i,i+1}\to\mathcal U_{\ba,i,i+1}$ are identities. Thus, we deduce
\[
L_{\ba,j,k}.s_i^D = \sum_{U\in\mathcal U_{\ba,j,k}} (-1)^{\Lambda_\ba(U)}l_{\tau_iU}=\sum_{U\in\mathcal U_{\ba,j,k}} (-1)^{\Lambda_\ba(U)}l_U = L_{\ba,j,k}
\]
and analogously $L_{\ba,i,i+1}.s_i^D=L_{\ba,i,i+1}$. Since $(-1)^{\chi(U)}=1$ in the above cases, these equations are also true if we act with $s_0^D$. 

It is easy to check (by going through the different cases in the table) that $\Lambda_\ba(\tau_iU)=\Lambda_{\bf b}(U)$ for all $U \in \mathcal U_{{\bf b},j,i}\cup U_{{\bf b},i+1,k}$ if we act with $s_i$, $1\leq i\leq k-1$. If we act with $s_0$, then we have $\Lambda_\ba(\tau_1U)=\Lambda_{\bf b}(U)\pm 1$. Moreover, $\tau_i\colon \mathcal U_{\ba,j,i+1}\to \mathcal U_{{\bf b},j,i}$ and $\tau_i\colon \mathcal U_{\ba,i,k}\to \mathcal U_{{\bf b},i+1,k}$ are self-inverse bijections. Hence, we compute
\begin{eqnarray*}
L_{\ba,j,i+1}.s_i^D=\sum_{U\in\mathcal U_{\ba,j,i+1}} (-1)^{\Lambda_\ba(U)}l_{\tau_iU}=\sum_{U\in\mathcal U_{{\bf b},j,i}} (-1)^{\Lambda_\ba(\tau_iU)}l_U 
										=\sum_{U\in\mathcal U_{{\bf b},j,i}} (-1)^{\Lambda_{\bf b}(U)}l_U=L_{{\bf b},j,i}.\quad\quad
\end{eqnarray*}
If we act with $s_0^D$, then the additional factor of $(-1)$ is killed by $(-1)^{\chi(U)}=-1$. Similarly, one checks that $L_{\ba,i,k}.s_i^D=L_{{\bf b},i+1,k}$.  

Next, we note the equalities of sets $\mathcal U_{\ba,i,k}=\mathcal U_{{\bf b},i,k}$ and $\mathcal U_{\ba,j,i+1}=\mathcal U_{{\bf b},j,i+1}$. Furthermore, we have $\Lambda_\ba(U)=\Lambda_{\bf b}(U)\pm 1$ for all $U\in \mathcal U_{\ba,i,k}\cup\mathcal U_{\ba,j,i+1}$. Again, this follows directly by looking at the table. Putting these two facts together we obtain $L_{{\bf b},i,k}=-L_{\ba,i,k}$ and $L_{{\bf b},j,i+1}=-L_{\ba,j,i+1}$. Thus,

\begin{eqnarray*}
\ba.s_i^D &=& L_{\ba,j,i+1}.s_i^D+L_{\ba,j,k}.s_i^D+L_{\ba,i,i+1}.s_i^D+L_{\ba,i,k}.s_i^D \\
						 &=& L_{{\bf b},j,i} + L_{\ba,j,k} + L_{\ba,i,i+1} + L_{{\bf b},i+1,k} \\
						 &= &L_{{\bf b},j,i} + L_{\ba,j,k} + L_{\ba,i,i+1} + L_{{\bf b},i+1,k} 
						  + L_{\ba,i,k} - L_{\ba,i,k} + L_{\ba,j,i+1} - L_{\ba,j,i+1} \\
						 &= &L_{{\bf b},j,i} + L_{\ba,j,k} + L_{\ba,i,i+1} + L_{{\bf b},i+1,k} 
						 + L_{\ba,i,k} + L_{{\bf b},i,k} + L_{\ba,j,i+1} + L_{{\bf b},j,i+1} 
						 = L_\ba + L_{\bf b}.\!\!\qedhere
\end{eqnarray*} 
\end{proof}

\subsection{Proof of Theorem~\ref{thmB}}
\label{sec:group_actions}
Let $x$ be a nilpotent element of Jordan type $\lambda$ contained in $\mathfrak{sp}_{2m}(\mathbb C)$ (resp.\ $\mathfrak{so}_{2m}(\mathbb C)$). Then the corresponding component group $A^x_{\mathrm{O}_{2m}}$ (resp.\ $A^x_{\mathrm{Sp}_{2m}}$) is isomorphic to $(\mathbb Z/2\mathbb Z)^t$, where $t$ is the number of distinct odd (resp.\ even) parts of the partition $\lambda$, see \cite[I.2.9]{Spa82}. Moreover, we identify $A^{2m-k,k}_{\mathrm{SO}_{2m}}\subseteq A^{2m-k,k}_{\mathrm{O}_{2m}}$ as the subgroup generated by products of an even number of generators of $A^{2m-k,k}_{\mathrm{O}_{2m}}$, see~\cite{Sho83}. This immediately yields the isomorphisms (\ref{compgroupsexplicit}) from the introduction.

Let $\alpha$ be a generator corresponding to a copy of $\mathbb Z/2\mathbb Z$ in $(\mathbb Z/2\mathbb Z)^t$. We have a left $A^{2m-k-1,k-1}_{\mathrm{Sp}_{2m}}$-action on $(\mathbb S^2)^m$, where $\alpha$ acts as the antipode on the first component, i.e.\ $\alpha.(x_1,x_2\ldots,x_m)=(-x_1,x_2\ldots,x_m)$. Thus, in homology we have
\[
\alpha.l_U\;=\;\begin{cases}
l_U &\text{if }1\notin U, \\
-l_U &\text{if }1\in U.
\end{cases}
\]
Note that the special orthogonal component group acts trivially, although it is not always trivial. 
The following proposition is now just a slight reformulation of Theorem~\ref{thmB} from the introduction. 
\begin{prop}
Given a cup diagram $\ba\in H_*(\SOfiber)$ and a generator $\alpha\in A^{2m-k-1,k-1}_{\mathrm{Sp}_{2m}}$ corresponding to a copy of $\mathbb Z/2\mathbb Z$ in $A^{2m-k-1,k-1}_{\mathrm{Sp}_{2m}}$, we claim that 
\[
\alpha.\ba:=\gamma_{2m-k,k}^{-1}\left(\alpha.\gamma_{2m-k,k}(\ba)\right)
\]
yields a well-defined right action of the component group on $H_*(\SOfiber)$ given as follows: 
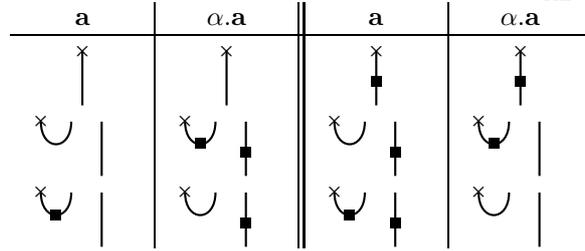
\centering
\begin{tabular}{ c | c || c | c}
$\ba$ & $\alpha.\ba$ & $\ba$ & $\alpha.\ba$ \\
\hline
\begin{tikzpicture}[scale=.8,baseline={(0,-.5)}]
\draw[thick] (0,0) -- +(0,-.9);
\node at (0,0) {$\times$};
\end{tikzpicture}
&\begin{tikzpicture}[scale=.8,baseline={(0,-.5)}]
\draw[thick] (0,0) -- +(0,-.9);
\node at (0,0) {$\times$};
\end{tikzpicture}
&
\begin{tikzpicture}[scale=.8,baseline={(0,-.5)}]
\draw[thick] (0,0) -- +(0,-.9);
\fill ([xshift=-2.5pt,yshift=-2.5pt]0,-.5) rectangle ++(5pt,5pt);
\node at (0,0) {$\times$};
\end{tikzpicture}
&\begin{tikzpicture}[scale=.8,baseline={(0,-.5)}]
\draw[thick] (0,0) -- +(0,-.9);
\fill ([xshift=-2.5pt,yshift=-2.5pt]0,-.5) rectangle ++(5pt,5pt);
\node at (0,0) {$\times$};
\end{tikzpicture}
\\
\begin{tikzpicture}[scale=.8,baseline={(0,-.5)}]
\draw[thick] (0,0) .. controls +(0,-.5) and +(0,-.5) .. +(.5,0);
\draw[thick] (1,0) -- +(0,-.9);
\node at (0,0) {$\times$};
\node at (1.5,0) {${}$};
\end{tikzpicture}
&
\begin{tikzpicture}[scale=.8,baseline={(0,-.5)}]
\draw[thick] (0,0) .. controls +(0,-.5) and +(0,-.5) .. +(.5,0);
\draw[thick] (1,0) -- +(0,-.9);
\fill ([xshift=-2.5pt,yshift=-2.5pt].25,-.365) rectangle ++(5pt,5pt);
\fill ([xshift=-2.5pt,yshift=-2.5pt]1,-.5) rectangle ++(5pt,5pt);
\node at (0,0) {$\times$};
\node at (1.5,0) {${}$};
\end{tikzpicture}
&
\begin{tikzpicture}[scale=.8,baseline={(0,-.5)}]
\draw[thick] (0,0) .. controls +(0,-.5) and +(0,-.5) .. +(.5,0);
\draw[thick] (1,0) -- +(0,-.9);
\fill ([xshift=-2.5pt,yshift=-2.5pt]1,-.5) rectangle ++(5pt,5pt);
\node at (0,0) {$\times$};
\node at (1.5,0) {${}$};
\end{tikzpicture}
&
\begin{tikzpicture}[scale=.8,baseline={(0,-.5)}]
\draw[thick] (0,0) .. controls +(0,-.5) and +(0,-.5) .. +(.5,0);
\draw[thick] (1,0) -- +(0,-.9);
\fill ([xshift=-2.5pt,yshift=-2.5pt].25,-.365) rectangle ++(5pt,5pt);
\node at (0,0) {$\times$};
\node at (1.5,0) {${}$};
\end{tikzpicture}
\\
\begin{tikzpicture}[scale=.8,baseline={(0,-.5)}]
\draw[thick] (0,0) .. controls +(0,-.5) and +(0,-.5) .. +(.5,0);
\draw[thick] (1,0) -- +(0,-.9);
\fill ([xshift=-2.5pt,yshift=-2.5pt].25,-.365) rectangle ++(5pt,5pt);
\node at (0,0) {$\times$};
\node at (1.5,0) {${}$};
\end{tikzpicture}
&
\begin{tikzpicture}[scale=.8,baseline={(0,-.5)}]
\draw[thick] (0,0) .. controls +(0,-.5) and +(0,-.5) .. +(.5,0);
\draw[thick] (1,0) -- +(0,-.9);
\fill ([xshift=-2.5pt,yshift=-2.5pt]1,-.5) rectangle ++(5pt,5pt);
\node at (0,0) {$\times$};
\node at (1.5,0) {${}$};
\end{tikzpicture}
&
\begin{tikzpicture}[scale=.8,baseline={(0,-.5)}]
\draw[thick] (0,0) .. controls +(0,-.5) and +(0,-.5) .. +(.5,0);
\draw[thick] (1,0) -- +(0,-.9);
\fill ([xshift=-2.5pt,yshift=-2.5pt].25,-.365) rectangle ++(5pt,5pt);
\fill ([xshift=-2.5pt,yshift=-2.5pt]1,-.5) rectangle ++(5pt,5pt);
\node at (0,0) {$\times$};
\node at (1.5,0) {${}$};
\end{tikzpicture}
&
\begin{tikzpicture}[scale=.8,baseline={(0,-.5)}]
\draw[thick] (0,0) .. controls +(0,-.5) and +(0,-.5) .. +(.5,0);
\draw[thick] (1,0) -- +(0,-.9);
\node at (0,0) {$\times$};
\node at (1.5,0) {${}$};
\end{tikzpicture}
\end{tabular}

\end{prop}
\begin{proof} 
If there is a ray connected to the first vertex there is nothing to show. So suppose there is a cup connecting the first vertex with some vertex $i>1$. We decompose $\mathcal U_\ba=\mathcal U_{\ba,1}\sqcup\mathcal U_{\ba,i}$, where $\mathcal U_{\ba,1}$ (resp.\ $\mathcal U_{\ba,i}$) consists of all sets $U\in\mathcal U_\ba$ containing $1$ (resp.\ $i$). Then we have
\[
L_\ba\;=\;\sum_{U\in\mathcal U_\ba} (-1)^{\Lambda_\ba(U)}l_U\; =\;\sum_{U\in\mathcal U_{\ba,1}} (-1)^{\Lambda_\ba(U)}l_U + \sum_{U\in\mathcal U_{\ba,i}} (-1)^{\Lambda_\ba(U)}l_U
\]
and acting by a generator $\alpha$ yields $\alpha.L_\ba=\sum_{U\in\mathcal U_{\ba,1}} (-1)^{\Lambda_\ba(U)+1}l_U + \sum_{U\in\mathcal U_{\ba,i}} (-1)^{\Lambda_\ba(U)}l_U=L_{\ba^-}$ which proves the claim, i.e.\ the geometric action can be described combinatorially as asserted.
\end{proof}


\section{A geometric construction of type B/C Specht modules for one-row bipartitions}\label{section:section3}

In this section we recall the notion of irreducible Specht modules for types B/C and D and provide a geometric construction of all Specht modules corresponding to pairs of $1$-row bipartitions on $H_*((\mathbb S^2)^m)$. This geometric construction will be the key ingredient in determining the decomposition of the Weyl group representation on $H_*(\SOfiber)$ into irreducibles in Section~\ref{section:section5}.

\subsection{Basic recollections on Specht modules}

We start by recalling the definition and basic properties of irreducible Specht modules for Weyl groups of types $B/C$ and $D$. Our exposition follows and summarizes~\cite{Can96}. 

A {\it partition} $\lambda=(\lambda_1,\ldots,\lambda_l)$ of $k$ is a sequence of $l$ weakly decreasing positive integers $\lambda_1\geq\lambda_2\geq\ldots\geq\lambda_l>0$ which sum up to $k$. We write $\vert\lambda\vert=k$ if $\lambda$ is a partition of $k$. Additionally, we also allow the empty partition $\emptyset$. We may depict the partition $\lambda$ as a {\it Young diagram}, i.e.\ a collection of $l$ left-aligned rows of square boxes, where the $i$th row has $\lambda_i$ boxes. 

\begin{defi}
A {\it bipartition} of $m$ is an ordered pair $(\lambda,\mu)$ of partitions such that $\vert\lambda\vert+\vert\mu\vert=m$.
\end{defi}

For the rest of this section we fix a bipartition $(\lambda,\mu)$ of $m$.

\begin{defi}
A {\it bitableau of shape} $(\lambda,\mu)$ is a filling of the boxes of the pair of Young diagrams with the numbers $1,\ldots,m$ such that each number appears exactly once. Additionally, each of the numbers comes equipped with a sign (a choice of either $+$ or $-$ for each number). A {\it standard bitableau} has only positive entries (no minus signs) and both underlying tableaux are standard tableaux in the usual sense, i.e.\ the entries of the boxes increase along rows and columns. 
\end{defi}

The group $\mathcal W_{C_m}$ acts (from the right) on the set of bitableaux of shape $(\lambda,\mu)$. By definition the generator $s_i^C$, $i\neq 0$, acts by substituting $i+1$ in place of $i$ at the position of $\pm i$ (the respective sign remains unchanged) and substitutes $i$ in place of $i+1$ at the position of $\pm (i+1)$ (again, the sign remains untouched). 
The generator $s_0^C$ changes the sign in front of $1$. Note that the action of $\mathcal W_{C_m}$ does not stabilize the subset of standard bitableaux.

 
\begin{defi}
Let $T$ be a bitableau of shape $(\lambda,\mu)$. A {\it row-permutation} (resp.\ {\it column-permutation}) of $T$ is an element $w\in \mathcal W_{C_m}$ which permutes the symbols in each row (resp.\ column) of $T$ and multiplies an arbitrary  number of symbols in the first (resp.\ second) tableau of $T$ with $-1$. The {\it row stabilizer} $R_T$ (resp.\ {\it column stabilizer} $C_T$) of $T$ is defined as the group of row-permutations (resp.\ column permutations).
\end{defi}




\begin{defi}
Two bitableaux $T,T^\prime$ of shape $(\lambda,\mu)$ are row equivalent if there exists a row permutation $w$ of $T$ such that $T^\prime=T.w$. The equivalence class containing a bitableau $T$ is denoted by $[T]$ and called a {\it bitabloid} of shape $(\lambda,\mu)$. 
\end{defi}

Let $M_{(\lambda,\mu)}$ be the $\mathbb C$-vector space with basis given by the bitabloids of shape $(\lambda,\mu)$. There is a well-defined action of $\mathcal W_{C_m}$ on $M_{(\lambda,\mu)}$, where $s_i^C$, $i\neq 0$, acts by setting $[T].s_i^C=[T.s_i^C]$ and $s_0^C$ acts as $[T].s_0^C=[T.s_0^C]$ if $1$ appears in the first tableau and $[T].s_0^C=-[T.s_0^C]$ if $1$ appears in the second tableau, see e.g.\ \cite[Section 1]{Can96}. 

Since $\mathcal W_{C_m}$ is isomorphic to the semidirect product $S_m\rtimes\left(\mathbb Z/2\mathbb Z\right)^m$, we can associate a unique pair $(\sigma,\alpha)$, where $\sigma\in S_m$ and $\alpha\in\left(\mathbb Z/2\mathbb Z\right)^m$, to any element $w\in\mathcal W_{C_m}$. We define $\text{sgn}(w):=\text{sgn}(\sigma)$, where $\text{sgn}(\sigma)$ is the usual signum of an element of the symmetric group. Moreover, we define $f(w)$ as the number of $(-1)$'s in the $m$-tuple $\alpha$.    

\begin{defi}
Given a bitableau $T$, we set $\kappa_T=\sum_{w\in C_T} (-1)^{f(w)}\text{sgn}(w)w$. The {\it bipolytabloid} $e_T$ associated with the tableau $T$ is defined as $e_T=[T].\kappa_T$.
\end{defi}

Following~\cite[Definition~3.3, Lemma~3.2]{Can96}, we define the Specht modules as follows:

\begin{defi}
Let $\lambda$ be a bipartition of $m$. The Specht module $V_{(\lambda,\mu)}$ is the submodule of $M_{(\lambda,\mu)}$ spanned by all bipolytabloids $e_T$.
\end{defi}

The importance of the Specht modules is justified by the following proposition which is \cite[Theorem 3.21]{Can96}, see also e.g.  \cite{Mac95}, \cite[\S 8.2]{Ser77}, \cite[Proposition~3]{MS16} for a classification of the irreducible modules using different methods.

\begin{prop}
The Specht modules $V_{(\lambda,\mu)}$, where $(\lambda,\mu)$ varies over all bipartitions of $m$, form a complete set of pairwise non-isomorphic irreducible $\mathcal W_{C_m}$-modules.
\end{prop}


We fix an irreducible representation $V_{(\lambda,\mu)}$ corresponding to $(\lambda,\mu)$. The restriction of $V_{(\lambda,\lambda)}$ to $\mathcal W_{D_m}$ decomposes into two non-isomorphic irreducible $\mathcal W_{D_m}$-modules of the same dimension which we denote by $V_\lambda^+$ and $V_\lambda^-$. If $\lambda\neq\mu$ the restriction of $V_{(\lambda,\mu)}$ to $\mathcal W_{D_m}$ remains irreducible and the restrictions of $V_{(\lambda,\mu)}$ and $V_{(\mu,\lambda)}$ to $\mathcal W_{D_m}$ are isomorphic. We denote this irreducible module by $V_{\{\lambda,\mu\}}$. Then we have, see {\cite{May75}}, or \cite[Theorem 17]{MS16} for a modern treatment:

\begin{prop}
The $\mathcal W_{D_m}$-representations $V_{\{\lambda,\mu\}}$, where $|\lambda|+|\mu|=m$, together with $V_\lambda^+$ and $V_\lambda^-$, where $|\lambda|=\frac{m}{2}$, form a complete set of pairwise non-isomorphic irreducible $\mathcal W_{D_m}$-representations.  
\end{prop} 

Now \cite[Theorem 4.2]{Can96} provides an important basis of the Specht modules as follows:

\begin{prop}\label{proposition:Specht_basis}
The bipolytabloids $e_T$, where $T$ varies over all standard bitableaux of shape $(\lambda,\mu)$, form a $\mathbb C$-basis of $V_{(\lambda,\mu)}$. 
\end{prop}

\subsection{A geometric construction of some Specht modules} 

The following fact is evident. 

\begin{lem}\label{lemma:bijection_line_diag_std_tableaux}
There is a bijection between the set of line diagrams on $m$ vertices with $l$ undotted rays and the set of standard $1$-row bitableaux of shape $((m-l),(l))$. This bijection is given by sending a line diagram $l_U$ to the unique standard bitableau $T_U$ whose second tableau is filled with the numbers in $U$.   
\end{lem} 


The group $\mathcal W_{C_m}$ acts (from the right) on $(\mathbb S^2)^m$ according to the following rules:
\begin{eqnarray*}
(x_1,x_2,\ldots,x_m).s^C_0&=&(-x_1,x_2,\ldots,x_m), \\
(x_1,\ldots,x_i,x_{i+1},\ldots,x_m).s^C_i&=&(x_1,\ldots,x_{i+1},x_i,\ldots,x_m), \hspace{2em} i\in\{1,\ldots,m-1\}.
\end{eqnarray*}
The induced action on homology is then given by 
\[
l_U.s^C_0\;=\;\begin{cases}
l_U &\text{if }1\notin U, \\
-l_U &\text{if }1\in U.
\end{cases}\,\hspace{1.6em}\text{and}\quad l_U.s^C_i=l_{\tau_iU}\,,\,\,i\in\{1,\ldots,m-1\}.
\]

The following proposition gives an elementary geometric construction of all $\mathcal W_{C_m}$-modules (and thus also of all $\mathcal W_{D_m}$-modules by restriction) corresponding to pairs of $1$-row partitions. It explicitly identifies the bipolytabloid basis from Proposition~\ref{proposition:Specht_basis} as the natural homology basis arising from the Cartesian product cell decomposition of $(\mathbb S^2)^m$.   
 
\begin{prop}\label{proposition:Specht_modules_as_homology}
The linear map $H_{2l}((\mathbb S^2)^m)\to V_{((m-l),(l))}$ defined by sending a line diagram $l_U$ to $e_{T_U}$ is an isomorphism of $\mathcal W_{C_m}$-modules.
\end{prop}


For the proof of Proposition~\ref{proposition:Specht_modules_as_homology} we first need the following technical lemma. 

\begin{lem}\label{lemma:row_equivalence}
If $T$ and $T^\prime$ are row equivalent $1$-row bitableaux, then $e_T=e_{T^\prime}$.
\end{lem}
\begin{proof}
By assumption we have $T^\prime=Tw$ for some row permutation $w$ of $T$. Let $v$ be any column permutation of $T$. Note that since $T$ is a one-row bitableau, the column permutation $v$ only changes some signs of the numbers contained in the second tableau and leaves the first tableau invariant. Furthermore, $w$ permutes the symbols in the second tableau (without any sign changes) and permutes as well as changes some signs in the first tableau. Thus, evidently, we have $T vw=T wv$. 

By the faithfulness of the action of $\mathcal W_{C_m}$ on the set of bitableaux we deduce that $vw=wv$ for all $\sigma\in C_T$. In particular, we have $C_T=w^{-1} C_Tw=C_{T w}$, where the second equality is even true for an arbitrary element $w$ of $\mathcal W_{C_m}$.  We obtain that
\[
\kappa_T=\sum_{v\in C_T} (-1)^{f(v)}\text{sgn}(v)v=\sum_{v\in C_{T w}} (-1)^{f(v)}\text{sgn}(v)v=\kappa_{T w}
\]  
and thus $e_T=[T].\kappa_T=[T].\kappa_{T w}=[T w].\kappa_{T w}=e_{T w}$, where the second last equality follows from the well-definedness of the action on row-equivalence classes.
\end{proof}






\begin{proof}[Proof of Proposition~\ref{proposition:Specht_modules_as_homology}]
By Proposition~\ref{proposition:Specht_basis} and Lemma~\ref{lemma:bijection_line_diag_std_tableaux} the map sends a basis to a basis. Thus, it suffices to check the $\mathcal W_{C_m}$-equivariance.

Let $i\neq 0$. If $i$ and $i+1$ are both dotted or both undotted in a given line diagram $l_U$, then $s_i^C$ acts as the identity. On the other hand, the numbers $i$ and $i+1$ are either both contained in the left or in the right tableau of $T_U$. Thus, $s_i^C$ acts as a row permutation on $T_U$ and $T_Us_i^C$ and $T_U$ are row equivalent. In particular, we get $e_{T_U}s_i^C=e_{T_Us_i^C}=e_{T_U}$, where the first equality follows from \cite[Lemma~3.2 (ii)]{Can96} and the second one from Lemma~\ref{lemma:row_equivalence}.  

If exactly one of the rays $i$ and $i+1$ is dotted, then $s_i^C$ moves a dot from $i$ to $i+1$ or the other way around. The action of $s_i^C$ swaps the numbers $i$ and $i+1$ which appear in two different constituents. After applying a row permutation to $T_Us_i^C$ we obtain the standard tableau with the same filling as $T_U$ except that precisely $i$ and $i+1$ have changed the respective tableau.

If the first line is dotted in $l_U$, i.e.\ if $1\notin U$, then the generator $s_0^C$ acts on $l_U$ as the identity. In this case, the number $1$ is contained in the first tableau of $T_U$ and we have $e_{T_U}s_0^C=e_{T_U}$ because we have $e_{T_U}s_0^C=(-1)^{i-1}e_{T_U}$ if $1$ occurs in the $i$th tableau of $T_U$, $i\in\{1,2\}$, see \cite[Lemma~3.16]{Can96}. Secondly, if $1\in U$, i.e.\ the first line is undotted in $l_U$, then $s_0^C$ acts as the identity with an additional minus sign in front of the diagram. In this case, the number $1$ is in the second tableau of $T_U$ and thus $e_{T_U}s_0^C=-e_{T_U}$ by \cite[Lemma~3.16]{Can96} as above. This finishes the proof of the $\mathcal W_{C_m}$-equivariance and the proposition follows.
\end{proof}


\section[sec]{\for{toc}{Diagrammatic description of the parabolic Hecke module}\except{toc}{Diagrammatic description of the parabolic Hecke module in types B/C \& D}}
\label{section:section4}

In this section we recall the definition of the parabolic Hecke module in types $B/C$ and $D$ and provide a diagrammatic description. This generalizes results obtained in \cite{LS13} for type $D$ to types $B/C$. In Section~\ref{section:section5} we compare the diagrammatic construction of the parabolic Hecke module and the Springer representation to prove Theorem~\ref{thm:induced_module}.

\subsection{Some Weyl group combinatorics} \label{subsec:Weyl group comb}

Let $\Lambda^m$ be the set of all ordered sequences of length $m$ consisting of symbols from the alphabet $\{\land,\lor\}$ and let $\Lambda^m_\mathrm{even}$ be the subset consisting of all sequences in which the symbol $\land$ occurs an even number of times. The Weyl group $\mathcal W_{D_m}$ of type $D_m$ acts transitively on $\Lambda^m_\mathrm{even}$ from the right, where $s_i^D$, $i\in\{1,\ldots,m-1\}$, permutes the symbols at positions $i$ and $i+1$ (counted from left to right) and $s_0^D$ exchanges a pair of symbols $\lor\lor$ at positions $1$ and $2$ for the pair $\land\land$ (and vice versa) and acts as the identity otherwise. The parabolic subgroup $\mathcal P_{D_m}$ of $\mathcal W_{D_m}$ generated by $s_1^D,\ldots,s_{m-1}^D$ stabilizes the sequence $\lambda^m_\lor\in\Lambda^m_\mathrm{even}$ consisting of $m$ successive $\lor$'s. This yields a bijection
\begin{equation} \label{eq:coset_identification_D}
\mathcal P^\mathrm{min}_{D_m}\xrightarrow\cong\Lambda^m_\mathrm{even}\,,\,\,w\mapsto \lambda^m_\lor.w\,,
\end{equation}
between $\Lambda^m_\mathrm{even}$ and the set $\mathcal P^\mathrm{min}_{D_m}$ of smallest representatives for the right cosets in $\mathcal P_{D_m}\setminus\mathcal W_{D_m}$ which consists of all $w\in\mathcal W_{D_m}$ satisfying $l_{D_m}(s_i^Dw)>l_{D_m}(w)$ for all $i\in\{1,\ldots,m-1\}$. Here, $l_{D_m}$ is the length function of the Coxeter group $\mathcal W_{D_m}$, see \cite{BB05}, \cite{Hum92} for details. 

\begin{rem} Given $w\in\mathcal P^\mathrm{min}_{D_m}$, we have $ws_i^D\in\mathcal P^\mathrm{min}_{D_m}$ if and only if the action of $s_i^D$ on the $\land\lor$-sequence associated with $w$ changes the sequence. If this is the case, we have $l_{D_m}(ws_i^D)>l_{D_m}(w)$ if and only if the action of $s_i^D$ on the sequence associated with $w$ either changes a pair $\land\lor$ to $\lor\land$ or $\lor\lor$ changes to $\land\land$, see \cite[Lemma~2.5]{ES15}. 
\end{rem}

Similarly, the Weyl group $\mathcal W_{C_{m-1}}$ of type $C_{m-1}$ acts transitively on $\Lambda^{m-1}$ from the right, where $s^C_i$, $i\in\{1,\ldots,m-2\}$, permutes the symbols at positions $i$ and $i+1$ and $s_0^C$ exchanges a $\land$ for a $\lor$ and vice versa at the first position. The parabolic subgroup $\mathcal P_{C_{m-1}}$ generated by $s_1^C,\ldots,s_{m-2}^C$ of $\mathcal W_{C_{m-1}}$ stabilizes the sequence $\lambda_\lor^{m-1}\in\Lambda^{m-1}$ which yields a bijection 
\begin{equation} \label{eq:coset_identification_C}
\mathcal P^\mathrm{min}_{C_{m-1}}\xrightarrow\cong\Lambda^{m-1}\,,\,\,w\mapsto\lambda^{m-1}_\lor.w\,,
\end{equation}
where $\mathcal P^\mathrm{min}_{C_{m-1}}$ denotes the set of smallest representatives for the right cosets in $\mathcal P_{C_{m-1}}\setminus\mathcal W_{C_{m-1}}$.
  
\begin{rem}
Given $w\in\mathcal P^\mathrm{min}_{C_{m-1}}$, we have $ws_i^C\in\mathcal P^\mathrm{min}_{C_{m-1}}$ if and only if the action of $s_i^C$ changes the sequence associated with $w$ in which case we have $l_{C_{m-1}}(ws_i^C)>l_{C_{m-1}}(w)$ if and only if either $\land\lor$ changes to $\lor\land$ or $\lor$ changes to $\land$.
\end{rem}

The map $(\,.\,)^\dagger\colon\Lambda^{m-1}\to\Lambda^m_\mathrm{even}$ given by 
\[
w \mapsto w^\dagger=\begin{cases}
\lor w &\text{ if }w\text{ contains an even number of }\land\text{'s},\\
\land w &\text{ if }w\text{ contains an odd number of }\land\text{'s},
\end{cases}
\]
is clearly a bijection which corresponds to a bijection between the sets of smallest coset representatives $\mathcal P^\mathrm{min}_{D_m}$ and $\mathcal P^\mathrm{min}_{C_{m-1}}$ via  (\ref{eq:coset_identification_D}) and (\ref{eq:coset_identification_C}), see also \cite[\S 9.7]{ES15}. 

\subsection{The parabolic Hecke module with distinguished bases} We continue by recalling some facts from \cite{KL79}, \cite{Soe97}. Let $\mathcal H_{D_m}$ be the Hecke algebra for $\mathcal W_{D_m}$ over the ring of formal Laurent polynomials $\mathcal L=\mathbb C[q,q^{-1}]$. It has an $\mathcal L$-basis $H_w$ indexed by the elements of the Weyl group $w\in\mathcal W_{D_m}$ which satisfy the algebra relations $H_vH_w=H_{vw}$ if $l_{D_m}(v)+l_{D_m}(w)=l_{D_m}(vw)$ and the quadratic relations $H_{s^D_i}^2=H_e+(q^{-1}-q)H_{s_i^D}$, $i\in\{0,\ldots,m-1\}$. As $\mathcal L$-algebra, $\mathcal H_{D_m}$ is generated  by the Kazhdan-Lusztig generators $C_{s^D_i}=H_{s^D_i}-q^{-1}H_e$, $i\in\{0,\ldots,m-1\}$. 

Let $\mathcal H_{\mathcal P_{D_m}}\subseteq \mathcal H_{\mathcal W_{D_m}}$ be the subalgebra generated by $C_{s_i^D}$, $i\in\{1,\ldots,m-1\}$. The algebra $\mathcal H_{\mathcal P_{D_m}}$ acts on $\mathcal L$ from the right by setting $1_\mathcal L.H_s=q^{-1}1_\mathcal L$ which allows us to define the tensor product $\mathcal M_{D_m}=\mathcal L\otimes_{\mathcal H_{\mathcal P_{D_m}}}\mathcal H_{\mathcal W_{D_m}}$ which is naturally a right $\mathcal H_{\mathcal W_{D_m}}$-module.

\begin{defi} This right $\mathcal H_{\mathcal W_{D_m}}$-module has a {\it standard basis} consisting of elements $M_w=1\otimes H_w$, where $w\in\mathcal P^\mathrm{min}_{D_m}$ and by \cite[\S 3]{Soe97} the $\mathcal H_{\mathcal W_{D_m}}$-action on $\mathcal M_{D_m}$ is given by
\begin{equation} \label{eq:explicit_module_structure}
M_w.C_{s_i^D}=\begin{cases}
M_{ws_i^D}-q^{-1}M_w 	&ws\in\mathcal P^\mathrm{min}_{D_m}\text{ and }l_{D_m}(ws_i^D)>l_{D_m}(w),\\
M_{ws_i^D}-qM_w				&ws\in\mathcal P^\mathrm{min}_{D_m}\text{ and }l_{D_m}(ws_i^D)<l_{D_m}(w),\\
0 								&ws\notin\mathcal P^\mathrm{min}_{D_m}.
\end{cases}
\end{equation}
\end{defi}

\begin{defi} The $\mathbb C$-linear involution $\overline{(\,.\,)}\colon\mathcal H_{\mathcal W_{D_m}}\to\mathcal H_{\mathcal W_{D_m}}$, $H\mapsto\overline{H}$, given by $q\mapsto q^{-1}$ and $H_w\mapsto H_{w^{-1}}^{-1}$ induces an involution on $\mathcal M_{D_m}$. For all $w\in\mathcal P^\mathrm{min}_{D_m}$ there exists precisely one element $\underline{M}_w\in\mathcal M_{D_m}$ such that $\overline{\underline{M}_w}=\underline{M}_w$ and $\underline{M}_w=M_w+\sum_{v\neq w} q^{-1}\mathbb Z[q^{-1}]M_v$. The set $\{\underline{M}_w\mid w\in\mathcal P^\mathrm{min}_{D_m}\}$ is the {\it Kazhdan-Lusztig basis} of $\mathcal M_{D_m}$ (see \cite[Theorem 3.5]{Soe97}). The {\it Kazhdan-Lusztig polynomials} $m_{v,w}\in\mathbb Z[q,q^{-1}]$ are defined implicitly by $\underline{M}_w=\sum_v m_{v,w}M_v$.
\end{defi}

\begin{rem} \label{rem:KL_remark}
In \cite[Theorem 2.10]{LS13} the parabolic Kazhdan-Lusztig polynomials were explicitly computed for the parabolic type $(D_m,A_{m-1})$. Moreover, by \cite[\S 9.7]{ES15}, for any given $v,w\in\mathcal P^\mathrm{min}_{C_{m-1}}$ we have $m^C_{v,w}=m^D_{v^\dagger,w^\dagger}$, i.e.\ the parabolic Kazhdan-Lusztig polynomials of type $(C_{m-1},A_{m-2})$ are determined by the Kazhdan-Lusztig polynomials of type $(D_{m-1},A_{m-1})$.
\end{rem}

\begin{lem} \label{lem:relating_C_D_module_action}
Let $(\,.\,)^\dagger\colon\mathcal M_{C_{m-1}} \to \mathcal M_{D_m}$ be the linear isomorphism sending $M_w$ to $M_{w^\dagger}$. Then the linear isomorphism $(\,.\,)^\dagger\colon\mathcal M_{C_{n-1}} \to \mathcal M_{D_n}$ sends a Kazhdan-Lusztig basis element $\underline{M}_w$ to the Kazhdan-Lusztig element $\underline{M}_{w^\dagger}$ and we have for $i\in\{1,\ldots,m-2\}$,  commutative diagrams
\[
\begin{xy}
	\xymatrix{
\mathcal M_{C_{m-1}} \ar[r]_\cong^{(\,.\,)^\dagger} \ar[d]_{\cdot C_{s_0^C}} & \mathcal M_{D_m} \ar[d]^{\cdot (C_{s_0^D}+C_{s_1^D})} \\ 
\mathcal M_{C_{m-1}} \ar[r]_\cong^{(\,.\,)^\dagger} & \mathcal M_{D_m} 
	}
\end{xy}
\hspace{4em}
\begin{xy}
	\xymatrix{
\mathcal M_{C_{m-1}} \ar[r]_\cong^{(\,.\,)^\dagger} \ar[d]_{\cdot C_{s_i^C}} & \mathcal M_{D_m} \ar[d]^{\cdot C_{s_{i+1}^D}} \\ 
\mathcal M_{C_{m-1}} \ar[r]_\cong^{(\,.\,)^\dagger} & \mathcal M_{D_m} 
	}
\end{xy}
\]
where the vertical maps are the actions with the indicated element of the respective Hecke algebra.
\end{lem}
\begin{proof}
The first part follows directly from the definition of the Kazhdan-Lusztig polynomials and Remark~\ref{rem:KL_remark}. We only prove the commutativity of the right square and omit the left one. We check the commutativity on standard basis vectors. Using (\ref{eq:explicit_module_structure}) one easily checks that the composition 
\[
\mathcal M_{C_{m-1}}\xrightarrow{\cdot C_{s_i^C}}\mathcal M_{C_{m-1}}\xrightarrow{(\,.\,)^\dagger}\mathcal M_{D_m}
\]
is given by
\[
M_w\longmapsto \begin{cases}
M_{(ws^C_i)^\dagger}-q^{-1}M_{w^\dagger} 	&ws^C_0\in\mathcal P^\mathrm{min}_{C_{m-1}}\text{ and }l(ws^C_i)>l(w),\\
M_{(ws^C_i)^\dagger}-qM_{w^\dagger} 				&ws^C_0\in\mathcal P^\mathrm{min}_{C_{m-1}}\text{ and }l(ws^C_i)<l(w),\\
0 								&ws^C_i\notin\mathcal P^\mathrm{min}_{C_{m-1}},
\end{cases}
\]
for $i\in\{1,\ldots,m-2\}$, and the composition
\[
\mathcal M_{C_{m-1}}\xrightarrow{(\,.\,)^\dagger}\mathcal M_{D_m}\xrightarrow{\cdot C_{s_{i+1}^D}}\mathcal M_{D_m}
\]
is given by
\[
M_w\longmapsto \begin{cases}
M_{w^\dagger s^D_{i+1}}-q^{-1}M_{w^\dagger} 	&w^\dagger s^D_{i+1}\in\mathcal P^\mathrm{min}_{D_m}\text{ and }l(ws^D_{i+1})>l(w),\\
M_{w^\dagger s^D_{i+1}}-qM_{w^\dagger} 				&w^\dagger s^D_{i+1}\in\mathcal P^\mathrm{min}_{D_m}\text{ and }l(ws^D_{i+1})<l(w),\\
0 								&w^\dagger s^D_{i+1}\notin\mathcal P^\mathrm{min}_{D_m}.
\end{cases}
\]
By comparing the two results using the combinatorial description, the claim follows.
\end{proof}

\subsection{Diagrammatical reformulation}
\label{sec:diagrammatic_Hecke_module}
Henceforth, we identify $\mathcal M_{D_m}$ with $\mathcal M^\textrm{comb}_m=\mathcal L[C_\mathrm{KL}(m)]$ (as $\mathcal L$-module) by sending $\underline{M}_w$ to the the cup diagram $\ba_w$ associated with $w\in\mathcal P_{D_m}^\mathrm{min}$ which is constructed as follows: Firstly, connect all neighbored $\lor\land$ symbols by an unmarked cup successively, thereby ignoring all pairs which are already connected. Secondly, connect the remaining neighbored $\land$ pairwise by a marked cup starting from the left. Finally, attach rays to the symbols which are not yet connected to a cup, where a ray is marked if and only if it is connected to a $\land$. The assignment $w\mapsto\ba_w$ defines a bijection between $\Lambda^m_\mathrm{even}$ and $C_\mathrm{KL}(m)$, see \cite[Section 5.2]{LS13}.

\begin{prop}[Hecke module in Type D] \label{prop:diagrammatic_Hecke_actionI}
The right action of the Hecke algebra $\mathcal H_{D_m}$ of type $D$ on $\mathcal M_{D_m}$ can be described via the identification $\mathcal M_{D_m}\cong\mathcal M^\textrm{comb}_m$ as follows: Representing the Kazhdan-Lusztig generators $C_{s_0^D}$, $C_{s_i^D}\in \mathcal H_{D_m}$, where $i\in\{1,\ldots,m-1\}$ pictorially as 
\[
C_{s_0^D}:=\begin{tikzpicture}[thick, scale=.8, baseline={(0,-.5)}]
\draw[thick] (0,0) .. controls +(0,-.5) and +(0,-.5) .. +(.5,0);
\fill ([xshift=-2.5pt,yshift=-2.5pt].25,-.365) rectangle ++(5pt,5pt);
\draw[thick] (0,-1) .. controls +(0,.5) and +(0,.5) .. +(.5,0);
\fill ([xshift=-2.5pt,yshift=-2.5pt].25,-.635) rectangle ++(5pt,5pt);
\draw[thick] (1,0) -- +(0,-1);
\draw[thick] (1.5,0) -- +(0,-1);
\draw[thick] (2.5,0) -- +(0,-1);

\node at (2,-.5) {\dots};
\begin{footnotesize}
\node at (0,.2) {1};
\node at (.5,.2) {2};
\end{footnotesize}
\end{tikzpicture}
\hspace{5em}
C_{s_i^D}:=\begin{tikzpicture}[thick, scale=.8, baseline={(0,-.5)}]
\draw[thick] (0,0) -- +(0,-1);
\draw[thick] (.5,0) -- +(0,-1);
\draw[thick] (1.5,0) .. controls +(0,-.5) and +(0,-.5) .. +(.5,0);
\draw[thick] (1.5,-1) .. controls +(0,.5) and +(0,.5) .. +(.5,0);
\draw[thick] (3,0) -- +(0,-1);

\node at (1,-.5) {\dots};
\node at (2.5,-.5) {\dots};

\begin{footnotesize}
\node at (1.4,.22) {i};
\node at (2,.2) {i+1};
\end{footnotesize}
\end{tikzpicture}
\]
then we can compute $\ba.C_{s_i^D}$, where $\ba\in C_\mathrm{KL}(m)$ is a cup diagram, by putting the picture corresponding to the generator $C_{s_i^D}$ on top of $\ba$ and applying the relations 
\begin{equation}
\label{firstrow}
\begin{tikzpicture}[thick, scale=.8]
\draw[dotted] (-.25,-.5) circle(16pt);
\draw[dotted] (3.8,-.5) circle(16pt);
\draw[thick] (-.25,-.5) circle(8pt);

\node at (1.8,-.5) {$= -(q+q^{-1})\cdot$};
\end{tikzpicture}
\hspace{3em}
\begin{tikzpicture}[thick, scale=.8]
\draw[dotted] (-.25,-.5) circle(16pt);
\draw[dotted] (1.9,-.5) circle(16pt);

\draw[thick] (-.75,-.5) -- +(1,0);
\draw[thick] (1.4,-.5) -- +(1,0);
\fill ([xshift=-2.5pt,yshift=-2.5pt]-.05,-.5) rectangle ++(5pt,5pt);
\fill ([xshift=-2.5pt,yshift=-2.5pt]-.45,-.5) rectangle ++(5pt,5pt);

\node at (.8,-.5) {$=$};
\end{tikzpicture}
\hspace{3em}
\begin{tikzpicture}[thick, scale=.8]
\draw[dotted] (-.25,-.5) circle(16pt);
\draw[thick] (-.25,-.5) circle(8pt);

\fill ([xshift=-2.5pt,yshift=-2.5pt]-.25,-.225) rectangle ++(5pt,5pt);

\node at (.8,-.5) {$=\,0$};
\end{tikzpicture}
\end{equation}
\begin{equation}
\label{secondrow}
\begin{tikzpicture}
\draw[dotted] (0,0) .. controls +(0,1) and +(0,1) .. +(1,0);
\draw[thick] (0.25,0) .. controls +(0,.5) and +(0,.5) .. +(.5,0);
\node at (1.5,.25) {$=0$};
\draw[dotted] (0,0) -- +(1,0);
\end{tikzpicture}
\hspace{4em}
\begin{tikzpicture}
\draw[dotted] (0,0) .. controls +(0,1) and +(0,1) .. +(1,0);
\draw[thick] (0.25,0) .. controls +(0,.5) and +(0,.5) .. +(.5,0);
\fill ([xshift=-2.5pt,yshift=-2.5pt].5,.365) rectangle ++(5pt,5pt);
\node at (1.5,.25) {$=1$};
\draw[dotted] (0,0) -- +(1,0);
\end{tikzpicture}
\end{equation}
which yields a new cup diagram in $C_\mathrm{KL}(m)$ equal to $\ba.C_{s_i^D}$. 
\end{prop}
The relations in \eqref{secondrow} only hold for connected components with both endpoints at the bottom of the diagram. The second relation hereby means that we can simply remove such a component if it is marked, whereas the first relation always kills the entire diagram. 

\begin{prop}[Hecke module in Type C]\label{prop:diagrammatic_Hecke_actionII}
The right action of the Hecke algebra $\mathcal H_{C_{m-1}}$ of type $C_{m-1}$ on $\mathcal M_{C_{m-1}}$ can be described via the identification $\mathcal M_{C_{m-1}}\cong\mathcal M^\textrm{comb}_m$ as follows: Given a cup diagram $\ba\in C_\mathrm{KL}(m)$ and a generator $C_{s_i^C}\in \mathcal H_C$ which we represent pictorially as 
\[
C_{s_0^C}:=\begin{tikzpicture}[thick, scale=.8, baseline={(0,-.5)}]
\draw[thick] (0,0) .. controls +(0,-.5) and +(0,-.5) .. +(.5,0);
\draw[thick] (0,-1) .. controls +(0,.5) and +(0,.5) .. +(.5,0);
\draw[thick] (1,0) -- +(0,-1);
\draw[thick] (1.5,0) -- +(0,-1);
\draw[thick] (2.5,0) -- +(0,-1);

\node at (2,-.5) {\dots};
\begin{footnotesize}
\node at (0,.2) {1};
\node at (.5,.2) {2};
\end{footnotesize}
\end{tikzpicture}
\,\,
+
\begin{tikzpicture}[thick, scale=.8, baseline={(0,-.5)}]
\draw[thick] (0,0) .. controls +(0,-.5) and +(0,-.5) .. +(.5,0);
\fill ([xshift=-2.5pt,yshift=-2.5pt].25,-.365) rectangle ++(5pt,5pt);
\draw[thick] (0,-1) .. controls +(0,.5) and +(0,.5) .. +(.5,0);
\fill ([xshift=-2.5pt,yshift=-2.5pt].25,-.635) rectangle ++(5pt,5pt);
\draw[thick] (1,0) -- +(0,-1);
\draw[thick] (1.5,0) -- +(0,-1);
\draw[thick] (2.5,0) -- +(0,-1);

\node at (2,-.5) {\dots};
\begin{footnotesize}
\node at (0,.2) {1};
\node at (.5,.2) {2};
\end{footnotesize}
\end{tikzpicture}
\hspace{3.5em}
C_{s_i^C}:=\begin{tikzpicture}[thick, scale=.8, baseline={(0,-.5)}]
\draw[thick] (0,0) -- +(0,-1);
\draw[thick] (.5,0) -- +(0,-1);
\draw[thick] (1.5,0) .. controls +(0,-.5) and +(0,-.5) .. +(.5,0);
\draw[thick] (1.5,-1) .. controls +(0,.5) and +(0,.5) .. +(.5,0);
\draw[thick] (3,0) -- +(0,-1);

\node at (1,-.5) {\dots};
\node at (2.5,-.5) {\dots};

\begin{footnotesize}
\node at (1.3,.22) {i+1};
\node at (2.1,.2) {i+2};
\end{footnotesize}
\end{tikzpicture}
\hspace{2em}
\text{with }i\in\{1,\ldots,m-2\},
\]
then $\ba.C_{s_i^C}$ can be computed by putting the picture corresponding to the generator $C_{s_i^C}$ on top of $\ba$ and simplifying the resulting pictures using the relations \eqref{firstrow} and \eqref{secondrow} again.
\end{prop}
\begin{proof}[Proof of Propositions~\ref{prop:diagrammatic_Hecke_actionI} and~\ref{prop:diagrammatic_Hecke_actionII}.]
Proposition~\ref{prop:diagrammatic_Hecke_actionI} is the analog of \cite[Theorem 4.17]{LS13} for the induced trivial module and can be proven by closely following the argument for the induced sign module. Proposition~\ref{prop:diagrammatic_Hecke_actionII} follows then together with Lemma~\ref{lem:relating_C_D_module_action}.
\end{proof}

\begin{lem} \label{lem:filtration}
There exists a filtration 
\begin{equation} \label{eq:filtration}
\{0\}\subseteq\mathcal M^\textrm{comb}_{m,\lfloor\frac{m}{2}\rfloor}\subseteq\ldots\subset\mathcal M^\textrm{comb}_{m,l}\subseteq\ldots\subseteq\mathcal M^\textrm{comb}_{m,0}=\mathcal M^\textrm{comb}_m
\end{equation}
of both right $\mathcal H_{D_m}$- and $\mathcal H_{C_{m-1}}$-modules, where $\mathcal M^\textrm{comb}_{m,l}$ is the $\mathcal L$-span of all cup diagrams in $C_\mathrm{KL}(m)$ with $l$ or more cups, $0\leq l\leq\lfloor\frac{m}{2}\rfloor$. 
\end{lem}
\begin{proof}
It follows from Propositions~\ref{prop:diagrammatic_Hecke_actionI} and~\ref{prop:diagrammatic_Hecke_actionII} that acting with a Kazhdan-Lusztig generator on a cup diagram with $l$ cups always yields a cup diagram with $l$ or more cups (the number of cups changes if and only if we use the second relation from \eqref{secondrow}). In particular, the filtration (\ref{eq:filtration}) of $\mathcal L$-modules is actually a filtration of right $\mathcal H_{D_m}$- and $\mathcal H_{C_{m-1}}$-modules. 
\end{proof}

\begin{rem}
The filtration \eqref{eq:filtration} can be identified with the respective filtrations of $\mathcal M_{D_m}$ and $\mathcal M_{C_{m-1}}$ by Kazhdan-Lusztig cells with subquotients $\mathcal M^\textrm{comb}_{m,l}/\mathcal M^\textrm{comb}_{m,l-1}$ the Kazhdan-Lusztig cell modules.  
\end{rem}

\section[sec]{\for{toc}{Connection to classical Springer theory}\except{toc}{Connection to classical Springer theory and the proof of Theorem~\ref{thm:induced_module}}}
\label{section:section5}

In this section we identify our representations in each homology degree for every two-row Springer fiber by providing an explicit decomposition into irreducible Specht modules. We use the results from Section~\ref{section:section4} but specialized at $q=1$ in which case the Hecke algebra becomes the group algebra and the parabolic Hecke module becomes the induced trivial representation. 

\subsection{Proof of Theorem~\ref{thm:induced_module}}\label{subsection:proof_induced_module}

By \cite[Lemma~5.19]{ES12}, the induced trivial representation of type $D$ decomposes as multiplicity free sum of irreducibles labeled by pairs of one-row partitions. If $m$ is odd, it follows that all of these simples must occur in the induced trivial representation in order for the dimension to add up to $2^{m-1}$ (which is the dimension of the induced trivial module because the $\{\land\lor\}$-sequences in $\Lambda^m_\mathrm{even}$ label a basis). Thus, we have a decomposition
\begin{equation} \label{eq:decomposition_into_irreps_odd}
\mathbb C\otimes_{\mathbb C[S_m]}\mathbb C[\mathcal W_{D_m}]\cong\bigoplus_{l=0}^{\frac{m-1}{2}} V_{\{(m-l),(l)\}}
\end{equation}
(cf.\ \cite[Lemma~5.19]{ES12} and \cite[Remark 5.21]{ES12}). 

If $m$ is even, a similar argument shows that all simples must occur exactly once, except one of the modules $V_{(\frac{m}{2})}^+$ or $V_{(\frac{m}{2})}^-$ is missing (which one depends on the two possible choices of maximal parabolic subgroups of $\mathcal W_{D_m}$) which yields a decomposition
\begin{equation} \label{eq:decomposition_into_irreps_even}
\mathbb C\otimes_{\mathbb C[S_m]}\mathbb C[\mathcal W_{D_m}]\cong V^+_{(\frac{m}{2})} \oplus\bigoplus_{l=0}^{\frac{m}{2}-1} V_{\{(m-l),(l)\}}
\end{equation}

In particular, it follows that the filtration (\ref{eq:filtration}) is already a composition series, i.e.\ a filtration which cannot be refined. Hence, the subquotients $\mathcal M^\textrm{comb}_{m,l}/\mathcal M^\textrm{comb}_{m,l-1}$ are irreducible and isomorphic to the summands in the decompositions (\ref{eq:decomposition_into_irreps_odd}) and (\ref{eq:decomposition_into_irreps_even}). 
Note that the subquotients $\mathcal M^\textrm{comb}_{m,l}/\mathcal M^\textrm{comb}_{m,l-1}$ of this filtration are isomorphic to the representations $H_{2l}(\SOfiberequalsize)$ constructed in Theorem~\ref{thm:skein_calculus}. This follows easily by using the relation $H_{s_i^D}=C_{s_i^D}+q^{-1}H_e$ and comparing the local relations for $q=1$. In particular, $H_*\left(\SOfiberequalsize\right)$ is isomorphic to the induced trivial module because both have the same irreducible constituents. This proves Theorem~\ref{thm:induced_module} for type $D$.

For the type $C$ case, note that by Proposition~\ref{prop:diagrammatic_Hecke_actionII} (specialized at $q=1$) the $\mathcal W_{C_{m-1}}$-representation $\mathcal M^\textrm{comb}_m$ is isomorphic to $\mathbb C\otimes_{\mathbb C[S_{m-1}]}\mathbb C[\mathcal W_{C_{m-1}}]$. Thus, it suffices to check that $\mathcal M^\textrm{comb}_m$ is isomorphic to $H_*(\SOfiberequalsize)$ as $\mathcal W_{C_{m-1}}$-module. Since $\mathcal M^\textrm{comb}_m$ is isomorphic to $H_*(\SOfiberequalsize)$ as $\mathcal W_{D_m}$-module by Theorem~\ref{thm:explicit_identification_D}, it suffices to check that the $\mathcal W_{C_{m-1}}$-action on $\mathcal M^\textrm{comb}_m$ is obtained by restricting the $\mathcal W_{D_m}$-action to the subgroup generated by $s_0^Ds_1^D,s_2^D,\ldots,s^D_{m-1}$. But this is clear from Propositions~\ref{prop:diagrammatic_Hecke_actionI} and~\ref{prop:diagrammatic_Hecke_actionII}.

\subsection{Explicit identifications in type D}

In the following we abbreviate $M=\{1,\ldots,m\}$. If $m$ is even, the Specht module $V_{((\frac{m}{2}),(\frac{m}{2}))}$ decomposes as a direct sum of two non-isomorphic irreducible $\mathcal W_{D_m}$-modules. In order to describe this decomposition note that the group $\mathbb Z/2\mathbb Z$ acts on the line diagram basis of $H_m((\mathbb S^2)^m)$ by sending $l_U$ to $l_{M\setminus U}$. Let $\mathcal R_m$ be a set of representatives of the orbits under this action. Then one of the irreducible summands of the $\mathcal W_{D_m}$-module $H_m((\mathbb S^2)^m)$ has a basis given by all sums $l_U+l_{M\setminus U}$, $U\in\mathcal R_m$, and the second one has a basis given by all sums $l_U-l_{M\setminus U}$, $U\in\mathcal R_m$. We write $V^+_{(\frac{m}{2})}$ (resp.\ $V^-_{(\frac{m}{2})}$) to denote these irreducible modules.

\begin{thm} \label{thm:explicit_identification_D}
We have the following isomorphisms of $\mathcal W_{D_m}$-modules:
\begin{enumerate}
\item\label{item:d1} If $k$ is odd, then we have, for all $l\in\{0,\ldots,\lfloor\frac{k}{2}\rfloor\}$, isomorphisms
\[
H_{2l}\left(\SOfiber\right)\cong V_{\{(m-l),(l)\}}.
\]
\item\label{item:d2} 
If $k$ is even (which implies $k=m$), then we have, for all $l\in\{0,\ldots,\frac{m}{2}-1\}$,  isomorphisms
\[
H_m(\SOfiberequalsize)\cong V_{(\frac{m}{2})}^{+} \hspace{1.6em}\text{and}\hspace{1.6em} H_{2l}\left(\SOfiberequalsize\right)\cong V_{\{(m-l),(l)\}}. 
\]
\end{enumerate}
The isomorphisms in (\ref{item:d1}) and (\ref{item:d2}) are explicitly given by 
\begin{equation}\label{eq:explicit_map}
\ba\mapsto\sum_{U\in\mathcal U_\ba}(-1)^{\Lambda_\ba(U)}e_{T_U}.
\end{equation}
\end{thm}
\begin{proof}
If $k$ is odd, then the injection 
\begin{equation} \label{eq:composition_yielding_iso_d1}
H_{2l}(\SOfiber)\hookrightarrow H_{2l}((\mathbb S^2)^m)\xrightarrow\cong V_{((m-l),(l))}
\end{equation}
obtained by composing the maps of $\mathcal W_{D_m}$-modules from Proposition~\ref{prop:explicit_isom} and~\ref{proposition:Specht_modules_as_homology} is given by 
\begin{equation} \label{eq:assignment_yielding_iso_d1}
\ba\mapsto L_\ba=\sum_{U\in\mathcal U_\ba}(-1)^{\Lambda_\ba(U)}l_U\mapsto\sum_{U\in\mathcal U_\ba}(-1)^{\Lambda_\ba(U)}e_{T_U}.
\end{equation}
This is an isomorphism for all $l\in\{0,\ldots,\lfloor\frac{k}{2}\rfloor\}$ because $H_{2l}(\SOfiber)\neq 0$ and hence the image of the $\mathcal W_{D_m}$-equivariant injection (\ref{eq:composition_yielding_iso_d1}) must be equal to the irreducible module $V_{((m-l),(l))}$. 

If $k$ is even, the same argument is true except for the top non-vanishing homology. Then, analogous to (\ref{eq:composition_yielding_iso_d1}), we have an injection  
\begin{equation}\label{eq:injection_d}
H_m(\SOfiberequalsize)\hookrightarrow H_m((\mathbb S^2)^m)\xrightarrow\cong V^+_{(\frac{m}{2})} \oplus V^-_{(\frac{m}{2})}.
\end{equation}
given by the same assignment as in (\ref{eq:assignment_yielding_iso_d1}).  

It remains to prove that the image of this map is contained in $V^+_{(\frac{m}{2})}$. Let $\ba$ be a cup diagram viewed as a basis vector of $H_m(\SOfiberequalsize)$. Since $m$ is even, $\ba$ consists of cups only. Given $U\in \mathcal U_{\ba}$, then $M\setminus U$ is also contained in $\mathcal U_{\ba}$. In fact, in comparison to $U$, $M\setminus U$ corresponds to choosing exactly the opposite endpoints of all cups in $\ba$. Thus, we can write
\[
L_\ba=\sum_{U\in\mathcal U_\ba}(-1)^{\Lambda_\ba(U)}l_U = \sum_{U\in\mathcal U'_\ba}\left((-1)^{\Lambda_\ba(U)}l_U + (-1)^{\Lambda_\ba(M\setminus U)}l_{M\setminus U}\right)= \sum_{U\in\mathcal U'_\ba}(-1)^{\Lambda_\ba(U)}\left(l_U + l_{M\setminus U}\right),
\]   
where $\mathcal U'_\ba\subseteq\mathcal U_\ba$ consists of all subsets which are also contained in $\mathcal R_m$. The condition that the number of marked rays plus unmarked cups must be even implies that we have an even number of unmarked cups. Note that $(-1)^{\Lambda_\ba(U)}= (-1)^{\Lambda_\ba(M\setminus U)}$ because opposite endpoints of unmarked cups contribute different signs to the respective exponent. Thus, the injective map (\ref{eq:injection_d}) has its image in the irreducible module $V^+_{(\frac{m}{2})}$.
\end{proof}


\subsection{Explicit identifications in type C}

If $m\neq k$ or $m=k$ is odd, then there is a non-trivial component group action on $H_*(\SOfiber)$. This action on $H_*(\SOfiber)$ restricts to an action on the set $C_\mathrm{KL}^{\leq\lfloor\frac{k}{2}\rfloor}(m)$. Each orbit consists of one or two elements. Let $\mathcal R$ be a chosen set of representatives of the orbits and let $\mathcal R_{\bf 1}$ (resp.\ $\mathcal R_{\bf 2}$) denote the set of representatives contained in a one-element orbit (resp.\ two-element orbit). We thus have a decomposition $\mathcal R=\mathcal R_{\bf 1}\cup\mathcal R_{\bf 2}$. Given a cup diagram $\ba\in\mathcal R_{\bf 2}$, we write $\ba^-$ to denote the second element in its orbit. 

An irreducible character $\phi$ of $(\mathbb Z/2\mathbb Z)^t$ is completely determined by the values $\phi(\alpha_j)\in\{\pm 1\}$, where $\alpha_j$ is the generator of the $j$th copy of $\mathbb Z/2\mathbb Z$. This yields a bijection between the set of irreducible characters of $(\mathbb Z/2\mathbb Z)^t$ and the set of all ordered $\{+,-\}$-sequences of length $t$. We write $\mathbb C_\epsilon$ to denote the irreducible module for $\epsilon$. Let $H_*^\epsilon(\SOfiber)$ be the isotypic subspace of $H_*(\SOfiber)$ corresponding to the irreducible representation $\mathbb C_\epsilon$ of $(\mathbb Z/2\mathbb Z)^t$ labeled by the tuple $\epsilon$. 

The following result describes the decomposition of the homology into irreducibles with respect to the component group action including an explicit basis of each isotypic subspace.

\begin{prop} \label{prop:basis_isotypic_comp}
We have the following decompositions into isotypic components:
\begin{enumerate}
\item If $m$ is odd, we have for all $l\in\{0,\ldots,\lfloor\frac{m}{2}\rfloor\}$ isomorphisms of $\mathbb Z/2\mathbb Z$-modules
\[
H_{2l}(\SOfiberequalsize)\cong\underbrace{\mathbb C_{(1)}\oplus\ldots\oplus\mathbb C_{(1)}}_{\binom{m-1}{l}\text{ copies}}\oplus\underbrace{\mathbb C_{(-1)}\oplus\ldots\oplus\mathbb C_{(-1)}}_{\binom{m-1}{l-1}\text{ copies}}.
\]
\item If $m\neq k$, we have for all $l\in\{0,\ldots,\frac{k-1}{2}\}$ isomorphisms of $\left(\mathbb Z/2\mathbb Z\right)^2$-modules
\[
H_{2l}(\SOfiber)\cong\mathbb C_{(1,1)}\oplus\ldots\oplus\mathbb C_{(1,1)}\oplus\mathbb C_{(-1,-1)}\oplus\ldots\oplus\mathbb C_{(-1,-1)}.
\]
\end{enumerate}
Moreover, the isotypic subspace $H_*^{(1)}(\SOfiberequalsize)$ (resp.\ $H_*^{(1,1)}(\SOfiber)$) has a basis consisting of all cup diagrams $\ba\in\mathcal R_{\bf 1}$ together with all elements $\bb+\bb^{-}$, where $\bb$ varies over all $\mathcal R_{\bf 2}$. The subspace $H_*^{(-1)}(\SOfiberequalsize)$ (resp.\ $H_*^{(-1,-1)}(\SOfiber)$) has a basis consisting of all elements $\bb-\bb^{-}$, where again $\bb$ varies over all $\mathcal R_{\bf 2}$.
\end{prop}
\begin{proof}
We only argue for $\SOfiberequalsize$ and omit the other case. Note that the collection of vectors
\[
\ba\in\mathcal R_{\bf 1}\,,\hspace{1em} \bb+\bb^{-}\,,\,\,\bb-\bb^{-}\,,\,\,\bb\in\mathcal R_{\bf 2}
\]
is clearly a basis of $H_*(\SOfiberequalsize)$. Moreover, by Theorem~\ref{thmB}, $\ba\in\mathcal R_{\bf 1}$ and $\bb+\bb^{-}$, $\bb\in\mathcal R_{\bf 2}$, are fixed points under the action of the component group and thus contained in $H_*^{(1)}(\SOfiberequalsize)$. Since the generators of the component group act by exchanging $\bb$ and $\bb^-$, it follows that the vectors $\bb-\bb^{-}$, $\bb\in\mathcal R_{\bf 2}$, are contained in $H_*^{(-1)}(\SOfiberequalsize)$.   

Firstly, since the elements of $\mathcal R_{\bf 1}$ all have a ray at the first vertex, the cardinality of $\mathcal R_{\bf 1}$ equals the cardinality of the set of all undecorated cup diagrams on $m-1$ vertices with $l$ cups. By \cite[Proposition~3]{SW12} the set of all such cup diagrams is in bijection with the set of standard Young tableaux of size $m-1$ and shape $(m-1-l,l)$. By the hook-length formula there are exactly $\binom{m-1}{l}-\binom{m-1}{l-1}$ such tableaux which equals the cardinality of $\mathcal R_{\bf 1}$. Secondly, we have an equality 
\[
\binom{m}{l}=\dim H_{2l}(\SOfiberequalsize)= 2|\mathcal R_{\bf 2}|+|\mathcal R_{\bf 1}| = 2|\mathcal R_{\bf 2}|+\binom{m-1}{l}-\binom{m-1}{l-1}
\]
which yields $|\mathcal R_{\bf 2}|=\binom{m-1}{l-1}$. Thus, $\dim H_{2l}^{(-1)}(\SOfiberequalsize)=\binom{m-1}{l-1}$ and $\dim H_{2l}^{(1)}(\SOfiberequalsize)=\binom{m-1}{l}$.
\end{proof}

\begin{thm} \label{thm:explicit_identification_C1}
We have the following isomorphisms of $\mathcal W_{C_{m-1}}$-modules:
\begin{enumerate}
\item \label{item:C_mm_odd} If $m$ is odd, then we have, for all $l\in\{0,\ldots,\lfloor\frac{m}{2}\rfloor\}$, isomorphisms
\begin{equation}\label{eq:explicit_C_isos_Case1}
H_{2l}^{(1)}(\SOfiberequalsize)\cong V_{((m-l-1),(l))} \hspace{1.6em}\text{and}\hspace{1.6em} H_{2l}^{(-1)}(\SOfiberequalsize)\cong V_{((l-1),(m-l))},
\end{equation}
where the first isomorphism in (\ref{eq:explicit_C_isos_Case1}) is given by
\begin{equation} \label{eq:formula_C_iso1}
\ba+\ba^{-}\mapsto 2\sum_{\begin{subarray}{l}U\in\mathcal U_\ba\\ 1\notin\mathcal U_\ba\end{subarray}}(-1)^{\Lambda_{\ba}(U)}e_{T_{U-1}} \hspace{.3em}(\ba\in\mathcal R_2),\hspace{1.6em}\ba\mapsto\sum_{U\in\mathcal U_\ba}(-1)^{\Lambda_{\ba}(U)}e_{T_{U-1}}\hspace{.3em}(\ba\in\mathcal R_1).
\end{equation}
and the second isomorphism in (\ref{eq:explicit_C_isos_Case1}) is given by
\begin{equation} \label{eq:formula_C_iso2}
\ba-\ba^{-}\mapsto 2\sum_{\begin{subarray}{l}U\in\mathcal U_\ba\\ 1\in\mathcal U_\ba\end{subarray}}(-1)^{\Lambda_{\ba}(U)}e_{T_{(M\setminus U)-1}}\hspace{.3em}(\ba\in\mathcal R_2).
\end{equation}
\item If $m\neq k$, then we have, for all $l\in\{0,\ldots,\frac{k-1}{2}\}$, isomorphisms
\[
H_{2l}^{(1,1)}\left(\SOfiber\right)\cong V_{((m-l-1),(l))} \hspace{1.6em}\text{and}\hspace{1.6em} H_{2l}^{(-1,-1)}(\SOfiber)\cong V_{((l-1),(m-l))},
\]
where the left isomorphism is given by the same formulas as in (\ref{eq:formula_C_iso1}) and the right isomorphism is given by (\ref{eq:formula_C_iso2}).
\item \label{item:C_mm_even} If $m=k$ is even, then we have isomorphisms
\begin{equation}
\label{drei}
H_m(\SOfiberequalsize)\cong V_{((\frac{m}{2}-1),(\frac{m}{2}))}\,\,,\,\,\,\ba\mapsto\sum_{\begin{subarray}{l}U\in\mathcal U_\ba\\ 1\notin\mathcal U_\ba\end{subarray}}(-1)^{\Lambda_{\ba}(U)}e_{T_{U-1}}.
\end{equation}
and for all $l\in\{0,\ldots,\frac{m}{2}-1\}$ 
\begin{align*}
H_{2l}(\SOfiberequalsize)\cong H_{2l}(\SOfiber)&\cong H_{2l}^{(1,1)}\left(\SOfiber\right)\oplus H_{2l}^{(-1,-1)}\left(\SOfiber\right)\\
&\cong V_{((m-l-1),(l))}\oplus V_{((l-1),(m-l))}
\end{align*}
which again can be described explicitly via (\ref{eq:formula_C_iso1}) and (\ref{eq:formula_C_iso2}). 
\end{enumerate}
\end{thm}
\begin{proof}
In order to relate the $\mathcal W_{C_{m-1}}$-module $H_*((\mathbb S^2)^{m-1})$ from Proposition~\ref{proposition:Specht_modules_as_homology} (with the bipolytabloid basis) to the $\mathcal W_{C_{m-1}}$-module $H_*((\mathbb S^2)^m)$ from Section~\ref{section:section2} (which contains the Springer representation as a subrepresentation) we define the following $\mathcal W_{C_{m-1}}$-equivariant injections:
\begin{itemize}
\item If $2l<m$ we define $\iota\colon H_{2l}((\mathbb S^2)^{m-1})\to H_{2l}((\mathbb S^2)^m)$ to be the map given by $l_U \mapsto l_{U+1}$, where $U+1$ is obtained from $U$ by adding $1$ to all elements in $U$, i.e.\ $l_{U+1}$ is obtained by adding a dotted line to the left of $l_U$. 
\item If $2l>m$ we define $\kappa\colon H_{2l}((\mathbb S^2)^{m-1})\to H_{2(m-l)}((\mathbb S^2)^m)$ to be the map $l_U\mapsto l_{M\setminus (U+1)}$, i.e.\ $l_{M\setminus (U+1)}$ is obtained from $l_U$ by first adding a dotted line to the left of $l_U$ and then exchanging all dotted lines by undotted lines and vice versa. 
\item If $2l=m$ we define $\chi\colon H_{2l}((\mathbb S^2)^{m-1})\to H_{2l}((\mathbb S^2)^m)$ to be the map $l_U\mapsto l_{U+1}+l_{M\setminus (U+1)}$.
\end{itemize}

In order to prove the first isomorphism in \eqref{eq:explicit_C_isos_Case1} we claim that we have a well-defined composition 
\[
H_{2l}^{(1)}(\mathcal S^{m,m}_\mathrm{KL})\hookrightarrow \im(\iota) \xrightarrow\cong H_{2l}((\mathbb S^2)^{m-1})\xrightarrow\cong V_{((m-l-1),(l))},
\]
i.e.\ we claim that $\im\left(\gamma_m\right)\subseteq\im\left(\iota\right)$. Note that the image of $\iota$ has a basis given by all line diagrams on $m$ vertices with $l$ undotted lines, where the line connected to the first vertex is dotted. If $\ba\in\mathcal R_1$, we obviously have $L_\ba\in \im(\iota)$. If $\ba\in\mathcal R_2$, then we can simplify
\begin{align*}
L_\ba+L_{{\ba}^-}&=\sum_{U\in\mathcal U_\ba}(-1)^{\Lambda_\ba(U)}l_U+\sum_{U\in\mathcal U_{{\ba}^-}}(-1)^{\Lambda_{{\ba}^-}(U)}l_U\\
								 &=\sum_{\begin{subarray}{l}U\in\mathcal U_\ba\\ 1\in\mathcal U_\ba\end{subarray}}(-1)^{\Lambda_\ba(U)}l_U+\sum_{\begin{subarray}{l}U\in\mathcal U_\ba\\ 1\notin\mathcal U_\ba\end{subarray}}(-1)^{\Lambda_\ba(U)}l_U+\sum_{\begin{subarray}{l}U\in\mathcal U_{{\ba}^-}\\ 1\in\mathcal U_{\ba^-}\end{subarray}}(-1)^{\Lambda_{{\ba}^-}(U)}l_U+\sum_{\begin{subarray}{l}U\in\mathcal U_{{\ba}^-}\\ 1\notin\mathcal U_{\ba^-}\end{subarray}}(-1)^{\Lambda_{{\ba}^-}(U)}l_U\\
								 &=2\sum_{\begin{subarray}{l}U\in\mathcal U_\ba\\ 1\notin\mathcal U_\ba\end{subarray}}(-1)^{\Lambda_\ba(U)}l_U,
\end{align*}
because in the second row the first and third summand cancel each other and the second and fourth summand are equal which shows that $L_\ba+L_{{\ba}^-}\in\im(\iota)$. For $\ba\in\mathcal R_1$ the composition is then explicitly given by
\[
\ba\mapsto L_\ba=\sum_{U\in\mathcal U_\ba}(-1)^{\Lambda_\ba(U)}l_U \mapsto \sum_{U\in\mathcal U_\ba}(-1)^{\Lambda_{\ba}(U)}l_{U-1} \mapsto\sum_{U\in\mathcal U_\ba}(-1)^{\Lambda_{\ba}(U)}e_{T_{U-1}}
\]
and for $\ba\in\mathcal R_2$ we have
\[
\ba+\ba^-\mapsto L_\ba+L_{{\ba}^-}=2\sum_{\begin{subarray}{l}U\in\mathcal U_\ba\\ 1\notin\mathcal U_\ba\end{subarray}}(-1)^{\Lambda_\ba(U)}l_U\mapsto 2\sum_{\begin{subarray}{l}U\in\mathcal U_\ba\\ 1\notin\mathcal U_\ba\end{subarray}}(-1)^{\Lambda_\ba(U)}l_{U-1}\mapsto 2\sum_{\begin{subarray}{l}U\in\mathcal U_\ba\\ 1\notin\mathcal U_\ba\end{subarray}}(-1)^{\Lambda_{\ba}(U)}e_{T_{U-1}}.
\]

For the second isomorphism in \eqref{eq:explicit_C_isos_Case1} we consider the composition
\[
H_{2l}^{(-1)}(\SOfiberequalsize)\hookrightarrow \im(\kappa)\xrightarrow\cong H_{2(m-l)}((\mathbb S^2)^{m-1})\xrightarrow\cong V_{((l-1),(m-l))}.
\]
The image of $\kappa$ has a basis given by all line diagrams on $m$ vertices with $m-l$ undotted lines, where the first line is undotted. A similar calculation as for $L_\ba+L_{{\ba}^-}$ shows that we have
\[
L_\ba-L_{\ba^-}=2\sum_{\begin{subarray}{l}U\in\mathcal U_\ba\\ 1\in\mathcal U_\ba\end{subarray}}(-1)^{\Lambda_\ba(U)}l_U
\]
for $\ba\in\mathcal R_1$ which shows that $L_\ba-L_{\ba^-}\in\im(\kappa)$. The composition is then given by
\[
\ba-\ba^-\mapsto L_\ba-L_{\ba^-}=2\sum_{\begin{subarray}{l}U\in\mathcal U_\ba\\ 1\in\mathcal U_\ba\end{subarray}}(-1)^{\Lambda_\ba(U)}l_U\mapsto 2\sum_{\begin{subarray}{l}U\in\mathcal U_\ba\\ 1\in\mathcal U_\ba\end{subarray}}(-1)^{\Lambda_{\ba}(U)}l_{(M\setminus U)-1}\mapsto 2\sum_{\begin{subarray}{l}U\in\mathcal U_\ba\\ 1\in\mathcal U_\ba\end{subarray}}(-1)^{\Lambda_{\ba}(U)}e_{T_{(M\setminus U)-1}}.
\]
Since the $\mathcal W_{C_{m-1}}$-equivariant injections above map to an irreducible target space we see that they are actually isomorphisms.

For the isomorphism in \eqref{drei} we consider the composition
\[
H_m(\SOfiberequalsize)\hookrightarrow\im(\chi)\xrightarrow\cong H_m((\mathbb S^2)^{m-1})\xrightarrow\cong V_{((\frac{m}{2}-1),(\frac{m}{2}))}.
\]
Let $\ba\in H_m(\SOfiberequalsize)$. Then we have
\[
L_\ba = \sum_{U\in\mathcal U_\ba} (-1)^{\Lambda_\ba(U)}l_U = \sum_{U\in\mathcal U'_\ba}(-1)^{\Lambda_\ba(U)}\left(l_U+l_{M\setminus U}\right),
\]   
as in the proof of Theorem~\ref{thm:explicit_identification_D}. From this description we deduce that $L_\ba$ is contained in the image of $\chi$ and the composition is given by 
\[
\ba\mapsto L_\ba=\sum_{U\in\mathcal U'_\ba}(-1)^{\Lambda_\ba(U)}\left(l_U+l_{M\setminus U}\right)\mapsto\sum_{\begin{subarray}{l}U\in\mathcal U_\ba\\ 1\notin\mathcal U_\ba\end{subarray}}(-1)^{\Lambda_{\ba}(U)}l_{U-1} \mapsto\sum_{\begin{subarray}{l}U\in\mathcal U_\ba\\ 1\notin\mathcal U_\ba\end{subarray}}(-1)^{\Lambda_{\ba}(U)}e_{T_{U-1}}.\qedhere
\]
\end{proof}

\subsection{Recovering the Springer correspondence} \label{subsec:Shoji}
Given a bipartition $\tau=(\lambda,\mu)$ such that $|\lambda|+|\mu|=m-1$, we can use the Shoji algorithm \cite{Sho79}, \cite{Sho83} to compute the Jordan type of the nilpotent endomorphism for which the representation $V_{(\lambda,\mu)}$ appears in the top degree of the homology of the corresponding Springer fiber of type $C$. The type $D$ case is similar and omitted due to the trivial component group action. We first recall the algorithm in our special case.

Given a pair of partitions $\tau=(\lambda,\mu)$, define a pair $(d_\tau^{(1)},\varepsilon_\tau^{(1)})$, where $d_\tau^{(1)}=(d_1^{(1)},d_2^{(1)},\ldots)\in\mathbb Z^{\mathbb N}$ and $\varepsilon_\tau^{(1)}=(\varepsilon_1^{(1)},\varepsilon_2^{(1)},\ldots)\in\{\pm 1\}^{\mathbb N}$, as 
\[
d_{2i-1}^{(1)}=2\lambda_i, d_{2i}^{(1)}=2\mu_i, \epsilon_i^{(1)}=1.
\]
Assuming that $(d_\tau^{(j)},\varepsilon_\tau^{(j)})$ is constructed, we define a new pair $(d_\tau^{(j+1)},\varepsilon_\tau^{(j+1)})$ by the rules
\[
\begin{cases}
d_i^{(j+1)}=d_i^{(j)}+1, d_{i+1}^{(j+1)}=d_i^{(j)}+1 &\text{if }d_{i+1}^{(j)}=d_i^{(j)}+2,\\
d_i^{(j+1)}=d_{i+1}^{(j)}-2, d_{i+1}^{(j+1)}=d_i^{(j)}+2 &\text{if }d_{i+1}^{(j)}>d_i^{(j)}+2,\\
d_i^{(j+1)}=d_i^{(j)} &\text{else},
\end{cases}
\]
and
\[
\begin{cases}
\varepsilon_i^{(j+1)}=-\varepsilon_i^{(j)}, \varepsilon_{i+1}^{(j+1)}=-\varepsilon_{i+1}^{(j)} &\text{if }d_{i+1}^{(j)}>d_i^{(j)}+2,\\
\varepsilon_i^{(j+1)}=\varepsilon_i^{(j)} &\text{else}.
\end{cases}
\] 
This is a well-defined algorithm which stabilizes, i.e.\ there exists an integer $q>0$ such that $(d_\tau^{(q)},\varepsilon_\tau^{(q)})=(d_\tau^{(q+1)},\varepsilon_\tau^{(q+1)})=\cdots$. We denote the stable pair by $(d_\tau,\varepsilon_\tau)$. 

Then $d_\tau$ is a partition labeling a nilpotent orbit of type $C$. Moreover, $\varepsilon_\tau$ can be used to define a character of the associated component group by setting $\phi_\tau(\alpha_{d_i})=\varepsilon_i$ for all even parts $d_i$ of the partition $d$. We have $d_i=d_j$ if and only if $\varepsilon_i=\varepsilon_j$ which makes this a well-defined assignment.

\begin{lem}
Let $\tau=((\lambda),(\mu))$ be a pair of one-row partitions, i.e.\ $\lambda,\mu\in\mathbb Z_{\geq 0}$. Then we have
\[
\begin{cases}
d_\tau=(2\lambda+1,2\lambda+1,0,\ldots)\,,\,\,\epsilon_\tau=(1,\ldots) &\text{if }\mu=\lambda+1,\\
d_\tau=(2\mu-2,2\lambda+2,0,\ldots)\,,\,\,\epsilon_\tau=(-1,-1,1,\ldots) &\text{if }\mu>\lambda+1,\\
d_\tau=(2\lambda,2\mu,0,\ldots)\,,\,\,\epsilon_\tau=(1,\ldots) &\text{if }\mu\leq\lambda,
\end{cases}
\]
where the non displayed entries of the sequences are all $0$ or $1$.
\end{lem} 
\begin{proof}
This follows directly from the Shoji algorithm above.
\end{proof}

\begin{prop}
The Springer representation in top degree cohomology of the $(m-1,m-1)$ algebraic Springer fiber of type $C_{m-1}$ is (in Lusztig's normalization) isomorphic to the irreducible labeled by $(\frac{m}{2}-1,\frac{m}{2})$ if $m$ is even. If $m$ is odd, then the top degree representation decomposes as the direct sum of $(\frac{m-1}{2}-1,\frac{m-1}{2}+1)$ and $(\frac{m-1}{2},\frac{m-1}{2})$, where the first irreducible is the isotypic component corresponding to the sign representation of $\mathbb Z/2\mathbb Z$. 

In case of the nilpotent $(n-k-1,k-1)$ with $m\neq k$, we have the irreducibles $(\frac{k-3}{2},\frac{n-k+1}{2})$ and $(\frac{n-k-1}{2},\frac{k-1}{2})$, where the first corresponds to the isotypic component associated with the character $(-1,-1)$ of $\left(\mathbb Z/2\mathbb Z\right)^2$ (and the other one to $(1,1)$). 
\end{prop}

\begin{rem}
By comparing the above proposition with Theorem~\ref{thm:explicit_identification_C1} (and keeping in mind the shift in the labels described in Remark~\ref{rem:wilbert}) we see that we have constructed (up to isomorphism) the original Springer representation.
\end{rem}

\appendix
\section{Dimension of cohomology}\label{appendix}
We construct a cell partition of the topological Springer fiber $\SOfiber$ generalizing a construction in~\cite{ES12}, see also \cite{Kho04}, \cite{Rus11} for similar cell partitions for topological Springer fibers of type $A$. By counting the cells we will obtain an explicit dimension formula for the cohomology in  Proposition~\ref{prop:dim_formula}. The Betti numbers of the two-row Springer fibers were independently computed in~\cite{Kim16}. His method uses restrictions of Springer representations. Our approach uses an explicit cell partition and thus sheds some additional light on the geometry of the Springer fiber.  

\subsection{A cell decomposition compatible with intersections}\label{subsec:cell_decomposition}

Given cup diagrams $\ba,\bb\in\cupdiags$, we write $\ba\rightarrow\bb$ if one of the following conditions is satisfied:
\begin{itemize}
\item The diagrams $\ba$ and $\bb$ are identical except at four not necessarily consecutive vertices $\alpha<\beta<\gamma<\delta\in\{1,\ldots,m\}$, where they differ by one of the following local moves: 

\begin{equation} \label{eq:local_moves1}
\begin{array}{cc}
\begin{tikzpicture}[scale=.8,thick]

\node at (-1,0) {I)};

\draw (0,0) node[above]{$\alpha$} .. controls +(0,-.5) and +(0,-.5) .. +(.5,0) node[above]{$\beta$};
\draw (1,0) node[above]{$\gamma$} .. controls +(0,-.5) and +(0,-.5) .. +(.5,0) node[above]{$\delta$};
\draw[->] (1.75,-.2) -- +(1,0);
\draw (3,0) node[above]{$\alpha$} .. controls +(0,-1) and +(0,-1) .. +(1.5,0) node[above]{$\delta$};
\draw (3.5,0) node[above]{$\beta$} .. controls +(0,-.5) and +(0,-.5) .. +(.5,0) node[above]{$\gamma$};
\end{tikzpicture}
&
\begin{tikzpicture}[scale=.8,thick]

\node at (-1,0) {II)};

\draw (0,0) node[above]{$\alpha$} .. controls +(0,-1) and +(0,-1) .. +(1.5,0) node[above]{$\delta$};
\draw (.5,0) node[above]{$\beta$} .. controls +(0,-.5) and +(0,-.5) .. +(.5,0) node[above]{$\gamma$};
\draw[->] (1.75,-.2) -- +(1,0);
\draw (3,0) node[above]{$\alpha$} .. controls +(0,-.5) and +(0,-.5) .. +(.5,0) node[above]{$\beta$};
\fill ([xshift=-2.5pt,yshift=-2.5pt]3.25,-.36) rectangle ++(5pt,5pt);
\draw (4,0) node[above]{$\gamma$} .. controls +(0,-.5) and +(0,-.5) .. +(.5,0) node[above]{$\delta$};
\fill ([xshift=-2.5pt,yshift=-2.5pt]4.25,-.36) rectangle ++(5pt,5pt);
\end{tikzpicture}
\\
\begin{tikzpicture}[scale=.8,thick]

\node at (-1,0) {III)};

\draw (0,0) node[above]{$\alpha$} .. controls +(0,-.5) and +(0,-.5) .. +(.5,0) node[above]{$\beta$};
\fill ([xshift=-2.5pt,yshift=-2.5pt].25,-.36) rectangle ++(5pt,5pt);
\draw (1,0) node[above]{$\gamma$} .. controls +(0,-.5) and +(0,-.5) .. +(.5,0) node[above]{$\delta$};
\draw[->] (1.75,-.2) -- +(1,0);
\draw (3,0) node[above]{$\alpha$} .. controls +(0,-1) and +(0,-1) .. +(1.5,0) node[above]{$\delta$};
\fill ([xshift=-2.5pt,yshift=-2.5pt]3.75,-.74) rectangle ++(5pt,5pt);
\draw (3.5,0) node[above]{$\beta$} .. controls +(0,-.5) and +(0,-.5) .. +(.5,0) node[above]{$\gamma$};
\end{tikzpicture}
&
\begin{tikzpicture}[scale=.8,thick]
\node at (-1,0) {IV)};

\draw (0,0) node[above]{$\alpha$} .. controls +(0,-1) and +(0,-1) .. +(1.5,0) node[above]{$\delta$};
\draw (.5,0) node[above]{$\beta$} .. controls +(0,-.5) and +(0,-.5) .. +(.5,0) node[above]{$\gamma$};
\fill ([xshift=-2.5pt,yshift=-2.5pt].75,-.74) rectangle ++(5pt,5pt);
\draw[->] (1.75,-.2) -- +(1,0);
\draw (3,0) node[above]{$\alpha$} .. controls +(0,-.5) and +(0,-.5) .. +(.5,0) node[above]{$\beta$};
\draw (4,0) node[above]{$\gamma$} .. controls +(0,-.5) and +(0,-.5) .. +(.5,0) node[above]{$\delta$};
\fill ([xshift=-2.5pt,yshift=-2.5pt]4.25,-.36) rectangle ++(5pt,5pt);
\end{tikzpicture}
\end{array}
\end{equation}
\item The diagrams $\ba$ and $\bb$ are identical except at three not necessarily consecutive vertices $\alpha<\beta<\gamma\in\{1,\ldots,m\}$, where they differ by one of the following local moves: 
\begin{equation} \label{eq:local_moves2}
\begin{array}{cc}
\begin{tikzpicture}[scale=.8,thick]
\node at (-1,0) {I')};

\draw (0,0) node[above]{$\alpha$} .. controls +(0,-.5) and +(0,-.5) .. +(.5,0) node[above]{$\beta$};
\draw (1,0) node[above]{$\gamma$} -- +(0,-.8);
\draw[->] (1.25,-.2) -- +(1,0);
\draw (2.5,0) node[above]{$\alpha$} -- +(0,-.8);
\draw (3,0) node[above]{$\beta$} .. controls +(0,-.5) and +(0,-.5) .. +(.5,0) node[above]{$\gamma$};
\end{tikzpicture}
&
\begin{tikzpicture}[scale=.8,thick]
\node at (-1,0) {II')};

\draw (0,0) node[above]{$\alpha$} -- +(0,-.8);
\draw (.5,0) node[above]{$\beta$} .. controls +(0,-.5) and +(0,-.5) .. +(.5,0) node[above]{$\gamma$};
\draw[->] (1.25,-.2) -- +(1,0);
\draw (2.5,0) node[above]{$\alpha$} .. controls +(0,-.5) and +(0,-.5) .. +(.5,0) node[above]{$\beta$};
\fill ([xshift=-2.5pt,yshift=-2.5pt]2.75,-.36) rectangle ++(5pt,5pt);
\draw (3.5,0) node[above]{$\gamma$} -- +(0,-.8);
\fill ([xshift=-2.5pt,yshift=-2.5pt]3.5,-.4) rectangle ++(5pt,5pt);
\end{tikzpicture}
\\
\begin{tikzpicture}[scale=.8,thick]
\node at (-1,0) {III')};

\draw (0,0) node[above]{$\alpha$} .. controls +(0,-.5) and +(0,-.5) .. +(.5,0) node[above]{$\beta$};
\fill ([xshift=-2.5pt,yshift=-2.5pt].25,-.36) rectangle ++(5pt,5pt);
\draw (1,0) node[above]{$\gamma$} -- +(0,-.8);
\draw[->] (1.25,-.2) -- +(1,0);
\draw (2.5,0) node[above]{$\alpha$} -- +(0,-.8);
\fill ([xshift=-2.5pt,yshift=-2.5pt]2.5,-.4) rectangle ++(5pt,5pt);
\draw (3,0) node[above]{$\beta$} .. controls +(0,-.5) and +(0,-.5) .. +(.5,0) node[above]{$\gamma$};
\end{tikzpicture}
&
\begin{tikzpicture}[scale=.8,thick]
\node at (-1,0) {IV')};

\draw (0,0) node[above]{$\alpha$} -- +(0,-.8);
\fill ([xshift=-2.5pt,yshift=-2.5pt]0,-.4) rectangle ++(5pt,5pt);
\draw (.5,0) node[above]{$\beta$} .. controls +(0,-.5) and +(0,-.5) .. +(.5,0) node[above]{$\gamma$};
\draw[->] (1.25,-.2) -- +(1,0);
\draw (2.5,0) node[above]{$\alpha$} .. controls +(0,-.5) and +(0,-.5) .. +(.5,0) node[above]{$\beta$};
\draw (3.5,0) node[above]{$\gamma$} -- +(0,-.8);
\fill ([xshift=-2.5pt,yshift=-2.5pt]3.5,-.4) rectangle ++(5pt,5pt);
\end{tikzpicture}
\end{array}
\end{equation}
\end{itemize}

We use these loval moves to define a partial order on $\cupdiags$ by setting $\ba\prec\bb$ if there exists a finite chain of arrows $\ba\rightarrow {\bf c}_1 \rightarrow\cdots\rightarrow {\bf c}_r \rightarrow\bb$.

\begin{rem}
The local moves (\ref{eq:local_moves1}) and (\ref{eq:local_moves2}) defined in an ad hoc manner above have a natural geometric interpretation in the context of perverse sheaves (constructible with respect to the Schubert stratification) on isotropic Grassmannians
, see \cite{Bra02} and \cite{ES15}. 
\end{rem}

Let $\ba\in\cupdiags$ and let $i_1<i_2<\cdots<i_{\lfloor \frac{k}{2}\rfloor}$ be the left endpoints of cups in $\ba$. 

\begin{defi}
Given a vertex $r\in\{1,\ldots,m\}$, we say that a cup connecting vertices $i<j$ or a ray connected to vertex $i$ is to the right of $r$ if $r\leq i$. Let $\sigma(r)$ be the number of marked cups and marked rays to the right of $r$.
\end{defi}

We have the following homeomorphism
\begin{eqnarray} \label{local_coordinates}
\begin{array}{cccl}
\overline{\xi}_\ba :&S_\ba&\longrightarrow&(\mathbb S^2)^{\lfloor \frac{k}{2}\rfloor}\\
&(x_1,\ldots,x_m)&\longmapsto&(y_1,\ldots,y_{\lfloor \frac{k}{2}\rfloor})
\end{array}
\end{eqnarray}
where $y_{r}=(-1)^{i_r+\sigma(i_r)}x_{i_r}$ and $1\leq r\leq \lfloor \frac{k}{2}\rfloor$. 

We define the homeomorphism $t\colon\mathbb S^2\to\mathbb S^2$ as the restriction of the linear endomorphism $(x,y,z)\mapsto (z,y,x)$ of $\mathbb R^3$ to $\mathbb S^2\subseteq\mathbb R^3$. Note that $t(p)=q$. Consider the involutive homeomorphism
\begin{equation} \label{eq:correcting_map}
\begin{array}{cccl}
\Phi_\ba\colon&(\mathbb S^2)^{\lfloor\frac{k}{2}\rfloor}&\longrightarrow&(\mathbb S^2)^{\lfloor\frac{k}{2}\rfloor}\\
&(y_1,\ldots,y_{\lfloor\frac{k}{2}\rfloor})&\longmapsto&\left(y_1,\ldots,y_s,t(y_{s+1}),\ldots,t(y_{\lfloor\frac{k}{2}\rfloor})\right)
\end{array}
\end{equation}
where $s$ is the number of cups which are not to the right of the second leftmost ray (if $m=k$, in which case there is no second ray, the map (\ref{eq:correcting_map}) is the identity).

In the following, the composition of the maps in (\ref{local_coordinates}) and (\ref{eq:correcting_map}) will play a crucial role.

\begin{lem}
The map $\Psi_\ba\colon S_\ba\to(\mathbb S^2)^{\lfloor\frac{k}{2}\rfloor}$ defined as the composition $\Phi_\ba\circ\overline{\xi}_\ba$ is a homeomorphism.
\end{lem}

\begin{ex}
The preimage of $(p,p,p,p)\in(\mathbb S^2)^4$ under $\Psi_\ba$, where 
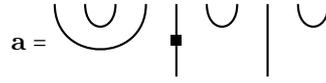
\[
\ba=\begin{tikzpicture}[scale=.8, baseline={(0,-.6)}]
\draw[thick] (.5,0) .. controls +(0,-.5) and +(0,-.5) .. +(.5,0);
\draw[thick] (0,0) .. controls +(0,-1) and +(0,-1) .. +(1.5,0);
\draw[thick] (2.5,0) .. controls +(0,-.5) and +(0,-.5) .. +(.5,0);
\draw[thick] (4,0) .. controls +(0,-.5) and +(0,-.5) .. +(.5,0);

\draw[thick] (2,0) -- +(0,-1.2);
\draw[thick] (3.5,0) -- +(0,-1.2);

\fill ([xshift=-2.5pt,yshift=-2.5pt]2,-.6) rectangle ++(5pt,5pt);
\end{tikzpicture}
\]
is given by $(p,-p,p,-p,p,p,-p,q,-q,q)$.
\end{ex}

In order to define a cell decomposition of $S_\ba$ for $\ba\in\cupdiags$, we proceed as in \cite[\S4.5]{ES12} and associate a graph $\Gamma_\ba=\left(\mathcal V(\Gamma_\ba),\mathcal E(\Gamma_\ba)\right)$ with each cup diagram $\ba\in\cupdiags$ as follows: 
\begin{itemize}
\item The vertices $\mathcal V(\Gamma_\ba)$ of $\Gamma_\ba$ are given by the cups $(i,j)$ in $\ba$. 
\item Two vertices $(i_1,j_1),(i_2,j_2) \in\mathcal V(\Gamma_\ba)$ are connected by an edge $(i_1,j_1)-(i_2,j_2)\in\mathcal E(\Gamma_\ba)$ in $\Gamma_\ba$ if and only if there exists some $\bb\in\cupdiags$ with $\bb\rightarrow\ba$ such that $\ba$ is obtained from $\bb$ by a local move of type \textbf{I)}-\textbf{IV)} at the vertices $i_1,i_2,j_1,j_2$.
\end{itemize}

As in \cite{ES12}, the graph $\Gamma_\ba$ is a forest whose roots $\mathcal R(\Gamma_\ba)$ are precisely the {\it outer cups} of $\ba$, i.e.\  the cups which are not nested in any other cup and do not contain any marked cup to their right. 

We assign to each subset $J \subseteq\mathcal R(\Gamma_\ba) \cup \mathcal E(\Gamma_\ba)$ the subset $C_J'$ of $(\mathbb S^2)^{\lfloor\frac{k}{2}\rfloor}$ given by all elements $(y_1,\ldots,y_{\lfloor\frac{k}{2}\rfloor})\in(\mathbb S^2)^{\lfloor\frac{k}{2}\rfloor}$ which satisfy the following relations:

\begin{enumerate}[label={(C{\arabic*})}]
\item If $(i_q,j_q) \in \mathcal R(\Gamma_\ba) \cap J$ then $y_q = p$,
\item if $(i_q,j_q) \in \mathcal R(\Gamma_\ba)$ but $(i_q,j_q) \not\in J$ then $y_q \neq p$,
\item if $(i_q,j_q) - (i_{q'},j_{q'}) \in \mathcal E(\Gamma_\ba) \cap J$ then $y_q = y_{q'}$,
\item if $(i_q,j_q) - (i_{q'},j_{q'}) \not\in \mathcal E(\Gamma_\ba) \cap J$ then $y_q \neq y_{q'}$.
\end{enumerate}

\begin{lem} \label{decomp_disjoint}
There is a decomposition
\begin{equation}\label{eq:celldec_spheres}
(\mathbb S^2)^{\lfloor\frac{k}{2}\rfloor} = \bigsqcup_{J \subseteq \mathcal R(\Gamma_\ba)\cup \mathcal E(\Gamma_\ba)} C_J'
\end{equation}
 into disjoint cells $C_J'$ homeomorphic to $\mathbb R^{2(\lfloor\frac{k}{2}\rfloor-|J|)}$. Moreover, pushing forward along \eqref{local_coordinates} gives a cell decomposition
\begin{equation}\label{eq:celldec_irrcomponent}
S_\ba = \bigsqcup_{J \subseteq \mathcal R(\Gamma_\ba)\cup \mathcal E(\Gamma_\ba)} C_J,
\end{equation}
where $C_J = (\Psi_\ba)^{-1}(C_J')$.
\end{lem}
\begin{proof}
The proof of (\ref{eq:celldec_spheres}) is the same as in the equal-row case \cite[Lemma~4.17]{ES12} because the additional rays do not play any role in the construction of the $C_J'$, and (\ref{eq:celldec_spheres}) implies (\ref{eq:celldec_irrcomponent}).
\end{proof}


The reason for choosing the homeomorphism $\Psi_\ba$ in the construction of the cell decomposition (\ref{eq:celldec_irrcomponent}) of $S_\ba$ is that the resulting decompositions are compatible with pairwise intersections in the sense of the next lemma which extends \cite[Proposition~4.24]{ES12} to the general two-block case.

\begin{lem} \label{intersection_and_cells}
Let $\ba,\bb \in \cupdiags$ such that $\bb \rightarrow\ba$.
\begin{enumerate}[label={({\arabic*})}]
\item If $\bb\rightarrow\ba$ is of type \textbf{I)}-\textbf{IV)}, then
\[
S_\ba \cap S_\bb = \bigcup_{J \subseteq \mathcal R(\Gamma_\ba) \cup \mathcal E(\Gamma_\ba), e\in J} C_J,
\]
where $e \in\mathcal E(\Gamma_\ba)$ is the edge of $\Gamma_\ba$ determined by the move $\bb \rightarrow\ba$.
\item If $\bb\rightarrow\ba$ is of type \textbf{I')}-\textbf{IV')}, then there is a unique cup $\alpha \in {\rm cups}(\ba)$ such that $\alpha \notin
{\rm cups}(\bb)$. Moreover, $\alpha \in \mathcal R(\Gamma_\ba)$ and
$$S_\ba \cap S_\bb = \bigcup_{J \subseteq \mathcal R(\Gamma_\ba) \cup \mathcal E(\Gamma_\ba),
\alpha \in J} C_J.$$
\end{enumerate}
\end{lem}
\begin{proof}
This is a straightforward case-by-case analysis. 
\end{proof}

\subsection{A cell partition of the topological Springer fiber}

Given a cup diagram $\ba\in\cupdiags$, define its {\it extension} $\widetilde{\ba}\in\mathbb B^{2m-k,2m-k}_\mathrm{KL}$ as the unique cup diagram obtained by connecting the loose endpoints of the $m-k$ rightmost rays to $m-k$ newly added vertices to the right of $\ba$ via an unmarked cup such that the resulting diagram is crossingless (see \cite[Definition 3.11]{Rus11} and \cite[\S5.3]{ES15} for a similar but different extension). Let $\mathbb B^{2m-k,2m-k}_{2m-k,k}$ denote the cup diagrams in $\mathbb B^{2m-k,2m-k}_\mathrm{KL}$ which are obtained as an extension of some diagram in $\mathbb B^{2m-k,k}_\mathrm{KL}$, i.e.\ the set of all cup diagrams in $\mathbb B^{2m-k,2m-k}_\mathrm{KL}$ whose $m-k$ rightmost vertices are right endpoints of unmarked cups. If $m=k$, then the extension procedure returns the same diagram. Note that the leftmost ray is never replaced by a cup.

\begin{ex}
Here is an example showing the extension of a cup diagram:
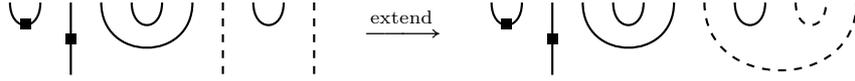
\[
\begin{tikzpicture}[scale=.8, baseline={(0,-.5)}]
\draw[thick] (0,0) .. controls +(0,-.5) and +(0,-.5) .. +(.5,0);
\draw[thick] (1.5,0) .. controls +(0,-1) and +(0,-1) .. +(1.5,0);
\draw[thick] (2,0) .. controls +(0,-.5) and +(0,-.5) .. +(.5,0);
\draw[thick] (4,0) .. controls +(0,-.5) and +(0,-.5) .. +(.5,0);

\draw[thick] (1,0) -- +(0,-1.2);
\draw[thick,dashed] (3.5,0) -- +(0,-1.2);
\draw[thick,dashed] (5,0) -- +(0,-1.2);

\fill ([xshift=-2.5pt,yshift=-2.5pt].25,-.36) rectangle ++(5pt,5pt);
\fill ([xshift=-2.5pt,yshift=-2.5pt]1,-.6) rectangle ++(5pt,5pt);
\end{tikzpicture}
\hspace{1.6em}
\xrightarrow{\text{extend}}
\hspace{1.6em}
\begin{tikzpicture}[scale=.8, baseline={(0,-.5)}]
\draw[thick] (0,0) .. controls +(0,-.5) and +(0,-.5) .. +(.5,0);
\draw[thick] (1.5,0) .. controls +(0,-1) and +(0,-1) .. +(1.5,0);
\draw[thick] (2,0) .. controls +(0,-.5) and +(0,-.5) .. +(.5,0);
\draw[thick,dashed] (3.5,0) .. controls +(0,-1.5) and +(0,-1.5) .. +(2.5,0);
\draw[thick] (4,0) .. controls +(0,-.5) and +(0,-.5) .. +(.5,0);
\draw[thick,dashed] (5,0) .. controls +(0,-.5) and +(0,-.5) .. +(.5,0);

\draw[thick] (1,0) -- +(0,-1.2);

\fill ([xshift=-2.5pt,yshift=-2.5pt].25,-.36) rectangle ++(5pt,5pt);
\fill ([xshift=-2.5pt,yshift=-2.5pt]1,-.6) rectangle ++(5pt,5pt);
\end{tikzpicture}
\]
For illustrative purposes the components of the diagrams which change during extension are drawn in dashed font. 
\end{ex}

Inspired by the combinatorial completion procedure we define an embedding (see also \cite[Section 5]{Rus11} for a similar map in type $A$)of topological Springer fibers 
\[
\eta_{2m-k,k}\colon\SOfiber\hookrightarrow\mathcal S^{2m-k,2m-k}_\mathrm{KL}\,,\,\,(x_1,\ldots,x_m)\mapsto (x_1,\ldots,x_m,-q,\ldots,-q).
\]
Note that $\eta_{2m-k,k}(S_\ba)\subseteq S_{\widetilde{\ba}}$.

Given $\ba,\bb\in\mathbb B^{2m-k,k}_\mathrm{KL}$, one easily verifies that $\bb\rightarrow\ba$ if and only if $\widetilde{\bb}\rightarrow\widetilde{\ba}$. Thus, we have an isomorphism of posets $\mathbb B^{2m-k,k}_\mathrm{KL}\cong \widetilde{\mathbb B}^{2m-k,2m-k}_\mathrm{KL}$. We equip $\widetilde{\mathbb B}^{2m-k,2m-k}_\mathrm{KL}$ with the induced total order from some fixed total order (which extends the partial order $\prec$ from before) on $\mathbb B^{2m-k,2m-k}_\mathrm{KL}$. We pull this total order over to $\mathbb B^{2m-k,k}_\mathrm{KL}$ via the isomorphism of posets. 

Let $S_{<\ba} = \bigcup_{\bb<\ba} S_\bb$.

\begin{lem} \label{reduction_for_intersection}
For any $\ba\in\mathbb B^{2m-k,k}_\mathrm{KL}$ we have
\begin{eqnarray}
\label{eqintersectionscells}
S_{<\ba} \cap S_\ba &=& \bigcup_{\bb\rightarrow\ba} (S_\bb \cap S_\ba).
\end{eqnarray}
\end{lem}
\begin{proof}
In case $m=k$ the claimed equality was proven in \cite[Lemma~4.23]{ES12}. 

Since $\rightarrow$ implies $<$, the left side is contained in the right side.

So, assume there exists $\bb<\ba\in\mathbb B^{2m-k,k}_\mathrm{KL}$ and $x\in S_\ba\cap S_\bb$ such that $x\notin\bigcup_{\bb\rightarrow\ba} S_\bb\cap S_\ba$. Then $\eta_{2m-k,k}(x)\in S_{\widetilde{\bb}}\cap S_{\widetilde{\ba}}$ (note that $\widetilde{\bb}<\widetilde{\ba}$ since $\bb<\ba$ and by definition of the total order $<$ on $\mathbb B^{2m-k,k}_\mathrm{KL}$). Hence,
\[
\bigcup_{\overset{\bb<\widetilde{\ba}}{\bb\in\mathbb B^{2m-k,2m-k}_\mathrm{KL}}} S_\bb\cap S_{\widetilde{\ba}} = \bigcup_{\overset{\bb\rightarrow\widetilde{\ba}}{\bb\in\mathbb B^{2m-k,2m-k}_\mathrm{KL}}} S_\bb\cap S_{\widetilde{\ba}} = \bigcup_{\overset{\bb\rightarrow\widetilde{\ba}}{\bb\in\widetilde{\mathbb B}^{2m-k,2m-k}_\mathrm{KL}}} S_\bb\cap S_{\widetilde{\ba}}
\]  
and therefore
\begin{align*}
\eta_{2m-k,k}(\SOfiber)\cap\bigcup_{\overset{\bb<\widetilde{\ba}}{\bb\in\mathbb B^{2m-k,2m-k}_\mathrm{KL}}} S_\bb\cap S_{\widetilde{\ba}} &= \eta_{2m-k,k}(\mathcal S^{2m-k,k}_{\mathrm{KL}})\cap\bigcup_{\overset{\bb\rightarrow\widetilde{\ba}}{\bb\in\widetilde{\mathbb B}^{2m-k,2m-k}_\mathrm{KL}}} S_\bb\cap S_{\widetilde{\ba}}\\
&= \eta_{2m-k,k}\left(\bigcup_{\bb\rightarrow \ba} S_\bb\cap S_\ba\right).
\end{align*}
Since $\eta_{2m-k,k}(x)$ is contained in the leftmost set above, it is also contained in the rightmost set. By the injectivity of $\eta_{2m-k,k}$, we deduce that $x\in\bigcup_{b\rightarrow a} S_\bb\cap S_\ba$, a contradiction.
\end{proof}

Fix a cup diagram $\ba\in\mathbb B^{2m-k,k}_\mathrm{KL}$. A cup in $\ba$ is called {\it special} if there exists some $\bb\in\mathbb B^{2m-k,k}_\mathrm{KL}$ such that $\bb\rightarrow\ba$, where $\ba$ and $\bb$ are related by a local move of type \textbf{I')}-\textbf{IV')}, and this cup is the unique cup which changes under the move $\ba\rightarrow\bb$. Let $\mathcal R(\Gamma_\ba)^{\rm sp}$ be the set of special cups in $\ba$. In particular, we have $\mathcal R(\Gamma_\ba)^{\rm sp}=\emptyset$ if $k=m$ is even, i.e.\ if there are no rays in $\ba$.

\begin{rem}
As in \cite[Remark 4.25]{ES12}, the moves \textbf{I')}-\textbf{IV')} imply that all outer cups (in the sense as defined above) of a given cup diagram $\ba\in\mathbb B^{2m-k,k}_\mathrm{KL}$ are special if the leftmost ray in $\ba$ is marked, whereas only outer cups to the right of the leftmost ray are special if the ray is unmarked.
\end{rem}

\begin{cor} \label{decreasing_intersection_and_cells}
Given $\ba \in \mathbb B^{2m-k,k}_\mathrm{KL}$, then
\begin{eqnarray*}
S_{<\ba} \cap S_\ba&=& \underset{\begin{array}{c} J \subseteq \mathcal R(\Gamma_\ba) \cup \mathcal E(\Gamma_\ba)  \\ J \cap (\mathcal E(\Gamma_\ba) \cup \mathcal R(\Gamma_\ba)^{\rm sp})\neq \emptyset \end{array}}{\bigcup} C_J.
\end{eqnarray*}
\end{cor}

\begin{proof}
This follows directly from the previous two lemmas.
\end{proof}

In order to argue inductively in the proof of the following proposition, we introduce the notion of a {\it partial/one-step extension}. This is defined as the ordinary extension introduced above, except that we only replace the rightmost ray with a cup.   

\begin{prop} \label{prop:dim_formula}
If $k\neq m$, then we have the following dimension formula:
\[
\dim H^*(\SOfiber) = \sum_{i=0}^{\frac{k-1}{2}}\binom{m}{i}.
\]
\end{prop}

\begin{rem} \label{rem:dim_homology_equal-row}
In the special case $m=k$ we have $\dim H^*(\SOfiber)=\sum_{i=0}^{\frac{m-1}{2}}\binom{m}{i}=2^{m-1}$ as proven in \cite[Prop.\ 4.28]{ES12}.
\end{rem}

\begin{proof}
We have a disjoint union $\SOfiber = \bigcup S_{<\ba} \cap S_\ba$. We thus obtain a cell partition of $\SOfiber$ by using the cell partition from Corollary~\ref{decreasing_intersection_and_cells} for each of the $S_{<\ba} \cap S_\ba$. Hence, the dimension of $H^*(\SOfiber)$ can be calculated by counting the cells contained in $S_\ba\setminus S_{<\ba}$ for all $\ba\in\mathbb B^{2m-k,k}_\mathrm{KL}$ and taking their sum. By Corollary~\ref{decreasing_intersection_and_cells} the cells contained in $S_\ba\setminus S_{<\ba}$ correspond bijectively to the outer cups in $\ba$ which are not special. This yields
\begin{equation} \label{eq:dim_cohomology}
\dim H^*(\SOfiber)=\sum_{\ba\in\mathbb B^{2m-k,k}_\mathrm{KL}} 2^{\vert\mathcal R(\Gamma_\ba)\setminus\mathcal R(\Gamma_\ba)^{\rm sp}\vert}.
\end{equation} 

It remains to show that the right hand side of (\ref{eq:dim_cohomology}) equals $\sum_{i=0}^{\frac{k-1}{2}}\binom{m}{i}$. We prove this claim by induction on $m-k$, i.e.\ by induction on the number of rays in the cup diagrams contained in $\mathbb B^{2m-k,k}_\mathrm{KL}$. Note that the base $m-k=0$ of the induction (i.e.\ no ray or one ray; depending on whether $m$ is even or odd) is true by Remark~\ref{rem:dim_homology_equal-row}. 

Thus, we consider a partition $(2m-k,k)$, $m-k>1$. Firstly, note that we have 
\begin{equation} \label{eq:comparing_cells}
2^{\vert\mathcal R(\Gamma_\ba)\setminus\mathcal R(\Gamma_\ba)^{\rm sp}\vert}=2^{\vert\mathcal R(\Gamma_{\tilde{\ba}})\setminus\mathcal R(\Gamma_{\tilde{\ba}})^{\rm sp}\vert}
\end{equation}
for all $\ba\in\mathbb B^{2m-k,k}_\mathrm{KL}$, where $\tilde{\ba}\in\mathbb B^{2m-k,k+2}_\mathrm{KL}$ denotes the one-step extension of $\ba$ (note that the completed cup is always special). Secondly, the elements $\ba\in\mathbb B^{2m-k,k+2}_\mathrm{KL}\setminus\widetilde{\mathbb B}^{2m-k,k+2}_\mathrm{KL}$ are precisely the cup diagrams with an unmarked ray connected to the rightmost vertex $m+1$. Deleting this ray induces a bijection between $\mathbb B^{2m-k,k+2}_\mathrm{KL}\setminus\widetilde{\mathbb B}^{2m-k,k+2}_\mathrm{KL}$ and $\mathbb B^{2(m-1)-k,k+2}_\mathrm{KL}$ and we have 
\begin{equation} \label{eq:comparing_cells2}
2^{\vert\mathcal R(\Gamma_\ba)\setminus\mathcal R(\Gamma_\ba)^{\rm sp}\vert}=2^{\vert\mathcal R(\Gamma_{\bb})\setminus\mathcal R(\Gamma_{\bb})^{\rm sp}\vert},
\end{equation}
where $\bb$ is the diagram obtained from $\ba$ by deleting the rightmost ray. By using (\ref{eq:comparing_cells}) in the first and (\ref{eq:comparing_cells2}) in the third equality we compute 
\begin{align*}
\sum_{a\in\mathbb B^{2m-k,k}_\mathrm{KL}} 2^{\vert\mathcal R(\Gamma_\ba)\setminus\mathcal R(\Gamma_\ba)^{\rm sp}\vert} &= \sum_{\tilde{\ba}\in\widetilde{\mathbb B}^{2m-k,k+2}_\mathrm{KL}} 2^{\vert\mathcal R(\Gamma_{\tilde{\ba}})\setminus\mathcal R(\Gamma_{\tilde{\ba}})^{\rm sp}\vert} \\
		&= \sum_{\ba\in\mathbb B^{2m-k,k+2}_\mathrm{KL}} 2^{\vert\mathcal R(\Gamma_\ba)\setminus\mathcal R(\Gamma_\ba)^{\rm sp}\vert} - \sum_{\ba\in\mathbb B^{2m-k,k+2}_\mathrm{KL}\setminus\widetilde{\mathbb B}^{2m-k,k+2}_\mathrm{KL}} 2^{\vert\mathcal R(\Gamma_\ba)\setminus\mathcal R(\Gamma_\ba)^{\rm sp}\vert} \\
		&= \sum_{\ba\in\mathbb B^{2m-k,k+2}_\mathrm{KL}} 2^{\vert\mathcal R(\Gamma_\ba)\setminus\mathcal R(\Gamma_\ba)^{\rm sp}\vert} - \sum_{\ba\in\mathbb B^{2(m-1)-k,k+2}_\mathrm{KL}} 2^{\vert\mathcal R(\Gamma_\ba)\setminus\mathcal R(\Gamma_\ba)^{\rm sp}\vert}.
\end{align*}

By (\ref{eq:dim_cohomology}) the two terms in the above difference equal the respective dimension of cohomology which we know by induction. Thus, the above chain of equalities equals
\[
\sum_{i=0}^{\frac{k+1}{2}}\binom{m+1}{i} - \sum_{i=0}^{\frac{k+1}{2}}\binom{m}{i}=\sum_{i=1}^{\frac{k+1}{2}}\left(\binom{m+1}{i}-\binom{m}{i}\right) = \sum_{i=1}^{\frac{k+1}{2}}\binom{m}{i-1} = \sum_{i=0}^{\frac{k-1}{2}}\binom{m}{i},
\]
which proves the claim.
\end{proof}



\bibliographystyle{amsalpha}
\bibliography{litlist}

\end{document}